\documentclass[a4paper,10pt]{article}
\usepackage[reqno]{amsmath}
\makeatletter
\newcommand{\leqnomode}{\tagsleft@true}
\newcommand{\reqnomode}{\tagsleft@false}
\makeatother

\usepackage[english,british]{babel}
\usepackage{amsmath,amsthm,amssymb,dsfont,CJK,esint}
\usepackage{adjustbox,lipsum}
\usepackage{tikz}
\usepackage{standalone}
\usepackage{mathtools}
\usepackage{fancyhdr}
\usepackage{graphicx}
\usepackage{bibentry}
\usepackage{enumitem}
\usepackage{xcolor}
\usepackage{stackengine}
\usepackage[
            pdfstartview=FitH,
            CJKbookmarks=true,
            bookmarksnumbered=true,
            bookmarksopen=true,
            colorlinks,
            pdfborder=001,
            linkcolor=red,
            citecolor=green,
            ]{hyperref}
\usepackage[a4paper,width=160mm,top=30mm,bottom=30mm]{geometry}
\hypersetup{
    colorlinks,
    citecolor=green,
    filecolor=black,
    linkcolor=red,
    urlcolor=black
}

\newtheorem{theorem}{Theorem}[section]
\newtheorem{lemma}[theorem]{Lemma}
\newtheorem{definition}[theorem]{Definiton}
\newtheorem{proposition}[theorem]{Proposition}

\newtheorem{conjecture}[theorem]{Conjecture}

\theoremstyle{definition}
\newtheorem{remx}[theorem]{Remark}
\newenvironment{remark}
  {\pushQED{\qed}\remx}
  {\popQED\endremx}

\makeatletter
\newsavebox\myboxA
\newsavebox\myboxB
\newlength\mylenA

\newcommand*\yoverline[2][0.75]{%
    \sbox{\myboxA}{$\m@th#2$}%
    \setbox\myboxB\null
    \ht\myboxB=\ht\myboxA%
    \dp\myboxB=\dp\myboxA%
    \wd\myboxB=#1\wd\myboxA
    \sbox\myboxB{$\m@th\overline{\copy\myboxB}$}
    \setlength\mylenA{\the\wd\myboxA}
    \addtolength\mylenA{-\the\wd\myboxB}%
    \ifdim\wd\myboxB<\wd\myboxA%
       \rlap{\hskip 0.5\mylenA\usebox\myboxB}{\usebox\myboxA}%
    \else
        \hskip -0.5\mylenA\rlap{\usebox\myboxA}{\hskip 0.5\mylenA\usebox\myboxB}%
    \fi}
\makeatother

\numberwithin{equation}{section}

\begin{document}




\newcommand{\diver}{\operatorname{div}}
\newcommand{\lin}{\operatorname{Lin}}
\newcommand{\curl}{\operatorname{curl}}
\newcommand{\ran}{\operatorname{Ran}}
\newcommand{\kernel}{\operatorname{Ker}}
\newcommand{\la}{\langle}
\newcommand{\ra}{\rangle}
\newcommand{\N}{\mathbb{N}}
\newcommand{\R}{\mathbb{R}}
\newcommand{\C}{\mathbb{C}}

\newcommand{\ld}{\lambda}
\newcommand{\fai}{\varphi}
\newcommand{\0}{0}
\newcommand{\n}{\mathbf{n}}
\newcommand{\uu}{{\boldsymbol{\mathrm{u}}}}
\newcommand{\UU}{{\boldsymbol{\mathrm{U}}}}
\newcommand{\buu}{\bar{{\boldsymbol{\mathrm{u}}}}}
\newcommand{\ten}{\\[4pt]}
\newcommand{\six}{\\[4pt]}
\newcommand{\nb}{\nonumber}
\newcommand{\hgamma}{H_{\Gamma}^1(\OO)}
\newcommand{\opert}{O_{\varepsilon,h}}
\newcommand{\barx}{\bar{x}}
\newcommand{\barf}{\bar{f}}
\newcommand{\hatf}{\hat{f}}
\newcommand{\xoneeps}{x_1^{\varepsilon}}
\newcommand{\xh}{x_h}
\newcommand{\scaled}{\nabla_{1,h}}
\newcommand{\scaledb}{\widehat{\nabla}_{1,\gamma}}
\newcommand{\vare}{\varepsilon}
\newcommand{\A}{{\bf{A}}}
\newcommand{\RR}{{\bf{R}}}
\newcommand{\B}{{\bf{B}}}
\newcommand{\CC}{{\bf{C}}}
\newcommand{\D}{{\bf{D}}}
\newcommand{\K}{{\bf{K}}}
\newcommand{\oo}{{\bf{o}}}
\newcommand{\id}{{\bf{Id}}}
\newcommand{\E}{\mathcal{E}}
\newcommand{\ii}{\mathcal{I}}
\newcommand{\sym}{\mathrm{sym}}
\newcommand{\lt}{\left}
\newcommand{\rt}{\right}
\newcommand{\ro}{{\bf{r}}}
\newcommand{\so}{{\bf{s}}}
\newcommand{\e}{{\bf{e}}}
\newcommand{\ww}{{\boldsymbol{\mathrm{w}}}}
\newcommand{\vv}{{\boldsymbol{\mathrm{v}}}}
\newcommand{\zz}{{\boldsymbol{\mathrm{z}}}}
\newcommand{\U}{{\boldsymbol{\mathrm{U}}}}
\newcommand{\G}{{\boldsymbol{\mathrm{G}}}}
\newcommand{\VV}{{\boldsymbol{\mathrm{V}}}}
\newcommand{\T}{{\boldsymbol{\mathrm{U}}}}
\newcommand{\II}{{\boldsymbol{\mathrm{I}}}}
\newcommand{\ZZ}{{\boldsymbol{\mathrm{Z}}}}
\newcommand{\hKK}{{{\bf{K}}}}
\newcommand{\f}{{\bf{f}}}
\newcommand{\g}{{\bf{g}}}
\newcommand{\lkk}{{\bf{k}}}
\newcommand{\tkk}{{\tilde{\bf{k}}}}
\newcommand{\W}{{\boldsymbol{\mathrm{W}}}}
\newcommand{\Y}{{\boldsymbol{\mathrm{Y}}}}
\newcommand{\EE}{{\boldsymbol{\mathrm{E}}}}
\newcommand{\F}{{\bf{F}}}
\newcommand{\spacev}{\mathcal{V}}
\newcommand{\spacevg}{\mathcal{V}^{\gamma}(\Omega\times S)}
\newcommand{\spacevb}{\bar{\mathcal{V}}^{\gamma}(\Omega\times S)}
\newcommand{\spaces}{\mathcal{S}}
\newcommand{\spacesg}{\mathcal{S}^{\gamma}(\Omega\times S)}
\newcommand{\spacesb}{\bar{\mathcal{S}}^{\gamma}(\Omega\times S)}
\newcommand{\skews}{H^1_{\barx,\mathrm{skew}}}
\newcommand{\kk}{\mathcal{K}}
\newcommand{\OO}{O}
\newcommand{\bhe}{{\bf{B}}_{\vare,h}}
\newcommand{\pp}{{\mathbb{P}}}
\newcommand{\ff}{{\mathcal{F}}}
\newcommand{\mWk}{{\mathcal{W}}^{k,2}(\Omega)}
\newcommand{\mWa}{{\mathcal{W}}^{1,2}(\Omega)}
\newcommand{\mWb}{{\mathcal{W}}^{2,2}(\Omega)}
\newcommand{\twos}{\xrightharpoonup{2}}
\newcommand{\twoss}{\xrightarrow{2}}
\newcommand{\bw}{\bar{w}}
\newcommand{\br}{\bar{{\bf{r}}}}
\newcommand{\bz}{\bar{{\bf{z}}}}
\newcommand{\tw}{{W}}
\newcommand{\tr}{{{\bf{R}}}}
\newcommand{\tz}{{{\bf{Z}}}}
\newcommand{\lo}{{{\bf{o}}}}
\newcommand{\hoo}{H^1_{00}(0,L)}
\newcommand{\ho}{H^1_{0}(0,L)}
\newcommand{\hotwo}{H^1_{0}(0,L;\R^2)}
\newcommand{\hooo}{H^1_{00}(0,L;\R^2)}
\newcommand{\hhooo}{H^1_{00}(0,1;\R^2)}
\newcommand{\dsp}{d_{S}^{\bot}(\barx)}
\newcommand{\LB}{{\bf{\Lambda}}}
\newcommand{\LL}{\mathbb{L}}
\newcommand{\mL}{\mathcal{L}}
\newcommand{\mhL}{\widehat{\mathcal{L}}}
\newcommand{\loc}{\mathrm{loc}}
\newcommand{\tqq}{\mathcal{Q}^{*}}
\newcommand{\tii}{\mathcal{I}^{*}}
\newcommand{\Mts}{\mathbb{M}}
\newcommand{\pot}{\mathrm{pot}}
\newcommand{\tU}{{\widehat{\bf{U}}}}
\newcommand{\tVV}{{\widehat{\bf{V}}}}
\newcommand{\pt}{\partial}
\newcommand{\bg}{\Big}
\newcommand{\hA}{\widehat{{\bf{A}}}}
\newcommand{\hB}{\widehat{{\bf{B}}}}
\newcommand{\hCC}{\widehat{{\bf{C}}}}
\newcommand{\hD}{\widehat{{\bf{D}}}}
\newcommand{\fder}{\partial^{\mathrm{MD}}}
\newcommand{\Var}{\mathrm{Var}}
\newcommand{\pta}{\partial^{0\bot}}
\newcommand{\ptaj}{(\partial^{0\bot})^*}
\newcommand{\ptb}{\partial^{1\bot}}
\newcommand{\ptbj}{(\partial^{1\bot})^*}
\newcommand{\geg}{\Lambda_\vare}
\newcommand{\tpta}{\tilde{\partial}^{0\bot}}
\newcommand{\tptb}{\tilde{\partial}^{1\bot}}
\newcommand{\ua}{u_\alpha}
\newcommand{\pa}{p\alpha}
\newcommand{\qa}{q(1-\alpha)}
\newcommand{\Qa}{Q_\alpha}
\newcommand{\Qb}{Q_\eta}
\newcommand{\ga}{\gamma_\alpha}
\newcommand{\gb}{\gamma_\eta}
\newcommand{\ta}{\theta_\alpha}
\newcommand{\tb}{\theta_\eta}


\newcommand{\mH}{\mathcal{H}}
\newcommand{\mD}{\mathcal{D}}
\newcommand{\csob}{\mathcal{S}}
\newcommand{\mA}{\mathcal{A}}
\newcommand{\mK}{\mathcal{K}}
\newcommand{\mS}{\mathcal{S}}
\newcommand{\mI}{\mathcal{I}}
\newcommand{\tas}{{2_*}}
\newcommand{\tbs}{{2^*}}
\newcommand{\tm}{{\tilde{m}}}
\newcommand{\tdu}{{\phi}}
\newcommand{\tpsi}{{\tilde{\psi}}}
\newcommand{\Z}{{\mathbb{Z}}}
\newcommand{\tsigma}{{\tilde{\sigma}}}
\newcommand{\tg}{{\tilde{g}}}
\newcommand{\tG}{{\tilde{G}}}
\newcommand{\mM}{\mathcal{M}}
\newcommand{\mC}{\mathcal{C}}
\newcommand{\wlim}{{\text{w-lim}}\,}

\selectlanguage{english}
\title{On sharp scattering threshold for the mass-energy double critical NLS via double track profile decomposition}
\author{Yongming Luo \thanks{Institut f\"{u}r Wissenschaftliches Rechnen, Technische Universit\"at Dresden, Germany} \thanks{\href{mailto:yongming.luo@tu-dresden.de}{Email: yongming.luo@tu-dresden.de}}
}

\date{}
\maketitle

\begin{abstract}
The present paper is concerned with the large data scattering problem for the mass-energy double critical NLS
\begin{align}
i\pt_t u+\Delta u\pm |u|^{\frac{4}{d}}u\pm |u|^{\frac{4}{d-2}}u=0\tag{DCNLS}
\end{align}
in $H^1(\R^d)$ with $d\geq 3$. In the defocusing-defocusing regime, Tao, Visan and Zhang show that the unique solution of DCNLS is global and scattering in time for arbitrary initial data in $H^1(\R^d)$. This does not hold when at least one of the nonlinearities is focusing, due to the possible formation of blow-up and soliton solutions. However, precise thresholds for a solution of DCNLS being scattering were open in all the remaining regimes. Following the classical concentration compactness principle, we impose sharp scattering thresholds in terms of ground states for DCNLS in all the remaining regimes. The new challenge arises from the fact that the remainders of the standard $L^2$- or $\dot{H}^1$-profile decomposition fail to have asymptotically vanishing diagonal $L^2$- and $\dot{H}^1$-Strichartz norms simultaneously. To overcome this difficulty, we construct a double track profile decomposition which is capable to capture the low, medium and high frequency bubbles within a single profile decomposition and possesses remainders that are asymptotically small in both of the diagonal $L^2$- and $\dot{H}^1$-Strichartz spaces.
\end{abstract}


\section{Introduction and main results}
In this paper, we study the large data scattering problem for the mass-energy double critical NLS
\begin{equation}\label{NLS double crit}
i\pt_t u+\Delta u+\mu_1|u|^{\tas-2}u+\mu_2|u|^{\tbs-2}u=0\quad\text{in $\R\times \R^d$}\tag{DCNLS}
\end{equation}
with $d\geq 3$, $\mu_1,\mu_2\in\{\pm1\}$, $\tas=2+\frac{4}{d}$ and $\tbs=2+\frac{4}{d-2}$. The equation \eqref{NLS double crit} is a special case of the NLS
with combined nonlinearities
\begin{equation}\label{NLS0}
i\pt_t u+\Delta u+\mu_1|u|^{p_1-2}u+\mu_2|u|^{p_2-2}u=0\quad\text{in $\R\times \R^d$}
\end{equation}
with $\mu_1,\mu_2\in\R$ and $p_1,p_2\in(2,\infty)$. \eqref{NLS0} is a prototype model arising from numerous physical applications such as nonlinear optics and Bose-Einstein condensation. The signs $\mu_i$ can be tuned to be defocusing ($\mu_i<0$) or focusing ($\mu_i>0$), indicating the repulsivity or attractivity of the nonlinearity. For a comprehensive introduction on the physical background of \eqref{NLS0}, we refer to \cite{phy_double_crit_1,Buslaev2001,LeMesurier1988} and the references therein. Formally, \eqref{NLS0} preserves
\begin{alignat*}{3}
&\text{the mass}&&\quad\mM(u)&&=\int_{\R^d}|u|^2\,dx,\\
&\text{the Hamiltonian}&&\quad\mH(u)&&=\int_{\R^d}\frac{1}{2}|\nabla u|^2-\frac{\mu_1}{p_1}|u|^{p_1}-\frac{\mu_2}{p_2}|u|^{p_2}\,dx,\\
&\text{the momentum}&&\quad\mathcal{P}(u)&&=\int_{\R^d}\mathrm{Im}(\bar{u}\nabla u)\,dx
\end{alignat*}
over time. It is also easy to check that any solution $u$ of \eqref{NLS0} is invariant under time and space translation. Direct calculation also shows that \eqref{NLS0} remains invariant under the Galilean transformation
\begin{align*}
u(t,x)\mapsto e^{i\xi\cdot x}e^{-it|\xi|^2}u(t,x-2\xi t)
\end{align*}
for any $\xi\in\R^d$. Moreover, we say that a function $P$ is a \textit{soliton} solution of \eqref{NLS0} if $P$ solves the equation
\begin{align}\label{NLS0 soliton}
-\Delta P+\omega P-\mu_1|P|^{p_1-2}P-\mu_2|P|^{p_2-2}P=0
\end{align}
for some $\omega\in\R$. One easily verifies that $u(t,x):= e^{i\omega t}P(x)$ is a solution of \eqref{NLS0}. As we will see later, the soliton solutions play a very important role in the study of dispersive equations, since they can be seen as the balance point between dispersive and nonlinear effects.

When $\mu_1=0$, \eqref{NLS0} reduces to the NLS
\begin{align}
i\pt_t u+\Delta u+\mu|u|^{p-2}u=0\label{NLS single}
\end{align}
with pure power type nonlinearity, which has been extensively studied in literature. In particular, a solution of \eqref{NLS single} also exhibits the scaling invariance
\begin{align}
u(t,x)\mapsto \ld^{\frac{2}{p-2}}u(\ld^2t,\ld x)\label{scaling}
\end{align}
for any $\ld>0$, which plays a fundamental role in the study of \eqref{NLS single}. We also say that \eqref{NLS single} is $s_c$-critical with $s_c=s_c(p)=\frac{d}{2}-\frac{2}{p-2}$. It is easy to verify that the $\dot{H}^{s_c}$-norm is also invariant under the scaling \eqref{scaling}. We are particularly interested in the cases $s_c=0$ and $s_c=1$: In order to guarantee one or more conservation laws, we demand the solution of the NLS to be at least of class $L^2$ or $\dot{H}^1$. Moreover, we see that the mass and Hamiltonian are invariant under the $0$- and $1$-scaling respectively.

Concerning the Cauchy problem \eqref{NLS single}, Cazenave and Weissler \cite{CazenaveWeissler1,Cazenave1990} show that \eqref{NLS single} with $p\in(2,\tbs]$ defined on some interval $I\ni t_0$ is locally well-posed in $H^{1}(\R^d)$ on the maximal lifespan $I_{\max}\ni t_0$. In particular, if $p\in (2,\tbs)$ (namely the problem is energy-subcritical), then $u$ blows-up at finite time $t_{\sup}:=\sup I_{\max}$ if and only if
\begin{align}\label{subcritical criterion}
\lim_{\,\,t\,\uparrow\, t_{\sup}} \|\nabla u(t)\|_2=\infty.
\end{align}
A similar result holds for the negative time direction. Combining with the Gagliardo-Nirenberg inequality, it is immediate that \eqref{NLS single} having defocusing energy-subcritical nonlinearity or mass-subcritical nonlinearity (regardless of the sign) is always globally well-posed in $H^1(\R^d)$. However, this does not hold for focusing mass-supercritical and energy-subcritical \eqref{NLS single}: One can construct blow-up solutions using the celebrated virial identity due to Glassey \cite{Glassey1977} for initial data possessing negative energy. By a straightforward modification (see for instance \cite{Cazenave2003}) the results from \cite{CazenaveWeissler1,Cazenave1990} extend naturally to \eqref{NLS0}.

The blow-up criterion \eqref{subcritical criterion} does not carry over to the energy-critical case, since in this situation the well-posedness result also depends on the profile of the initial data. Using the so called induction on energy method, Bourgain \cite{BourgainRadial} is able to show that the defocusing energy-critical NLS is globally well-posed and scattering\footnote{For \eqref{NLS single} with pure mass- or energy-critical nonlinearity, the scattering space is referred to $L^2(\R^d)$ or $\dot{H}^1(\R^d)$ respectively, while for \eqref{NLS double crit} we consider scattering in $H^1(\R^d)$.} (we refer to Definition \ref{scaling} below for a precise definition of a scattering solution) for any radial initial data in $\dot{H}^1(\R^d)$ in the case $d=3$. Using the interaction Morawetz inequalities, the I-team \cite{defocusing3d} successfully removes the radial assumption in \cite{BourgainRadial}. The result in \cite{defocusing3d} is later extended to arbitrary dimension $d\geq 4$ \cite{defocusing4d,defocusing5dandhigher} and the well-posedness and scattering problem for the defocusing energy-critical NLS is completely resolved.

Utilizing the Glassey's virial arguments one verifies that a solution of the focusing energy-critical NLS is not always globally well-posed and scattering. On the other hand, appealing to standard contraction iteration we are able to show that the focusing energy-critical NLS is globally well-posed and scattering for small initial data. It turns out that the strict threshold, under which the small data theory takes place, can be described by the Aubin-Talenti-function
$$W(x):=\bg(1+\frac{|x|^2}{d(d-2)}\bg)^{-\frac{d-2}{2}},$$
which solves the Lane-Emden equation
\begin{align*}
-\Delta W=W^{\tbs-1}
\end{align*}
and is an optimizer of the Sobolev inequality
\begin{align*}
\mS:=\inf_{u\in\mD^{1,2}(\R^d)}\frac{\|u\|^2_{\dot{H}^1}}{\|u\|^2_\tbs}.
\end{align*}
Using the concentration compactness principle, Kenig and Merle \cite{KenigMerle2006} are able to prove the following large data scattering result concerning the focusing energy-critical NLS:
\begin{theorem}[\cite{KenigMerle2006}]\label{kenig merle theorem}
Let $d\in\{3,4,5\}$, $p=\tbs$ and $\mu=1$. Let also $u$ be a solution of \eqref{NLS single} with $u(0)=u_0\in \dot{H}^1_{\mathrm{rad}}(\R^d)$, $\mH^*(u_0)<\mH^*(W)$ and $\|u_0\|_{\dot{H}^1}<\|W\|_{\dot{H}^1}$, where
$$ \mH^*(u):=\frac{1}{2}\|u\|_2^2-\frac{1}{\tbs}\|u\|_\tbs^\tbs.$$
Then $u$ is global and scattering in time.
\end{theorem}
The result by Kenig and Merle is later extended by Killip and Visan \cite{KillipVisan2010focusing} to arbitrary dimension $d\geq 5$, where the radial assumption is also removed. Until very recently, Dodson \cite{Dodson4dfocusing} also removes the radial assumption in the case $d=4$. The 3D large data scattering problem for general initial data in $\dot{H}^1(\R^3)$ still remains open.

Based on the methodologies developed for the energy-critical NLS, Dodson is able to prove similar global well-posedness and scattering results for the mass-critical NLS. For the defocusing case, Dodson \cite{dodson1d,dodson2d,dodson3d} shows that a solution of the defocusing mass-critical NLS is always global and scattering in time for any initial data $u_0\in L^2(\R^d)$ with $d\geq 1$. To formulate the corresponding result for the focusing case, we denote by $Q$ the unique positive and radial solution of the stationary focusing mass-critical NLS
\begin{align*}
-\Delta Q+Q=Q^{\tas-1}.
\end{align*}
For the existence and uniqueness of $Q$, we refer to \cite{weinstein} and \cite{Kwong_uniqueness} respectively. The following result is due to Dodson \cite{Dodson4dmassfocusing} concerning the focusing mass-critical NLS:
\begin{theorem}[\cite{Dodson4dmassfocusing}]
Let $d\geq 1$, $p=\tas$ and $\mu=1$. Let also $u$ be a solution of \eqref{NLS single} with $u(0)=u_0\in L^2(\R^d)$ and $\mM(u_0)<\mM(Q)$. Then $u$ is global and scattering in time.
\end{theorem}

In recent years, problems with combined nonlinearities \eqref{NLS0} have been attracting much attention from the mathematical community. The mixed type nature of \eqref{NLS0} prevents itself to be scale-invariant and several arguments for \eqref{NLS single} fail to hold, which makes the analysis for \eqref{NLS0} rather delicate and challenging. A systematic study on \eqref{NLS0} is initiated by Tao, Visan and Zhang in their seminal paper \cite{TaoVisanZhang}. In particular, based on the interaction Morawetz inequalities the authors show that a solution of \eqref{NLS0} with $\mu_1,\mu_2<0$ and $p_1=\tas$, $p_2=\tbs$ (namely the defocusing-defocusing double critical regime) is always global and scattering in time for any initial data $u_0\in H^1(\R^d)$. As can be expected, this does not hold when at least one of the $\mu_i$ in \eqref{NLS0} is negative. Using concentration compactness and perturbation arguments initiated by \cite{Ibrahim2011}, Akahori, Ibrahim, Kikuchi and Nawa \cite{Akahori2013} are able to formulate a sharp scattering threshold for \eqref{NLS0} in the case $d\geq 5$, $\mu_1,\mu_2>0$, $p_1\in(\tas,\tbs)$ and $p_2=\tbs$ (namely the focusing energy-critical NLS perturbed by a focusing mass-supercritical and energy-subcritical nonlinearity). The methodology of \cite{Ibrahim2011,Akahori2013} becomes nowadays a golden rule for the study on large data scattering problems of NLS with combined nonlinearities. In this direction, we refer to the representative papers \cite{MiaoDoubleCrit,killip_visan_soliton,Cheng2020,Carles_Sparber_2021,luo2021sharp} for large data scattering results of \eqref{NLS0} in different regimes, where at least one of the nonlinearities possesses critical growth.

\subsubsection*{Main results}

In this paper, we study the most interesting and difficult case \eqref{NLS double crit}, where the mass- and energy-critical nonlinearities exist simultaneously in the equation. Roughly speaking, we can not consider \eqref{NLS double crit} as the energy-critical NLS perturbed by the mass-critical nonlinearity, nor vice versa, due to the endpoint critical nature of the potential terms. Nevertheless, it is quite natural to have the following heuristics on the long time dynamics of \eqref{NLS double crit} based on the results for NLS with single mass- or energy-critical potentials:
\begin{itemize}
\item For the defocusing-defocusing case, we expect that both of the mass- and energy-critical nonlinear terms are harmless, and a solution of \eqref{NLS double crit} should be global and scattering in time for arbitrary initial data $u_0$ from $H^1(\R^d)$.
\item For the focusing-defocusing case, we expect that under the stabilization of the defocusing energy-critical potential, a solution of \eqref{NLS double crit} should always be global. However, a bifurcation of scattering and soliton solutions might occur, which is determined by the mass of the initial data. In view of scaling, we conjecture that the threshold is given by $\mM(Q)$.
\item For the defocusing-focusing case, we expect that the scattering threshold should be uniquely determined by the Hamiltonian of the initial data. In view of scaling, we conjecture that the threshold is given by $\mH^*(W)$.
\end{itemize}
We should discuss the focusing-focusing case separately, which is the most subtle one among the four regimes. One might expect that the restriction for the scattering threshold is coming from both of the mass and energy sides. In particular, a reasonable guess about the threshold would be
$$ \mM(u_0)<\mM(Q)\,\wedge\, \mH(u_0)<\mH^*(W).$$
This is however not the case. As shown by the following result by Soave, the actual energy threshold is strictly less than $\mH^*(W)$.
\begin{theorem}[\cite{SoaveCritical}]\label{soave}
Let $d\geq 3$ and $\mu_1=\mu_2=1$. Define
\begin{align}\label{variational problem for mc}
m_c:=\inf_{u\in H^1(\R^d)}\{\mH(u):\mM(u)=c,\,\mK(u)=0\},
\end{align}
where $\mK$ is defined by
\begin{align*}
\mK(u)&:=\|\nabla u\|_2^2-\frac{d}{d+2}\|u\|_\tas^\tas-\|u\|_\tbs^\tbs.
\end{align*}
Then
\begin{itemize}
\item[(i)]\textbf{Existence of ground state}: For any $c\in(0,\mM(Q))$, the variational problem \eqref{variational problem for mc} has a positive and radially symmetric minimizer $P_c$ with $m_c=\mH(P_c)\in(0,\mH^*(W))$. Moreover, $P_c$ is a solution of
\begin{align}\label{standing wave eq}
-\Delta P_c+\omega P_c=P_c^{\tas-1}+P_c^{\tbs-1}
\end{align}
for some $\omega>0$.
\item[(ii)]\textbf{Blou-up criterion}: Assume that $u_0\in H^1(\R^d)$ satisfies
$$\mM(u_0)\in(0,\mM(Q))\,\wedge\,\mH(u_0)<m_{\mM(u_0)}\,\wedge\,\mK(u_0)<0.$$
Assume also that $|x|u_0\in L^2(\R^d)$. Then the solution $u$ of \eqref{NLS double crit} with $u(0)=u_0$ blows-up in finite time.
\end{itemize}
\end{theorem}

\begin{remark}
The quantity $\mK(u)$ is referred to the virial of $u$, which is closely related to the Glassey's virial identity and plays a fundamental role in the study of NLS.
\end{remark}
We make the intuitive heuristics into the following rigorous statements:
\begin{conjecture}\label{conjecture}
Let $d\geq 3$ and consider \eqref{NLS double crit} on some time interval $I\ni 0$. Let $u$ be the unique solution of \eqref{NLS double crit} with $u(0)=u_0\in H^1(\R^d)$. We also define
$$ \mK(u):=\|\nabla u\|_2^2-\mu_1\frac{d}{d+2}\|u\|_\tas^\tas-\mu_2\|u\|_\tbs^\tbs.$$
Then
\begin{itemize}
\item[(i)] \textbf{Defocusing-defocusing regime}: Let $\mu_1=\mu_2=-1$. Then $u$ is global and scattering in time.
\item[(ii)] \textbf{Focusing-defocusing regime}: Let $\mu_1=1$ and $\mu_2=-1$. Then $u$ is a global solution. If additionally $\mM(u_0)<\mM(Q)$, then $u$ is also scattering in time.
\item[(iii)] \textbf{Defocusing-focusing regime}: Let $\mu_1=-1$ and $\mu_2=1$. Assume that
$$ \mH(u_0)<\mH^*(W)\,\wedge\, \mK(u_0)> 0.$$
Then $u$ is global and scattering in time.
\item[(iv)] \textbf{Focusing-focusing regime}: Let $\mu_1=\mu_2=1$.
Assume that
$$\mM(u_0)<\mM(Q)\,\wedge\, \mH(u_0)<m_{\mM(u_0)}\,\wedge\, \mK(u_0)> 0.$$
Then $u$ is global and scattering in time.
\end{itemize}
\end{conjecture}
As mentioned previously, Conjecture \ref{conjecture} (i) is already proved by Tao, Visan and Zhang \cite{TaoVisanZhang}. Moreover, Conjecture \ref{conjecture} (iii) is proved by Cheng, Miao and Zhao \cite{MiaoDoubleCrit} in the case $d\leq 4$ and the author \cite{luo2021scattering} in the case $d\geq 5$, both under the additional assumption that $u_0$ is radially symmetric.

In this paper, we prove Conjecture \ref{conjecture} for general initial data from $H^1(\R^d)$. Our main result is as follows:
\begin{theorem}\label{main theorem}
We assume in the cases $d=3$, $\mu_1=-1$, $\mu_2=1$ and $d=3$, $\mu_1=\mu_2=1$ additionally that $u_0$ is radially symmetric. Then Conjecture \ref{conjecture} holds for any $d\geq 3$.
\end{theorem}

\begin{remark}
The radial assumption by Theorem \ref{main theorem} is removable as long as Theorem \ref{kenig merle theorem} also holds for general non-radial initial data from $\dot{H}^1(\R^3)$, which is widely believed to be true.
\end{remark}

The sharpness of the scattering threshold for the focusing-focusing \eqref{NLS double crit} is already revealed by Theorem \ref{soave}. The criticality of the threshold for the defocusing-focusing \eqref{NLS double crit} is more subtle, since in general there exists no soliton solution for the corresponding stationary equation, see \cite[Thm. 1.2]{SoaveCritical}. Nevertheless, we have the following variational characterization of the scattering threshold:

\begin{proposition}\label{proposition mc for df}
Let $\mu_1=-1$ and $\mu_2=1$. Let $m_c$ be defined through \eqref{variational problem for mc}. Then $m_c=\mH^*(W)$ and \eqref{variational problem for mc} has no optimizer for any $c\in(0,\infty)$.
\end{proposition}
The proof of Proposition \ref{proposition mc for df} follows the same line of \cite[Prop. 1.2]{MiaoDoubleCrit}, but we will consider the variational problem on a manifold with prescribed mass, which complexifies the arguments at several places. Moreover, it is shown in \cite{MiaoDoubleCrit} that any solution of the defocusing-focusing \eqref{NLS double crit} with initial data $u_0$ satisfying
$$ |x|u_0\in L^2(\R^d)\,\wedge\,\mH(u_0)<\mH^*(W)\,\wedge\, \mK(u_0)< 0$$
must blow-up in finite time. This gives a complete description of the criticality of the scattering threshold for the defocusing-focusing \eqref{NLS double crit}.

For the focusing-defocusing regime, it is shown by Zhang \cite{Zhang2006} and Tao, Visan and Zhang \cite{TaoVisanZhang} that a solution of the focusing-defocusing \eqref{NLS double crit} is always globally well-posed, hence the blow-up solutions are ruled out. Using simple variational arguments we will show the existence of ground states at arbitrary mass level larger than $\mM(Q)$.

\begin{proposition}\label{proposition for ground state}
Let $\mu_1=1$ and $\mu_2=-1$. Define
\begin{align}\label{variational problem fd}
\gamma_c:=\inf_{u\in H^1(\R^d)}\{\mH(u):M(u)=c\}.
\end{align}
Then
\begin{itemize}
\item[(i)] The mapping $c\mapsto \gamma_c$ is monotone decreasing on $(0,\infty)$, equal to zero on $(0,M(Q)]$ and negative on $(M(Q),\infty)$.
\item[(ii)] For all $c\in(0,M(Q)]$, \eqref{variational problem fd} has no minimizer.
\item[(iii)] For all $c\in(M(Q),\infty)$, \eqref{variational problem fd} has a positive and radially symmetric minimizer $S_c$. Consequently, $S_c$ is a solution of
\begin{align}
-\Delta S_c+\omega S_c=S_c^{\tas-1}-S_c^{\tbs-1}
\end{align}
with some $\omega\in\bg(0,\frac{2}{d}\bg(\frac{d}{d+2}\bg)^{\frac{d}{2}}\bg)$.
\end{itemize}
\end{proposition}

It remains an interesting problem what can be said about the focusing-defucosing model by the borderline case $\mM(u_0)=\mM(Q)$. As suggested by the results in \cite{carles2020soliton,Murphy2021CPDE}, we conjecture that scattering also takes place at the critical mass. We plan to tackle this problem in a forthcoming paper.

\subsubsection*{Roadmap for the large data scattering results}
To prove Theorem \ref{main theorem}, we follow the standard concentration compactness arguments initiated by Kenig and Merle \cite{KenigMerle2006}. In view of the stability theory (Lemma \ref{long time pert}), the main challenge will be to verify the smallness condition
\begin{align}\label{smalness example}
\|\la\nabla\ra e\|_{L_{t,x}^{\frac{2(d+2)}{d+4}}(\R)}\ll 1
\end{align}
for the error term $e$. Roughly speaking, to achieve \eqref{smalness example} we demand the remainders $w_n^k$ given by the linear profile decomposition to satisfy the asymptotic smallness condition
\begin{align}\label{smallness remainders 1}
\lim_{k\to K^*}\lim_{n\to\infty}\|e^{it\Delta}w_n^k\|_{L_{t,x}^{\frac{2(d+2)}{d}}\,\cap\, L_{t,x}^{\frac{2(d+2)}{d-2}}(\R)}=0.
\end{align}
However, this is impossible by applying solely the $L^2$- or $\dot{H}^1$-profile decomposition. To solve this problem, Cheng, Miao and Zhao \cite{MiaoDoubleCrit} establish a profile decomposition which is obtained by first applying the $L^2$-profile decomposition to the (radial) underlying sequence $(\la\nabla\ra\psi_n)_n$ and then undoing the transformation. The robustness of such profile decomposition lies in the fact that the remainders satisfy the even stronger asymptotic smallness condition
\begin{align*}
\lim_{k\to K^*}\lim_{n\to\infty}\|\la\nabla\ra e^{it\Delta}w_n^k\|_{L_{t,x}^{\frac{2(d+2)}{d}}(\R)}=0.
\end{align*}
\eqref{smallness remainders 1} follows immediately from the Strichartz inequality. However, the radial assumption is essential, which guarantees that the Galilean boosts appearing in the $L^2$-profile decomposition are constantly equal to zero. Indeed, we may also apply the full $L^2$-profile decomposition to the possibly non-radial underlying sequence, by also taking the non-vanishing Galilean boosts into account. However, by doing in such a way the Galilean boosts are generally unbounded, and such unboundedness induces a very strong loss of compactness which leads to the failure of decomposition of the Hamiltonian. Heuristically, the occurrence of the compactness defect is attributed to the fact that the profile decomposition in \cite{MiaoDoubleCrit} can still be seen as a variant of the $L^2$-profile decomposition, hence it is insufficiently sensitive to the high frequency bubbles.

Our solution is based on a refinement of the classical profile decompositions. Notice that the profile decompositions are obtained by an iterative process. At each iterative step we will face a bifurcate decision: either
\begin{align*}
\text{(i) } &\limsup_{n\to\infty}\|e^{it\Delta}w_n^k\|_{L_{t,x}^{\frac{2(d+2)}{d}}(\R)}\geq \limsup_{n\to\infty}\|e^{it\Delta}w_n^k\|_{L_{t,x}^{\frac{2(d+2)}{d-2}}(\R)},\text{ or}\nonumber\\
\text{(ii) } &\limsup_{n\to\infty}\|e^{it\Delta}w_n^k\|_{L_{t,x}^{\frac{2(d+2)}{d}}(\R)}< \limsup_{n\to\infty}\|e^{it\Delta}w_n^k\|_{L_{t,x}^{\frac{2(d+2)}{d-2}}(\R)}.
\end{align*}
In the former case, we apply the $L^2$-decomposition to continue, while in the latter case we apply the $\dot{H}^1$-decomposition. Then \eqref{smallness remainders 1} follows immediately from the construction of the profile decomposition. Moreover, since at each iterative step we are applying the profile decomposition to a bounded sequence in $H^1(\R^d)$, the resulting Galilean boosts are thus bounded. Using this additional property of the Galilean boosts we are able to show that the Hamiltonian of the bubbles are perfectly decoupled as desired. We refer to Lemma \ref{linear profile} for details.

On the other hand, we will build up the minimal blow-up solution using the mass-energy-indicator (MEI) functional $\mD$. This is firstly introduced in \cite{killip_visan_soliton} for studying the large data scattering problems for 3D focusing-defocusing cubic-quintic NLS and later further applied in \cite{luo2021sharp} for the 2D and 3D focusing-focusing cubic-quintic NLS. The usage of the MEI-functional is motivated by the fact that the underlying inductive scheme relies only on the mass and energy of the initial data and the scattering regime is immediately readable from the mass-energy diagram, see Fig. \ref{MEI} below. The idea can be described as follows: a mass-energy pair $(\mM(u),\mH(u))$ being admissible will imply $\mD(u)\in(0,\infty)$; In order to escape the admissible region $\Omega$, a function $u$ must approach the boundary of $\Omega$ and one deduces that $\mD(u)\to\infty$. We can therefore assume that the supremum $\mD^*$ of $\mD(u)$ running over all admissible $u$ is finite, which leads to a contradiction and we conclude that $\mD^*=\infty$, which will finish the desired proof. However, in the regime $\mu_2=1$ a mass-energy pair being admissible does not automatically imply the positivity of the virial $\mK$. In particular, it is not trivial at the first glance that the linear profiles have positive virial. We will appeal to the geometric properties of the MEI-functional $\mD$, combining with the variational arguments from \cite{Akahori2013}, to overcome this difficulty.

\begin{figure}
  \centering
  \includegraphics[width=150mm,height=30mm]{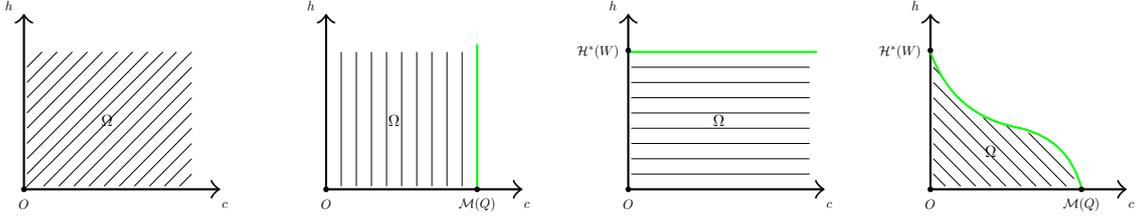}
\caption{An illustration for the admissible domains $\Omega$ in different regimes, where the shadow region is the intersection of $\Omega$ and $(0,\infty)^2$. From left to right: D-D-regime, F-D-regime, D-F-regime and F-F-regime.}\label{MEI}
\end{figure}

\begin{remark}
By straightforward modification of the method developed in this paper, we are also able to give a new proof for the scattering result in the defocusing-defocusing regime using the concentration compactness principle.
\end{remark}

\subsubsection*{Outline of the paper}
The paper is organized as follows: In Section \ref{section preliminaries} we establish the small data and stability theories for the \eqref{NLS double crit}. In Section \ref{section profile decomp} we construct the double track profile decomposition. Section \ref{section f f} to Section \ref{section f d} are devoted to the proof of Theorem \ref{main theorem}, Proposition \ref{proposition mc for df} and Proposition \ref{proposition for ground state}. In Appendix \ref{section endpoint} we establish the endpoint values of the curve $c\mapsto m_c$ for the focusing-focusing \eqref{NLS double crit}.

\subsection{Notations and definitions}
We will use the notation $A\lesssim B$ whenever there exists some positive constant $C$ such that $A\leq CB$. Similarly we define $A\gtrsim B$ and we will use $A\sim B$ when $A\lesssim B\lesssim A$. We denote by $\|\cdot\|_p$ the $L^p(\R^d)$-norm for $p\in[1,\infty]$. We similarly define the $H^1(\R^d)$-norm by $\|\cdot\|_{H^1}$. The following quantities will be used throughout the paper:
\begin{align*}
\mM(u)&:=\|u\|^2_2,\nonumber\\
\mH(u)&:=\frac{1}{2}\|\nabla u\|_2^2-\frac{\mu_1}{\tas}\|u\|^{\tas}_{\tas}-\frac{\mu_2}{\tbs}\|u\|^{\tbs}_{\tbs},\nonumber\\
\mK(u)&:=\|\nabla u\|_2^2-\mu_1\frac{d}{d+2}\|u\|_\tas^\tas-\mu_2\|u\|_\tbs^\tbs,\nonumber\\
\mI(u)&:=\mH(u)-\frac{1}{2}\mK(u)=\frac{\mu_2}{d}\|u\|_\tbs^\tbs.
\end{align*}
We will also frequently use the scaling operator
\begin{align*}
T_\ld u(x)&:=\ld^{\frac{d}{2}} u(\ld x).
\end{align*}
One easily verifies that the $L^2$-norm is invariant under this scaling. Throughout the paper, we denote by $g_{\xi_0,x_0,\ld_0}$ the $L^2$-symmetry transformation which is defined by
\begin{align*}
g_{\xi_0,x_0,\ld_0}f(x):=\ld_0^{-\frac{d}{2}}e^{i\xi_0\cdot x}f(\ld_0^{-1}(x-x_0))
\end{align*}
for $(\xi_0,x_0,\ld_0)\in\R^d\times\R^d\times(0,\infty)$.

We denote by $Q$ the unique positive and radially symmetric solution of
\begin{align*}
-\Delta Q+Q=Q^{\tas-1}.
\end{align*}
We denote by $\mathrm{C}_{\mathrm{GN}}$ the optimal $L^2$-critical Gagliardo-Nirenberg constant, i.e.
\begin{align}\label{def gn l2 crit}
\mathrm{C}_{\mathrm{GN}}=\inf_{u\in H^1(\R^d)\setminus\{0\}}\frac{\|\nabla u\|_2^2\|u\|_2^{\frac{4}{d}}}{\|u\|_{\tas}^{\tas}}.
\end{align}
Using Pohozaev identities (see for instance \cite{lions1}), the uniqueness of $Q$ and scaling arguments one easily verifies that
\begin{align}\label{GN-L1}
\mathrm{C}_{\mathrm{GN}}=\frac{d}{d+2}(\mM(Q))^{\frac{2}{d}}.
\end{align}
We also denote by $\csob$ the optimal constant for the Sobolev inequality, i.e.
\begin{align*}
\csob:=\inf_{u\in\mathcal{D}^{1,2}(\R^d)\setminus \{0\}}\frac{\|\nabla u\|_2^2}{\|u\|_{2^*}^2}.
\end{align*}
Here, the space $\mathcal{D}^{1,2}(\R^d)$ is defined by
\begin{align*}
\mathcal{D}^{1,2}(\R^d):=\{u\in L^{2^*}(\R^d):\nabla u\in L^2(\R^d)\}.
\end{align*}

For an interval $I\subset \R$, the space $L_t^qL_x^r(I)$ is defined by
\begin{align*}
L_t^qL_x^r(I):=\{u:I\times \R^2\to\C:\|u\|_{L_t^qL_x^r(I)}<\infty\},
\end{align*}
where
\begin{align*}
\|u\|^q_{L_t^qL_x^r(I)}:=\int_{\R}\|u\|^q_r\,dt.
\end{align*}
The following spaces will be frequently used throughout the paper:
\begin{align*}
S(I)&:=L_t^\infty L_x^2(I)\cap L_t^2 L_x^{\tbs}(I),\nonumber\\
V_{\tbs}(I)&:=L_t^{\frac{2(d+2)}{d-2}}L_x^{\frac{2d(d+2)}{d^2+4}}(I),\nonumber\\
W_{\tbs}(I)&:=L_{t,x}^{\frac{2(d+2)}{d-2}}(I),\nonumber\\
W_{\tas}(I)&:=L_{t,x}^{\frac{2(d+2)}{d}}(I).
\end{align*}

A pair $(q,r)$ is said to be $L^2$-admissible if $q,r\in[2,\infty]$, $\frac{2}{q}+\frac{d}{r}=\frac{d}{2}$ and $(q,r,d)\neq(2,\infty,2)$. For any $L^2$-admissible pairs $(q_1,r_1)$ and $(q_2,r_2)$ we have the following Strichartz estimates: if $u$ is a solution of
\begin{align*}
i\pt_t u+\Delta u=F(u)
\end{align*}
in $I\subset\R$ with $t_0\in I$ and $u(t_0)=u_0$, then
\begin{align*}
\|u\|_{L_t^q L_x^r(I)}\lesssim \|u_0\|_2+\|F(u)\|_{L_t^{q_2'} L_x^{r_2'}(I)},
\end{align*}
where $(q_2',r_2')$ is the H\"older conjugate of $(q_2,r_2)$. For a proof, we refer to \cite{EndpointStrichartz,Cazenave2003}.

In this paper, we use the following concepts for solution and scattering of \eqref{NLS double crit}:

\begin{definition}[Solution]
A function $u: I\times \R^d\to \C$ is said to be a solution of \eqref{NLS double crit} on the interval $I\subset\R$ if for any compact $J\subset I$, $u\in C(J;H^1(\R^d))$ and for all $t,t_0\in I$
\begin{align*}
u(t)=e^{i(t-t_0)\Delta}u(t_0)+i\int_{t_0}^te^{i(t-s)\Delta}[\mu_1|u|^{\frac{4}{d}}u+\mu_2|u|^{\frac{4}{d-2}}u](s)\,ds.
\end{align*}
\end{definition}

\begin{definition}[Scattering]\label{scattering definition}
A global solution $u$ of \eqref{NLS double crit} is said to be forward in time scattering if there exists some $\phi_+\in H^1(\R^d)$ such that
\begin{align*}
\lim_{t\to\infty}\|u(t)-e^{it\Delta}\phi_+\|_{H^1}=0.
\end{align*}
A backward in time scattering solution is similarly defined. $u$ is then called a scattering solution when it is both forward and backward in time scattering.
\end{definition}

We define the Fourier transformation of a function $f$ by
\begin{align*}
\hat{f}(\xi)=\mathcal{F}(f)(\xi):=(2\pi)^{-\frac{d}{2}}\int_{\R^d}f(x)e^{-i\xi\cdot x}\,dx.
\end{align*}
For $s\in\R$, the multipliers $|\nabla|^s$ and $\la\nabla\ra^s$ are defined by the symbols
\begin{align*}
|\nabla|^s f(x)&=\mathcal{F}^{-1}\bg(|\xi|^s\hat{f}(\xi)\bg)(x),\\
\la\nabla\ra^s f(x)&=\mathcal{F}^{-1}\bg((1+|\xi|^2)^{\frac{s}{2}}\hat{f}(\xi)\bg)(x).
\end{align*}
Let $\psi\in C^\infty_c(\R^2)$ be a fixed radial, non-negative and radially decreasing function such that $\psi(x)=1$ if $|x|\leq 1$ and $\psi(x)=0$ for $|x|\geq \frac{11}{10}$. Then for $N>0$, we define the Littlewood-Paley projectors by
\begin{align*}
P_{\leq N} f(x)&=\mathcal{F}^{-1}\bg(\psi\bg(\frac{\xi}{N}\bg)\hat{f}(\xi)\bg)(x),\\
P_{N} f(x)&=\mathcal{F}^{-1}\bg(\bg(\psi\bg(\frac{\xi}{N}\bg)-\psi\bg(\frac{2\xi}{N}\bg)\bg)\hat{f}(\xi)\bg)(x),\\
P_{> N} f(x)&=\mathcal{F}^{-1}\bg(\bg(1-\psi\bg(\frac{\xi}{N}\bg)\bg)\hat{f}(\xi)\bg)(x).
\end{align*}
We also record the following well-known Bernstein inequalities which will be frequently used throughout the paper: For all $s\geq 0$ and $1\leq p\leq q\leq\infty$ we have
\begin{align*}
\|P_{> N}f\|_{L^p}&\lesssim N^{-s}\||\nabla|^s P_{> N}f\|_{L^p},\\
\||\nabla|^s P_{\leq N}f\|_{L^p}&\lesssim N^{s}\| P_{\leq N}f\|_{L^p},\\
\||\nabla|^{\pm s} P_{ N}f\|_{L^p}&\sim N^{\pm s}\| P_{ N}f\|_{L^p},\\
\|P_{\leq N}f\|_{L^q}&\lesssim N^{\frac{n}{p}-\frac{n}{q}}\|P_{\leq N}f\|_{L^p},\\
\|P_{N}f\|_{L^q}&\lesssim N^{\frac{n}{p}-\frac{n}{q}}\|P_{N}f\|_{L^p}.
\end{align*}
The following useful elementary inequality will be frequently used in the paper: For $s\in\{0,1\}$ and $z_1,\cdots,z_k\in\C$ we have
\begin{align}\label{elementary ineq}
&\bg||\nabla|^s\bg(\bg|\sum_{j=1}^k z_j\bg|^\alpha \bg(\sum_{j=1}^k z_j\bg)-\sum_{j=1}^k |z_j|^\alpha z_j\bg)\bg|
\nonumber\\
\lesssim_{k,\alpha}&\left\{
             \begin{array}{ll}
             \sum_{j\neq j'}||\nabla|^s z_j||z_{j'}|^{\alpha},&\text{if $0<\alpha\leq 1$},\\
             \sum_{j\neq j'}||\nabla|^s z_j||z_{j'}|(|z_{j}|+|z_{j'}|)^{\alpha-1},&\text{if $\alpha> 1$}.
             \end{array}
\right.
\end{align}

We end this section with the following useful local smoothing result:
\begin{lemma}[\cite{KillipVisan2010focusing}]
Given $\phi\in\dot{H}^1(\R^d)$ we have
\begin{align}
\|\nabla e^{it\Delta}\phi\|^3_{L_{t,x}^2([-T,T]\times \{|x|\leq R\})}\lesssim T^{\frac{2}{d+2}}R^{\frac{3d+2}{d+2}}
\|e^{it\Delta}\phi\|_{W_\tbs(\R)}\|\nabla \phi\|^2_2.\label{local killip visan 2}
\end{align}
\end{lemma}

\section{Small data and stability theories}\label{section preliminaries}
We record in this section the small data and stability theories for \eqref{NLS double crit}. The proof of the small data theory is standard, see for instance \cite{Cazenave2003,KillipVisanNotes}. We will therefore omit the details of the proof here.
\begin{lemma}[Small data theory]\label{well posedness lemma}
For any $A>0$ there exists some $\beta>0$ such that the following is true: Suppose that $t_0\in I$ for some interval $I$. Suppose also that $u_0\in H^1(\R^d)$ with
\begin{gather}
\|u_0\|_{H^1}\leq A,\\
\| e^{i(t-t_0)\Delta}u_0\|_{W_{\tas}\cap W_{\tbs}(I)}\leq \beta.
\end{gather}
Then \eqref{NLS double crit} has a unique solution $u\in C(I;H^1(\R^d))$ with $u(t_0)=u_0$ such that
\begin{align}
\|\la\nabla\ra u\|_{S(I)}&\lesssim \|u_0\|_{H^1},\\
\|u\|_{W_{\tas}\cap W_{\tbs}(I)}&\leq 2\| e^{i(t-t_0)\Delta}u_0\|_{W_{\tas}\cap W_{\tbs}(I)}.
\end{align}
By the uniqueness of the solution $u$ we can extend $I$ to some maximal open interval $I_{\max}=(T_{\min},T_{\max})$. We have the following blow-up criterion: If $T_{\max}<\infty$, then
\begin{align*}
\|u\|_{W_{\tas}\cap W_\tbs([T,T_{\max}))}=\infty
\end{align*}
for any $T\in I_{\max}$. A similar result holds for $T_{\min}>-\infty$. Moreover, if
\begin{align*}
\|u\|_{W_{\tas}\cap W_\tbs(I_{\max})}<\infty,
\end{align*}
then $I_{\max}=\R$ and $u$ scatters in time.
\end{lemma}

\begin{remark}\label{remark}
Using Strichartz and Sobolev inequalities we infer that
\begin{align*}
\| e^{i(t-t_0)\Delta}u_0\|_{W_{\tas}\cap W_{\tbs}(I)}\lesssim \|u_0\|_{H^1}.
\end{align*}
Thus Lemma \ref{well posedness lemma} is applicable for all $u_0$ with sufficiently small $H^1$-norm.
\end{remark}

We will also need the following persistence of regularity result for \eqref{NLS double crit}.

\begin{lemma}[Persistence of regularity for \eqref{NLS double crit}]\label{nls persistence}
Let $u$ be a solution of \eqref{NLS double crit} on some interval $I$ with $t_0\in I$ and $\|u\|_{W_\tas\cap W_\tbs(I)}<\infty$. Then
\begin{align}
\||\nabla|^s u\|_{S(I)}\leq C(\|u\|_{W_\tas\cap W_\tbs(I)},\||\nabla|^s u(t_0)\|_2).
\end{align}
\end{lemma}

\begin{proof}
We divide $I$ into $m$ subintervals $I_1,I_2,\cdots,I_m$ with $I_j=[t_{j-1},t_j]$ such that
\begin{align*}
\|u\|_{W_\tas\cap W_\tbs(I_j)}\leq \eta\ll 1.
\end{align*}
for some small $\eta$ which is to be determined later. Then by Strichartz we have
\begin{align*}
\||\nabla|^s u\|_{S(I_j)}\lesssim \||\nabla|^s u(t_j)\|_2+(\eta^{\frac{4}{d}}+\eta^{\frac{4}{d-2}})\||\nabla|^s u\|_{S(I_j)}.
\end{align*}
Therefore choosing $\eta$ sufficiently small (where the smallness depends only on the Strichartz constants and is uniform for all subintervals $I_j$) and starting with $j=1$ we have
\begin{align*}
\||\nabla|^s u\|_{S(I_1)}\leq C(\||\nabla|^s u(t_0)\|_2).
\end{align*}
In particular,
\begin{align*}
\||\nabla|^s u(t_1)\|_{2}\leq C(\||\nabla|^s u(t_0)\|_2).
\end{align*}
Arguing inductively for all $j=2,\cdots,m-1$ and summing the estimates on all subintervals up yield the desired claim.
\end{proof}

Now we prove the stability theory for \eqref{NLS double crit}, which is a stronger version of the ones from \cite{MiaoDoubleCrit,luo2021scattering} under the enhanced condition \eqref{condition a}.

\begin{lemma}[Stability theory]\label{long time pert}
Let $d\geq 3$ and let $u\in C(I;H^1(\R^d))$ be a solution of \eqref{NLS double crit} defined on some interval $I\ni t_0$. Assume also that $w\in C(I;H^1(\R^d))$ is an approximate solution of the following perturbed NLS
\begin{align}
i\pt_t w+\Delta w+\mu_1|w|^{\frac{4}{d}}w+\mu_2|w|^{\frac{4}{d-2}}w+e=0
\end{align}
such that
\begin{align}
\|u\|_{L_t^\infty H_x^1(I)}&\leq B_1,\label{condition c1}\\
\|w\|_{W_\tas\cap W_\tbs (I)}&\leq B_2\label{condition c2}
\end{align}
for some $B_1,B_2>0$. Then there exists some positive $\beta_0=\beta_0(B_1,B_2)\ll 1$ with the following property: if
\begin{align}
\|u(t_0)-w(t_0)\|_{H^1}&\leq \beta,\label{condition a}\\
\|\la\nabla\ra e\|_{L_{t,x}^{\frac{2(d+2)}{d+4}}(I)}&\leq\beta\label{condition b}
\end{align}
for some $0<\beta<\beta_0$, then
\begin{align}\label{also proxy}
\| \la\nabla\ra (u-w)\|_{S (I)}\lesssim_{B_1,B_2}\beta^\kappa.
\end{align}
for some $\kappa\in(0,1)$.
\end{lemma}

\begin{proof}
From the results given in \cite{MiaoDoubleCrit,luo2021scattering} we already know that
\begin{align*}
\|u-w\|_{W_\tas\cap W_\tbs (I)}&\lesssim_{B_1,B_2}\beta^\kappa,\\
\|\la\nabla \ra u\|_{S(I)}+\|\la\nabla \ra w\|_{S(I)}&\lesssim_{B_1,B_2}1
\end{align*}
for some $\kappa\in(0,1)$. We divide $I$ into $O\bg(\frac{C(B_1,B_2)}{\delta}\bg)$ intervals $I_1,\cdots,I_m$ such that
\begin{align*}
\| u\|_{W_\tas\cap W_\tbs (I_j)}+\| w\|_{W_\tas\cap W_\tbs (I_j)}\leq \delta
\end{align*}
for all $j=1,\cdots,m$, where $\delta>0$ is some small number to be determined. Denote $I_1=[t_0,t_1]$. Using H\"older and \eqref{elementary ineq} we infer that
\begin{align}
&\||\nabla|^s (|u|^{\frac{4}{d}}u-|w|^{\frac{4}{d}}w)\|_{L_{t,x}^{\frac{2(d+2)}{d+4}}(I_1)}\nonumber\\
\lesssim&
\left\{
             \begin{array}{ll}
             \|u-w\|_{W_\tas(I_1)}(\|u\|^{\frac{4-d}{d}}_{W_\tas(I_1)}+\|w\|^{\frac{4-d}{d}}_{W_\tas(I_1)})
             \||\nabla|^s w\|_{W_\tas(I_1)}\\
             \quad+(\|v\|^{\frac{4}{d}}_{W_\tas(I_1)}+\|w\|^{\frac{4}{d}}_{W_\tas(I_1)})\||\nabla|^s (u- w)\|_{W_\tas(I_1)},&\text{if $d=3$},\\
             \\
             (\|u\|^{\frac{4}{d}}_{W_\tas(I_1)}+\|w\|^{\frac{4}{d}}_{W_\tas(I_1)})\||\nabla|^s (u- w)\|_{W_\tas(I_1)}
             \\
             \quad+\|u-w\|^{\frac{4}{d}}_{W_\tas(I_1)}(\||\nabla|^s u\|_{W_\tas(I_1)}+
             \||\nabla|^s w\|_{W_\tas(I_1)}),&\text{if $d\geq 4$},\\
             \end{array}
\right.
\\
\nonumber\\
&\||\nabla|^s (|u|^{\frac{4}{d-2}}u-|w|^{\frac{4}{d-2}}w)\|_{L_{t,x}^{\frac{2(d+2)}{d+4}}(I_1)}\nonumber\\
\lesssim&
\left\{
             \begin{array}{ll}
             \|u-w\|_{W_\tbs(I_1)}(\|u\|^{\frac{6-d}{d-2}}_{W_\tbs(I_1)}+\|w\|^{\frac{6-d}{d-2}}_{W_\tbs(I_1)})
             \||\nabla|^s w\|_{W_\tas(I_1)}\\
             \quad+(\|u\|^{\frac{4}{d-2}}_{W_\tbs(I_1)}+\|w\|^{\frac{4}{d-2}}_{W_\tbs(I_1)})\||\nabla|^s (u- w)\|_{W_\tas(I_1)},&\text{if $d\leq5$},\\
             \\
             (\|u\|^{\frac{4}{d-2}}_{W_\tbs(I_1)}+\|w\|^{\frac{4}{d-2}}_{W_\tbs(I_1)})\||\nabla|^s (u- w)\|_{W_\tas(I_1)}
             \\
             \quad+\|u-w\|^{\frac{4}{d-2}}_{W_\tbs(I_1)}(\||\nabla|^s u\|_{W_\tas(I_1)}+\||\nabla|^s w\|_{W_\tas(I_1)}),&\text{if $d\geq 6$}\\
             \end{array}
\right.
\end{align}
for $s\in\{0,1\}$. By Strichartz we also see that
\begin{align}
\||\nabla|^s (u-w)\|_{S(I_1)}&\lesssim \||\nabla|^s (u(t_0)-w(t_0))\|_{L^2}
+\||\nabla|^s (|u|^{\frac{4}{d}}u-|w|^{\frac{4}{d}}w)\|_{L_{t,x}^{\frac{2(d+2)}{d+4}}(I_1)}\nonumber\\
&\quad\quad+\||\nabla|^s (|u|^{\frac{4}{d-2}}u-|w|^{\frac{4}{d-2}}w)\|_{L_{t,x}^{\frac{2(d+2)}{d+4}}(I_1)}
+\||\nabla|^s e\|_{L_{t,x}^{\frac{2(d+2)}{d+4}}(I_1)}.
\end{align}
Now we absorb the terms on the r.h.s. with $\|\nabla (u-w)\|_{W_\tas(I_1)}$ to the l.h.s. (which is possible by choosing $\delta$ sufficiently small) to deduce that
\begin{align*}
\||\nabla|^s (u-w)\|_{S(I_1)}\lesssim \beta^{\kappa}
\end{align*}
for some (possibly smaller) $\kappa\in(0,1)$. In particular, we have
\begin{align*}
\|u(t_1)-w(t_1)\|_{H^1}\lesssim \beta^{\kappa}.
\end{align*}
Therefore we can proceed with the previous arguments for all $I_2,\cdots,I_m$ to conclude that
\begin{align*}
\||\nabla|^s (u-w)\|_{S(I_j)}\lesssim \beta^{\kappa}
\end{align*}
for all $j=1,\cdots,m$. The claim follows by summing the estimates on each subinterval up.
\end{proof}


\section{Double track profile decomposition} \label{section profile decomp}
In this section we construct the double track profile decomposition for a bounded sequence in $H^1(\R^d)$. We begin with the following inverse Strichartz inequality along the $\dot{H}^1$-track, which is originally proved in \cite{killip_visan_soliton} in the case $d=3$ and can be extended to arbitrary dimension $d\geq 3$ straightforwardly.

\begin{lemma}[Inverse Strichartz inequality, $\dot{H}^1$-track, \cite{killip_visan_soliton}]\label{track h1}
Let $d\geq 3$ and $(f_n)_n\subset H^1(\R^d)$. Suppose that
\begin{align}
\lim_{n\to\infty}\|f_n\|_{H^1}=A<\infty\quad\text{and}\quad\lim_{n\to\infty}\|e^{it\Delta} f_n\|_{W_\tbs(\R)}=\vare>0.
\end{align}
Then up to a subsequence, there exist $\phi\in\dot{H}^1(\R^d)$ and $(t_n,x_n,\ld_n)_n\subset\R\times\R^d\times(0,\infty)$ such that $\ld_n\to \ld_\infty\in[0,\infty)$, and if $\ld_\infty>0$, then $\phi\in H^1(\R^d)$. Moreover,
\begin{align}
\ld_n^{\frac{d}{2}-1}(e^{it_n\Delta}f_n)(\ld_n x+x_n)\rightharpoonup \phi(x)\text{ weakly in }
\left\{
             \begin{array}{ll}
             H^1(\R^d),&\text{if $\ld_\infty>0$},\\
             \dot{H}^1(\R^d),&\text{if $\ld_\infty=0$}.
             \end{array}
\right.
\end{align}
Setting
\begin{align}
\phi_n:=
\left\{
             \begin{array}{ll}
             \ld_n^{-\frac{d}{2}-1}e^{-it_n\Delta}\bg[\phi(\frac{x-x_n}{\ld_n})\bg],&\text{if $\ld_\infty>0$},\\
             \\
             \ld_n^{-\frac{d}{2}-1}e^{-it_n\Delta}\bg[(P_{>\ld_n^\theta}\phi)(\frac{x-x_n}{\ld_n})\bg],&\text{if $\ld_\infty=0$}
             \end{array}
\right.
\end{align}
for some fixed $\theta\in(0,1)$, we have
\begin{align}
&\lim_{n\to\infty}(\|f_n\|^2_{\dot{H}^1}-\|f_n-\phi_n\|^2_{\dot{H}^1})=\|\phi\|^2_{\dot{H}^1}\gtrsim A^2\bg(\frac{\vare}{A}\bg)^{\frac{d(d+2)}{4}},
\label{refined strichartz decomp 1}\\
&\lim_{n\to\infty}(\|f_n\|^2_{\dot{H}^1}-\|f_n-\phi_n\|^2_{\dot{H}^1}-\|\phi_n\|^2_{\dot{H}^1})=0,
\label{refined strichartz decomp 3}\\
&\lim_{n\to\infty}(\|f_n\|^2_{2}-\|f_n-\phi_n\|^2_{2}-\|\phi_n\|^2_{2})=0.\label{refined strichartz decomp 2}
\end{align}
Furthermore, we have
\begin{align}
\text{(i) } &\ld_n\equiv 1\quad\text{or}\quad \ld_n\to 0,\\
\text{(ii) } &t_n\equiv 0\quad\text{or}\quad \frac{t_n}{\ld_n^2}\to \pm\infty
\end{align}
and
\begin{align}
&\|f_n\|_\tas^\tas=\|\phi_n\|_\tas^\tas+\|f_n-\phi_n\|_\tas^\tas+o_n(1),\\
&\|f_n\|_\tbs^\tbs=\|\phi_n\|_\tbs^\tbs+\|f_n-\phi_n\|_\tbs^\tbs+o_n(1).
\end{align}
\end{lemma}
Next, we establish the inverse Strichartz inequality along the $L^2$-track by using the arguments from the proof of Lemma \ref{track h1} and from \cite{KillipVisanNotes,torus0}. For each $j\in\Z$, define $\mC_j$ by
\begin{align*}
\mC_j:=\bg\{\Pi_{i=1}^d [2^jk_i,2^j(k_i+1))\subset\R^d:k\in\Z^d\bg\}
\end{align*}
and $\mC:=\cup _{j\in\Z}\,\mC_j$. Given $Q\in\mC$ we define $f_Q$ by $\hat{f}_Q:=\chi_Q \hat{f}$, where $\chi_Q$ is the characteristic function of the cube $Q$. We have the following improved Strichartz estimate:
\begin{lemma}[Improved Strichartz estimate, \cite{KillipVisanNotes}]\label{refined strichartz l2}
Let $d\geq 1$ and $q:=\frac{2(d^2+3d+1)}{d^2}$. Then
\begin{align}
\|e^{it\Delta}f\|_{W_\tas(\R)}\lesssim \|f\|_2^{\frac{d+1}{d+2}}\bg(\sup_{Q\in\mC}|Q|^{\frac{d+2}{dq}-\frac{1}{2}}\|e^{it\Delta}f_Q\|_{L_{t,x}^q(\R)}\bg)^{\frac{1}{d+2}}.
\end{align}
\end{lemma}

\begin{lemma}[Inverse Strichartz inequality, $L^2$-track]\label{refined l2 lemma 1}
Let $d\geq 3$ and $(f_n)_n\subset H^1(\R^d)$. Suppose that
\begin{align}
\lim_{n\to\infty}\|f_n\|_{H^1}=A<\infty\quad\text{and}\quad\lim_{n\to\infty}\|e^{it\Delta} f_n\|_{W_\tas(\R)}=\vare>0.
\end{align}
Then up to a subsequence, there exist $\phi\in L^2(\R^d)$ and $(t_n,x_n,\xi_n,\ld_n)_n\subset\R\times\R^d\times\R^d\times(0,\infty)$ such that $\limsup_{m\to\infty}|\xi_n|<\infty$ and $\lim_{n\to\infty}\ld_n=:\ld_\infty\in(0,\infty]$. Moreover,
\begin{align}\label{l2 refined strichartz basic 5}
&\,\,\ld_n^{\frac{d}{2}}e^{-i\xi_n\cdot(\ld_n x+x_n)}(e^{it_n\Delta}f_n)(\ld_n x+x_n)\nonumber\\
\nonumber\\
\rightharpoonup &\,\,\phi(x)\text{ weakly in }
\left\{
             \begin{array}{ll}
             H^1(\R^d),&\text{if $\limsup_{n\to\infty}|\ld_n\xi_n|<\infty$},\\
             L^2(\R^d),&\text{if $|\ld_n\xi_n|\to\infty$}.
             \end{array}
\right.
\end{align}
Addtionally, if $\limsup_{n\to\infty}|\ld_n\xi_n|<\infty$, then $\xi_n\equiv 0$. Setting
\begin{align}
\phi_n:=
\left\{
             \begin{array}{ll}
             \ld_n^{-\frac{d}{2}}e^{-it_n\Delta}\bg[\phi(\frac{x-x_n}{\ld_n})\bg],&\text{if $\ld_\infty<\infty$},\\
             \\
             \ld_n^{-\frac{d}{2}}e^{-it_n\Delta}\bg[e^{i\xi_n\cdot x}
             (P_{\leq\ld_n^\theta}\phi)(\frac{x-x_n}{\ld_n})\bg],&\text{if $\ld_\infty=\infty$}
             \end{array}
\right.
\end{align}
for some fixed $\theta\in(0,1)$, we have
\begin{align}
&\lim_{n\to\infty}(\|f_n\|^2_{2}-\|f_n-\phi_n\|^2_{2})=\|\phi\|^2_{2}\gtrsim A^2\bg(\frac{\vare}{A}\bg)^{2(d+1)(d+2)},
\label{l2 refined strichartz decomp 1}\\
&\lim_{n\to\infty}(\|f_n\|^2_{\dot{H}^1}-\|f_n-\phi_n\|^2_{\dot{H}^1}-\|\phi_n\|^2_{\dot{H}^1})=0,
\label{l2 refined strichartz decomp 3}\\
&\lim_{n\to\infty}(\|f_n\|^2_{2}-\|f_n-\phi_n\|^2_{2}-\|\phi_n\|^2_{2})=0.\label{l2 refined strichartz decomp 2}
\end{align}
\end{lemma}

\begin{proof}
For $R>0$, denote by $f^R$ the function such that $\mathcal{F}(f^R)=\chi_R\hat{f}$, where $\chi_R$ is the characteristic function of the ball $B_R(0)$. First we obtain that
\begin{align}\label{spare0}
\sup_{n\in\N}\|f_n-f_n^R\|_2^2
=\sup_{n\in\N}\int_{|\xi|\geq R}|\hat{f}_n(\xi)|^2\,d\xi\leq R^{-2}\sup_{n\in\N}\|f_n\|_{\dot{H}^1}^2\lesssim R^{-2}A^2\to 0
\end{align}
as $R\to\infty$. Combining with Strichartz, we infer that there exists some $K_1>0$ such that for all $R\geq K_1$ one has
\begin{align*}
\sup_{n\in\N}\|f^R_n\|_{2}\lesssim A\quad\text{and}\quad\sup_{n\in\N}\|e^{it\Delta} f^R_n\|_{W_\tas(\R)}\gtrsim \vare.
\end{align*}
Applying Lemma \ref{refined strichartz l2} to $(f^R_n)_n$, we know that there exists $(Q_n)_n\subset\mC$ such that
\begin{align}\label{l2 refined strichartz basic 1}
\vare^{d+2}A^{-(d+1)}\lesssim \inf_{n\in\N}|Q_n|^{\frac{d+2}{dq}-\frac{1}{2}}\|e^{it\Delta}(f^R_n)_{Q_n}\|_{L^q_{t,x}(\R)}.
\end{align}
Let $\ld_n^{-1}$ be the side-length of $Q_n$. Denote also by $\xi_n$ the center of $Q_n$. Since $q\in(\frac{2(d+2)}{d},\frac{2(d+2)}{d-2})$ for $d\geq 3$, using H\"older and Strichartz we obtain that
\begin{align*}
\sup_{n\in\N}\|e^{it\Delta}(f^R_n)_{Q_n}\|_{L^q_{t,x}(\R)}\lesssim \sup_{n\in\N}\|f_n\|_{H^1}\lesssim A.
\end{align*}
Combining with the fact that $\frac{d+2}{dq}-\frac{1}{2}<0$, we deduce that $\sup_{n\in\N}|Q_n|\lesssim 1$. Since $(\mathcal{F}(f_n^R))_n$ are supported in $B_R(0)$, we may assume that $(Q_n)_n\subset B_{R'}(0)$ for some sufficiently large $R'>0$. Therefore $(\ld_n)_n$ is bounded below and $(\xi_n)_n$ is bounded in $\R^d$. H\"older yields
\begin{align*}
|Q_n|^{\frac{d+2}{dq}-\frac{1}{2}}\|e^{it\Delta}(f^R_n)_{Q_n}\|_{L^q_{t,x}(\R)}
&\lesssim\ld_n^{\frac{d}{2}-\frac{d+2}{q}}\|e^{it\Delta}(f^R_n)_{Q_n}\|^{\frac{d(d+2)}{d^2+3d+1}}_{W_\tas(\R)}
\|e^{it\Delta}(f^R_n)_{Q_n}\|^{\frac{d+1}{d^2+3d+1}}_{L_{t,x}^\infty(\R)}
\nonumber\\
&\lesssim \ld_n^{\frac{d}{2}-\frac{d+2}{q}}\vare^{\frac{d(d+2)}{d^2+3d+1}}\|e^{it\Delta}(f^R_n)_{Q_n}\|^{\frac{d+1}{d^2+3d+1}}_{L_{t,x}^\infty(\R)}.
\end{align*}
Combining with \eqref{l2 refined strichartz basic 1} we infer that there exist $(t_n,x_n)_n\subset\R\times\R^d$ such that
\begin{align}\label{l2 refined strichartz basic 2}
\liminf_{n\to\infty}\ld_n^{\frac{d}{2}}|[e^{it_n\Delta}(f^R_n)_{Q_n}](x_n)|\gtrsim \vare^{(d+1)(d+2)}A^{-(d^2+3d+1)}.
\end{align}
Define
\begin{align*}
h_n(x)&:=\ld_n^{\frac{d}{2}}e^{-i\xi_n(\ld_n x+x_n)}(e^{it_n\Delta}f_n)(\ld_n x+x_n),\\
h^R_n(x)&:=\ld_n^{\frac{d}{2}}e^{-i\xi_n(\ld_n x+x_n)}(e^{it_n\Delta}f^R_n)(\ld_n x+x_n).
\end{align*}
It is easy to verify that $\|h_n\|_2=\|f_n\|_2$. By the $L^2$-boundedness of $(f_n)_n$ we know that there exists some $\phi\in L^2(\R^d)$ such that $h_n\rightharpoonup \phi$ weakly in $L^2(\R^d)$. Arguing similarly, we infer that $(h^R_n)_n$ converges weakly to some $\phi^R\in L^2(\R^d)$. By definition of $\phi$ and $\phi^R$ we see that
\begin{align*}
\|\phi-\phi^R\|^2_2=\lim_{n\to\infty}\la h_n-h_n^R,\phi-\phi^R\ra_{L^2}\leq (\limsup_{n\to\infty}\|h_n-h_n^R\|_2)\|\phi-\phi^R\|_2.
\end{align*}
Using \eqref{spare0} we then obtain that
\begin{align}\label{spare2}
\phi^R\to \phi\quad\text{in $L^2(\R^d)$ as $R\to\infty$}.
\end{align}
Now define the function $\chi$ such that $\hat{\chi}$ is the characteristic function of the cube $[-\frac{1}{2},\frac{1}{2})^d$. From \eqref{l2 refined strichartz basic 2}, the weak convergence of $h^R_n$ to $\phi^R$ in $L^2(\R^d)$ and change of variables it follows
\begin{align}\label{l2 refined strichartz basic 3}
\la\phi^R,\chi\ra=\lim_{n\to\infty}\ld_n^{\frac{d}{2}}|[e^{it_n\Delta}(f^R_n)_{Q_n}](x_n)|\gtrsim \vare^{(d+1)(d+2)}A^{-(d^2+3d+1)}.
\end{align}
On the other hand, using H\"older we also have
\begin{align*}
|\la\phi^R,\chi\ra|\leq\|\phi^R\|_2\|\chi\|_2.
\end{align*}
Thus
\begin{align}\label{spare3}
\|\phi^R\|_2^2\geq C\vare^{2(d+1)(d+2)}A^{-2(d^2+3d+1)}
\end{align}
for some $C=C(d)>0$ which is uniform for all $R\geq K_1$. Now using \eqref{spare2} and \eqref{spare3} we finally deduce that
\begin{align}
\|\phi\|_2^2 \geq\|\phi^R\|_2^2-\frac{C}{2}\vare^{2(d+1)(d+2)}A^{-2(d^2+3d+1)}
\geq \frac{C}{2}\vare^{2(d+1)(d+2)}A^{-2(d^2+3d+1)}
\end{align}
for sufficiently large $R$, which gives the lower bound of \eqref{l2 refined strichartz decomp 1}. From now on we fix $R$ such that the lower bound of \eqref{l2 refined strichartz decomp 1} is valid for this chosen $R$ and let $(t_n,x_n,\xi_n,\ld_n)_n$ be the corresponding symmetry parameters. Since $L^2(\R^d)$ is a Hilbert space, from the weak convergence of $h_n$ to $\phi$ in $L^2(\R^d)$ we obtain that
\begin{align*}
\lim_{n\to\infty}(\|h_n\|^2_{2}-\|\phi\|^2_{2}-\|h_n-\phi\|^2_{2})
=2\lim_{n\to\infty}\mathrm{Re}\,\la \phi,h_n-\phi\ra_{L^2}=0.
\end{align*}
Combining with the fact that
\begin{align*}
\|P_{\leq\ld_n^\theta}\phi-\phi\|_2\to 0\quad\text{as $n\to\infty$}
\end{align*}
for $\ld_n\to\infty$ we conclude the equalities of \eqref{l2 refined strichartz decomp 1} and \eqref{l2 refined strichartz decomp 2}. In the case $\limsup_{n\to\infty}|\ld_n\xi_n|<\infty$, using the boundedness of $(\ld_n\xi_n)_n$ and chain rule, we also infer that $\|h_n\|_{H^1}\lesssim \|f_n\|_{H^1}$. By the $H^1$-boundedness of $(f_n)_n$ and uniqueness of weak convergence we deduce additionally that $\phi\in H^1(\R^d)$ and \eqref{l2 refined strichartz basic 5} follows.

Next we show that we may assume $\xi_n\equiv 0$ under the additional condition $\limsup_{n\to\infty}|\ld_n\xi_n|<\infty$. Define
\begin{align*}
\mathcal{T}_{a,b}u(x):=be^{ia\cdot x}u(x)
\end{align*}
for $a\in\R^d$ and $b\in\C$ with $|b|=1$. Let also
\begin{align*}
(\ld\xi)_\infty&:=\lim_{n\to\infty}\ld_n\xi_n,\\
e^{i(\xi\cdot x)_\infty}&:=\lim_{n\to\infty}e^{i\xi_n\cdot x_n}.
\end{align*}
By the boundedness of $(\ld_n\xi_n)_n$ we infer that $\mathcal{T}_{\ld_n\xi_n,e^{i\xi_n\cdot x_n}}$ is an isometry on $L^2(\R^d)$ and converges strongly to $\mathcal{T}_{(\ld\xi)_\infty,e^{i(\xi\cdot x)_\infty}}$ as operators on $H^1(\R^d)$. We may replace $h_n$ by $\ld_n^{\frac{d}{2}}(e^{it_n\Delta}f_n)(\ld_n x+x_n)$ and $\phi$ by $\mathcal{T}_{(\ld\xi)_\infty,e^{i(\xi\cdot x)_\infty}}\phi$ and \eqref{l2 refined strichartz basic 5}, \eqref{l2 refined strichartz decomp 1} and \eqref{l2 refined strichartz decomp 3} carry over.

Finally, we prove \eqref{l2 refined strichartz decomp 3}. For the case $\ld_\infty<\infty$ we additionally know that $\phi\in H^1(\R^d)$ and $\xi_n\equiv 0$. Using the fact that $\dot{H}^1$ is a Hilbert space and change of variables we obtain that
\begin{align*}
o_n(1)=\|h_n\|_{\dot{H^1}}-\|h_n-\phi\|_{\dot{H^1}}-\|\phi\|_{\dot{H^1}}=\ld_n^2(\|f_n\|_{\dot{H^1}}-\|f_n-\phi_n\|_{\dot{H^1}}-\|\phi_n\|_{\dot{H^1}}).
\end{align*}
Combining with the lower boundedness of $(\ld_n)_n$, this implies that
\begin{align*}
\|f_n\|_{\dot{H^1}}-\|f_n-\phi_n\|_{\dot{H^1}}-\|\phi_n\|_{\dot{H^1}}=\ld_n^{-2}o_n(1)=o_n(1),
\end{align*}
which gives \eqref{l2 refined strichartz decomp 3} in the case $\ld_\infty<\infty$. Assume now $\ld_\infty=\infty$. Using change of variables and chain rule we obtain that
\begin{align}
&\,\|f_n\|^2_{\dot{H}^1}-\|f_n-\phi_n\|^2_{\dot{H}^1}-\|\phi_n\|^2_{\dot{H}^1}\nonumber\\
=&\,|\xi_n|^2\bg(\|h_n\|_2^2-\|h_n-P_{\leq \ld_n^\theta}\phi\|^2_{2}-\|P_{\leq \ld_n^\theta}\phi\|^2_{2}\bg)\nonumber\\
&+2\ld_n^{-1}\mathrm{Re}\bg(\la i\xi_n (h_n-P_{\leq \ld_n^\theta}\phi),\nabla P_{\leq \ld_n^\theta}\phi\ra
+\la i\xi_n P_{\leq \ld_n^\theta}\phi,\nabla (h_n-P_{\leq \ld_n^\theta}\phi)\ra\bg)\nonumber\\
&+\ld_n^{-2}\bg(\|h_n\|^2_{\dot{H}^1}-\|h_n-P_{\leq \ld_n^\theta}\phi\|^2_{\dot{H}^1}-\|P_{\leq \ld_n^\theta}\phi\|^2_{\dot{H}^1}\bg)\nonumber\\
=:&\,I_1+I_2+I_3.
\end{align}
Using the boundedness of $(\xi_n)_n$ and \eqref{l2 refined strichartz decomp 2} we infer that $I_1\to 0$. For $I_2$, using Bernstein and the boundedness of $(\xi_n)_n$ in $\R^d$ and of $(h_n-P_{\leq \ld_n^\theta}\phi)$ in $L^2(\R^d)$ we see that
\begin{align*}
|I_2|\lesssim \ld_n^{-1}\|h_n-P_{\leq \ld_n^\theta}\phi\|_2\|\nabla P_{\leq \ld_n^\theta}\phi\|_2\lesssim \ld_n^{-(1-\theta)}\to 0.
\end{align*}
Finally, $I_3$ can be similarly estimated using Bernstein inequality, we omit the details here. Summing up we conclude \eqref{l2 refined strichartz decomp 2}.
\end{proof}

\begin{lemma}
We have
\begin{align}
\text{(i) } &\ld_n\equiv 1\quad\text{or}\quad \ld_n\to \infty,\\
\text{(ii) } &t_n\equiv 0\quad\text{or}\quad \frac{t_n}{\ld_n^2}\to \pm\infty.
\end{align}
\end{lemma}

\begin{proof}
If $\ld_n\to\infty$, then there is nothing to prove. Otherwise assume that $\ld_\infty<\infty$. By the boundedness of $(\xi_n)_n$ we also know that $\phi\in H^1(\R^d)$ and $(\ld_n\xi_n)_n$ is bounded, thus $\xi_n\equiv 0$ and $h_n(x)$ reduces to $\ld_n^{\frac{d}{2}}(e^{it_n\Delta}f_n)(\ld_n x+x_n)$. Define
\begin{align*}
\mathcal{J}_{\ld}f(x):=\ld^{-\frac{d}{2}}f(\ld^{-1}x).
\end{align*}
Then $\mathcal{J}_{\ld_n}$ and $\mathcal{J}^{-1}_{\ld_n}$ converge strongly to $\mathcal{J}_{\ld_\infty}$ and $\mathcal{J}^{-1}_{\ld_\infty}$ strongly as operators in $H^1(\R^d)$. We may redefine $\ld_n\equiv 1$ and replace $\phi$ by $\mathcal{J}_{\ld_\infty}\phi$, and all the statements from Lemma \ref{refined l2 lemma 1} continue to hold.

We now prove (ii). If $\frac{t_n}{\ld_n^2}\to\pm\infty$, then we are done. Otherwise assume that $\frac{t_n}{\ld_n^2}\to\tau_\infty\in\R$. Recall that for $(\xi_0,x_0,\ld_0)\in\R^d\times\R^d\times(0,\infty)$ the operator $g_{\xi_0,x_0,\ld_0}$ is defined by
\begin{align*}
g_{\xi_0,x_0,\ld_0}f(x)=\ld_0^{-\frac{d}{2}}e^{i\xi_0\cdot x}f(\ld_0^{-1}(x-x_0)).
\end{align*}
Then
\begin{align*}
f_n=e^{-it_n\Delta}[g_{\xi_n,x_n,\ld_n}h_n](x)
\end{align*}
and
\begin{align*}
\phi_n=
\left\{
             \begin{array}{ll}
             e^{-it_n\Delta}[g_{\xi_n,x_n,\ld_n}\phi](x),&\text{if $\ld_\infty<\infty$},\\
             \\
             e^{-it_n\Delta}[g_{\xi_n,x_n,\ld_n}P_{\leq\ld_n^\theta}\phi](x),&\text{if $\ld_\infty=\infty$}.
             \end{array}
\right.
\end{align*}
Using the invariance of the NLS-flow under the Galilean transformation we infer that
\begin{align}\label{Galilean invariance}
e^{-it_n\Delta}[g_{\xi_n,x_n,\ld_n}f](x)=g_{\xi_n,x_n-2t_n\xi_n,\ld_n}[e^{it_n|\xi_n|^2}e^{-i\frac{t_n}{\ld_n^2}\Delta}f](x).
\end{align}
Denote $\beta:=\lim_{n\to\infty}e^{it_n|\xi_n|^2}$. We can therefore redefine $t_n$ by $0$, $x_n$ by $x_n-2t_n\xi_n$ and $\phi$ by $\beta e^{-i\tau_\infty\Delta}\phi$. One easily checks that up to \eqref{l2 refined strichartz decomp 3} in the case $\ld_\infty=\infty$, the statements from Lemma \ref{refined l2 lemma 1} carry over, due to the strong continuity of the linear Schr\"odinger flow on $H^1(\R^d)$ and the fact that $g$ is an isometry on $L^2(\R^d)$. To see \eqref{l2 refined strichartz decomp 3} in the case $\ld_\infty=\infty$, we obtain that
\begin{align}
&\|g_{\xi_n,x_n-2t_n\xi_n,\ld_n}[e^{it_n|\xi_n|^2}e^{-i\frac{t_n}{\ld_n^2}\Delta}P_{\leq \ld_n^\theta}\phi]
-g_{\xi_n,x_n-2t_n\xi_n,\ld_n}[\beta e^{-i\tau_\infty\Delta}P_{\leq \ld_n^\theta}\phi]\|_{\dot{H}^1}\nonumber\\
\lesssim &\,|\xi_n|\|e^{it_n|\xi_n|^2}e^{-i\frac{t_n}{\ld_n^2}\Delta}P_{\leq \ld_n^\theta}\phi-\beta e^{-i\tau_\infty\Delta}P_{\leq \ld_n^\theta}\phi\|_2\nonumber\\
&+\ld_n^{-1}\|e^{it_n|\xi_n|^2}e^{-i\frac{t_n}{\ld_n^2}\Delta}P_{\leq \ld_n^\theta}\phi-\beta e^{-i\tau_\infty\Delta}P_{\leq \ld_n^\theta}\phi\|_{\dot{H}^1}
=:I_1+I_2.
\end{align}
By the boundedness of $(\xi_n)_n$ one easily verifies that $I_1\to 0$. Using Bernstein we see that
\begin{align}
|I_2|\lesssim \ld_n^{-(1-\theta)}\|P_{\leq \ld_n^\theta}\phi\|_2\lesssim \ld^{-(1-\theta)}_n \|\phi\|_2\to 0.
\end{align}
This completes the desired proof.
\end{proof}

\begin{remark}
Using \eqref{Galilean invariance}, redefining the parameters and taking Lemma \ref{track h1} into account we may assume that
\begin{align*}
\phi_n=
\left\{
             \begin{array}{ll}
             \ld_ng_{0,x_n,\ld_n}[e^{it_n\Delta}P_{>\ld_n^\theta}\phi](x),&\text{if $\ld_\infty=0$},\\
             \\
             e^{it_n\Delta}\phi(x-x_n),&\text{if $\ld_\infty=1$},\\
             \\
             g_{\xi_n,x_n,\ld_n}[e^{it_n\Delta}P_{\leq\ld_n^\theta}\phi](x),&\text{if $\ld_\infty=\infty$}.
             \end{array}
\right.
\end{align*}
\end{remark}

\begin{lemma}
We have
\begin{align}
&\|f_n\|_\tas^\tas=\|\phi_n\|_\tas^\tas+\|f_n-\phi_n\|_\tas^\tas+o_n(1),\label{decomp tas}\\
&\|f_n\|_\tbs^\tbs=\|\phi_n\|_\tbs^\tbs+\|f_n-\phi_n\|_\tbs^\tbs+o_n(1).\label{decomp tbs}
\end{align}
\end{lemma}

\begin{proof}
Assume first that $\ld_\infty=\infty$. Using Bernstein and Sobolev we infer that
\begin{align*}
\|\phi_n\|_\tbs= \ld^{-1}_n\|P_{\leq\ld_n^\theta}\phi\|_{\dot{H}^1}
\lesssim  \ld^{-(1-\theta)}_n \|\phi\|_2\to 0.
\end{align*}
Hence $\|\phi_n\|_\tbs=o_n(1)$. Therefore by triangular inequality
\begin{align*}
\bg|\|f_n\|_\tbs-\|f_n-\phi_n\|_\tbs\bg|\leq \|\phi_n\|_\tbs\to 0
\end{align*}
and \eqref{decomp tbs} follows. Now suppose that $\ld_\infty=1$ and $t_n\to\pm\infty$. For $\beta>0$ let $\psi\in\mathcal{S}(\R^d)$ such that
\begin{align*}
 \|\phi-\psi\|_{H^1}\leq\beta.
\end{align*}
Define
\begin{align*}
\psi_n:=e^{it_n\Delta}\psi(x-x_n).
\end{align*}
Then by dispersive estimate we deduce that
\begin{align*}
\|\psi_n\|_{\tbs}\lesssim |t_n|^{-1}\|\psi\|_{(\tbs)'}\to 0.
\end{align*}
On the other hand, by Sobolev we have
\begin{align*}
\|\psi_n-\phi_n\|_\tbs\lesssim \|\psi-\phi\|_{\dot{H}^1}\leq \beta.
\end{align*}
Hence $\|\psi_n\|_\tbs\lesssim \beta$ for all sufficiently large $n$. Therefore by triangular inequality
\begin{align*}
\bg|\|f_n\|_\tbs-\|f_n-\psi_n\|_\tbs\bg|\lesssim \beta
\end{align*}
and \eqref{decomp tbs} follows by taking $\beta$ arbitrarily small. Now we assume $\ld_\infty=1$ and $t_n\equiv 0$. Then we additionally know that $\phi\in H^1(\R^d)$ and $h_n\rightharpoonup \phi$ in $H^1(\R^d)$. Using the Brezis-Lieb lemma we deduce that
\begin{align*}
\|h_n\|_\tbs^\tbs=\|\phi\|_\tbs^\tbs+\|h_n-\phi\|_\tbs^\tbs+o_n(1).
\end{align*}
Undoing the transformation we obtain \eqref{decomp tbs}.

We now consider \eqref{decomp tas}. When $\ld_\infty=\infty$ or $\ld_\infty=1$ and $t_n\to\pm\infty$, then $\|\psi_n\|_\tbs\to 0$, and by H\"older we will also have $\|\psi_n\|_\tas\to 0$, thus \eqref{decomp tas} follows. For the case $\ld_\infty=1$ and $t_n\equiv 0$, \eqref{decomp tas} follows again from the Brezis-Lieb lemma. This completes the desired proof.
\end{proof}

Having all the preliminaries we are in the position to establish the double track profile decomposition.

\begin{lemma}[Double track profile decomposition]\label{linear profile}
Let $(\psi_n)_n$ be a bounded sequence in $H^1(\R^d)$. Then up to a subsequence, there exist nonzero linear profiles $(\tdu^j)_j\subset \dot{H}^1(\R^d)\cup L^2(\R^d)$, remainders $(w_n^k)_{k,n}\subset H^1(\R^d)$, parameters $(t^j_n,x^j_n,\xi^j_n,\ld^j_n)_{j,n}\subset\R\times\R^d\times\R^d\times(0,\infty)$ and $K^*\in\N\cup\{\infty\}$, such that
\begin{itemize}
\item[(i)] For any finite $1\leq j\leq K^*$ the parameters satisfy
\begin{align}
1&\gtrsim_j\lim_{n\to\infty}|\xi_n^j|,\nonumber\\
\lim_{n\to\infty}\ld^j_n&=:t_\infty^j\in\{0,\pm\infty\},\nonumber\\
\lim_{n\to\infty}\ld^j_n&=:\ld_\infty^j\in\{0,1,\infty\},\nonumber\\
t_n^j&\equiv 0\quad\text{if $t_\infty^j=0$},\nonumber\\
\ld_n^j&\equiv 1\quad\text{if $\ld_\infty^j=1$},\nonumber\\
\xi_n^j&\equiv 0\quad\text{if $\ld_\infty^j\in\{0,1\}$}.
\end{align}

\item[(ii)]For any finite $1\leq k\leq K^*$ we have the decomposition
\begin{align}\label{decomp}
\psi_n=\sum_{j=1}^k T^j_n P_n^j\tdu^j+w_n^k.
\end{align}
Here, the operators $T_n^j$ and $P_n^j$ are defined by
\begin{align}
T^j_n u(x):=
\left\{
             \begin{array}{ll}
             \ld^j_n g_{0,x^j_n,\ld^j_n}[e^{it^j_n\Delta}u](x),&\text{if $\ld^j_\infty=0$},\\
             \\ \,
             \,[e^{it^j_n\Delta}u](x-x^j_n),&\text{if $\ld^j_\infty=1$},\\
             \\
             g_{\xi^j_n,x^j_n,\ld^j_n}[e^{it^j_n\Delta}u](x),&\text{if $\ld^j_\infty=\infty$}
             \end{array}
\right.
\end{align}
and
\begin{align}
P^j_n u:=
\left\{
             \begin{array}{ll}
             P_{>(\ld_n^j)^\theta}u,&\text{if $\ld^j_\infty=0$},\\
             \\
             u,&\text{if $\ld^j_\infty=1$},\\
             \\
             P_{\leq(\ld_n^j)^\theta}u,&\text{if $\ld^j_\infty=\infty$}
             \end{array}
\right.
\end{align}
for some $\theta\in(0,1)$. Moreover,
\begin{align}
\tdu^j\in
\left\{
             \begin{array}{ll}
             \dot{H}^1(\R^d),&\text{if $\ld^j_\infty=0$},\\
             \\
             H^1(\R^d),&\text{if $\ld^j_\infty=1$},\\
             \\
             L^2(\R^d),&\text{if $\ld^j_\infty=\infty$}.
             \end{array}
\right.
\end{align}

\item[(iii)] The remainders $(w_n^k)_{k,n}$ satisfy
\begin{align}\label{to zero wnk}
\lim_{k\to K^*}\lim_{n\to\infty}\|e^{it\Delta}w_n^k\|_{W_\tas\cap W_\tbs(\R)}=0.
\end{align}

\item[(iv)] The parameters are orthogonal in the sense that
\begin{align}\label{orthog of pairs}
 \frac{\ld_n^k}{\ld_n^j}+ \frac{\ld_n^j}{\ld_n^k}+\ld_n^k|\xi_n^j-\xi_n^k|+\bg|t_k\bg(\frac{\ld_n^k}{\ld_n^j}\bg)^2-t_n^j\bg|
+\bg|\frac{x_n^j-x_n^k-2t_n^k(\ld_n^k)^2(\xi_n^j-\xi_n^k)}{\ld_n^k}\bg|\to\infty
\end{align}
for any $j\neq k$.

\item[(v)] For any finite $1\leq k\leq K^*$ we have the energy decompositions
\begin{align}\label{orthog L2 and H1}
\||\nabla|^s\psi_n\|_2^2&=\sum_{j=1}^k\||\nabla|^sT_n^jP_n^j\tdu^j\|_2^2+\||\nabla|^s w_n^k\|_2^2+o_n(1),\\
\mH(\psi_n)&=\sum_{j=1}^k\mH(T_n^jP_n^j\tdu^j)+\mH(w_n^k)+o_n(1),\label{conv of h}\\
\mK(\psi_n)&=\sum_{j=1}^k\mK(T_n^jP_n^j\tdu^j)+\mK(w_n^k)+o_n(1),\label{conv of k}\\
\mI(\psi_n)&=\sum_{j=1}^k\mI(T_n^jP_n^j\tdu^j)+\mI(w_n^k)+o_n(1)\label{conv of i}
\end{align}
with $s\in\{0,1\}$.
\end{itemize}
\end{lemma}

\begin{proof}
We construct the linear profiles iteratively and start with $k=0$ and $w_n^0:=\psi_n$. We assume initially that the linear profile decomposition is given and its claimed properties are satisfied for some $k$. Define
\begin{align*}
\vare_{k}:=\lim_{n\to\infty}\|e^{it\Delta}w_n^k\|_{W_\tas\cap W_\tbs(\R)}.
\end{align*}
If $\vare_k=0$, then we stop and set $K^*=k$. Otherwise we have
\begin{align}
\text{(i) } &\limsup_{n\to\infty}\|e^{it\Delta}w_n^k\|_{W_\tas(\R)}\geq \limsup_{n\to\infty}\|e^{it\Delta}w_n^k\|_{W_\tbs(\R)},\text{ or}\nonumber\\
\text{(ii) } &\limsup_{n\to\infty}\|e^{it\Delta}w_n^k\|_{W_\tas(\R)}< \limsup_{n\to\infty}\|e^{it\Delta}w_n^k\|_{W_\tbs(\R)}.\label{decision}
\end{align}
For the first situation we apply the $L^2$-decomposition to $w_n^k$, while for the latter case we apply the $\dot{H}^1$-decomposition. In both cases we obtain the sequence $(\tdu^{k+1},w_n^{k+1},t_n^{k+1},x_n^{k+1},\xi_n^{k+1},\ld_n^{k+1})_{n}.$ We should still need to check that the items (iii) and (iv) are satisfied for $k+1$. That the other items are also satisfied for $k+1$ follows directly from the construction of the linear profile decomposition. If $\vare_k=0$, then item (iii) is automatic; otherwise we have $K^*=\infty$ and $\vare_j>0$ for all $j\in\N\cup\{0\}$. Let $S_1\subset\N$ denote the set of indices such that for each $j\in S_1$, we apply the $\dot{H}^1$-profile decomposition at the $j-1$-step. Also define $S_2:=\N\setminus S_1$. Using \eqref{refined strichartz decomp 1}, \eqref{l2 refined strichartz decomp 1} and \eqref{orthog L2 and H1} we obtain that
\begin{align}
&\sum_{j\in S_1}A_{j-1}^2\bg(\frac{\vare_{j-1}}{A_{j-1}}\bg)^{\frac{d(d+2)}{4}}+\sum_{j\in S_2}A_{j-1}^2\bg(\frac{\vare_{j-1}}{A_{j-1}}\bg)^{2(d+1)(d+2)}\nonumber\\
\lesssim &\sum_{j\in S_1}\|\tdu^j\|^2_{\dot{H}^1}+\sum_{j\in S_2}\|\tdu^j\|^2_{2}
=\sum_{j\in S_1}\lim_{n\to\infty}\|T_jP_n^j\phi^j\|^2_{\dot{H}^1}+\sum_{j\in S_2}\lim_{n\to\infty}\|T_jP_n^j\phi^j\|^2_{2}\nonumber\\
\leq& \lim_{n\to\infty}\|\psi_n\|^2_{H^1}= A_0^2,
\end{align}
where $A_j:=\lim_{n\to\infty}\|w_n^j\|_{H^1}$. By \eqref{orthog L2 and H1} we know that $(A_j)_j$ is monotone decreasing, thus also bounded. Since $S_1\cup S_2=\N$, at least one of both is an infinite set. Suppose that $|S_1|=\infty$. Then
\begin{align*}
A_j^2\bg(\frac{\vare_j}{A_j}\bg)^{\frac{d(d+2)}{4}}\to 0\quad\text{as $j\to\infty$}.
\end{align*}
Combining with the boundedness of $(A_j)_j$ we immediately conclude that $\vare_i\to 0$. The same also holds for the case $|S_2|=\infty$ and the proof of item (iii) is complete. Finally we show item (iv). Denote
\begin{align*}
g_n^j:=
\left\{
             \begin{array}{ll}
             \ld^j_n g_{0,x^j_n,\ld^j_n},&\text{if $\ld^j_\infty=0$},\\
             \\
             g_{\xi^j_n,x^j_n,\ld^j_n},&\text{if $\ld^j_\infty\in\{1,\infty\}$}.
             \end{array}
\right.
\end{align*}
Assume that item (iv) does not hold for some $j<k$. By the construction of the profile decomposition we have
\begin{align*}
w_n^{k-1}=w_n^j-\sum_{l=j+1}^{k-1}g_n^l e^{-it_n^l}P_n^l \tdu^l.
\end{align*}
Then by definition of $\tdu^k$ we know that
\begin{align}
\tdu^k&=\wlim_{n\to\infty}e^{-it_n^k\Delta}[(g_n^k)^{-1}w_n^{k-1}]\nonumber\\
&=\wlim_{n\to\infty}e^{-it_n^k\Delta}[(g_n^j)^{-1}w_n^{j}]-\sum_{l=j+1}^{k-1}\wlim_{n\to\infty}e^{-it_n^k\Delta}[(g_n^k)^{-1}P_n^l \tdu^l],
\end{align}
where the weak limits are taken in the $\dot{H}^1$- or $L^2$-topology, depending on the bifurcation \eqref{decision}. Our aim is to show that $\tdu^k$ is zero, which leads to a contradiction and proves item (iv). We first consider the case $\ld_\infty^k=\infty$. Then the weak limit is taken w.r.t. the $L^2$-topology. For the first summand, we obtain that
\begin{align*}
e^{-it_n^k\Delta}[(g_n^k)^{-1}w_n^{j}]=(e^{-it_n^k\Delta}(g_n^k)^{-1}g_n^je^{it_n^j\Delta})[e^{-it_n^j\Delta}(g_n^j)^{-1}w_n^j].
\end{align*}
Direct calculation yields
\begin{align}\label{composite of g's}
&e^{-it_n^k\Delta}(g_n^k)^{-1}g_n^je^{it_n^j\Delta}\nonumber\\
=&\,\beta_{n}^{j,k}g_{\ld_n^k(\xi_n^j-\xi_n^k),\frac{x_n^j-x_n^k-2t_n^k(\ld_n^k)^2(\xi_n^j-\xi_n^k)}{\ld_n^k},\frac{\ld_n^j}{\ld_n^k}}
e^{-i\bg(t_n^k\bg(\frac{\ld_n^k}{\ld_n^j}\bg)^2-t_n^j\bg)\Delta}.
\end{align}
with $\beta_{n}^{j,k}=e^{i(\xi_n^j-\xi_n^k)x_n^k+t_n^k(\ld_n^k)^2|\xi_n^j-\xi_n^k|^2}$. Therefore, the failure of item (iv) will lead to the strong convergence of the adjoint of $e^{-it_n^k\Delta}(g_n^k)^{-1}g_n^je^{it_n^j\Delta}$ in $L^2(\R^d)$. On the other hand, we must have $\ld_\infty^j=\infty$, otherwise item (iv) would be satisfied. By construction of the profile decomposition we have
\begin{align*}
e^{-it_n^j\Delta}(g_n^j)^{-1}w_n^j\rightharpoonup 0\quad\text{in $L^2(\R^d)$}
\end{align*}
and we conclude that the first summand weakly converges to zero in $L^2(\R^d)$. Now we consider the single terms in the second summand. We can rewrite each single summand to
\begin{align*}
e^{-it_n^k\Delta}[(g_n^k)^{-1}P_n^l \tdu^l]=(e^{-it_n^k\Delta}(g_n^k)^{-1}g_n^je^{it_n^j\Delta})[e^{-it_n^j\Delta}(g_n^j)^{-1}P_n^l \tdu^l].
\end{align*}
By the previous arguments it suffices to show that
\begin{align*}
e^{-it_n^j\Delta}(g_n^j)^{-1}P_n^l \tdu^l\rightharpoonup 0\quad\text{in $L^2(\R^d)$}.
\end{align*}
Assume first $\ld_\infty^l=0$. In this case, we can in fact show that
\begin{align}
e^{-it_n^j\Delta}(g_n^j)^{-1}P_n^l \tdu^l\to 0\quad\text{in $L^2(\R^d)$}.
\end{align}
Indeed, using Bernstein we have
\begin{align*}
\|e^{-it_n^j\Delta}(g_n^j)^{-1}P_n^l \tdu^l\|_2=\ld_n^l\|P_{>(\ld_n^l)^\theta}\tdu^l\|_2\lesssim (\ld_n^l)^{1-\theta}\|\tdu^l\|_{\dot{H}^1}
\to 0.
\end{align*}
Next we consider the cases $\ld_\infty^{l}\in\{1,\infty\}$. By the construction of the decomposition and the inductive hypothesis we know that $\tdu^l\in L^2(\R^d)$ and item (iv) is satisfied for the pair $(j,l)$. Using the fact that
\begin{align*}
\|P_{\leq(\ld_n^l)^\theta}\tdu^l-\tdu^l\|_2\to 0\quad\text{when $\ld_n^l\to \infty$}
\end{align*}
and density arguments, it suffices to show that
\begin{align*}
I_n:=e^{-it_n^j\Delta}(g_n^j)^{-1}g_n^l e^{it_n^l\Delta} \tdu\rightharpoonup 0\quad\text{in $L^2(\R^d)$}
\end{align*}
for arbitrary $\tdu\in C_c^\infty(\R^d)$. Using \eqref{composite of g's} we obtain that
\begin{align*}
I_n=\beta_n^{j,l}g_{\ld_n^l(\xi_n^j-\xi_n^l),\frac{x_n^j-x_n^l-2t_n^l(\ld_n^l)^2(\xi_n^j-\xi_n^l)}{\ld_n^l},\frac{\ld_n^j}{\ld_n^l}}
e^{-i\bg(t^l_n\bg(\frac{\ld_n^l}{\ld_n^j}\bg)^2-t_n^j\bg)\Delta}\tdu.
\end{align*}
Assume first that $\lim_{n\to\infty}\frac{\ld_n^j}{\ld_n^l}+\frac{\ld_n^l}{\ld_n^j}=\infty$. Then for any $\psi\in C_c^\infty(\R^d)$ we have
\begin{align*}
|\la I_n,\psi\ra|\leq \min\bg\{\bg(\frac{\ld_n^j}{\ld_n^l}\bg)^{\frac{d}{2}}\|\tdu\|_1\|\psi\|_\infty,\,
\bg(\frac{\ld_n^j}{\ld_n^l}\bg)^{-\frac{d}{2}}\|\psi\|_1\|\tdu\|_\infty\bg\}\to 0.
\end{align*}
So we may assume that $\lim_{n\to\infty}\frac{\ld_n^j}{\ld_n^l}\in(0,\infty)$. Suppose now $t^l_n\bg(\frac{\ld_n^l}{\ld_n^j}\bg)^2-t_n^j\to\pm\infty$. Then the weak convergence of $I_n$ to zero in $L^2(\R^d)$ follows immediately from the dispersive estimate. Hence we may also assume that $\lim_{n\to\infty}t^l_n\bg(\frac{\ld_n^l}{\ld_n^j}\bg)^2-t_n^j\in\R$. Finally, it is left with the options
\begin{align*}
|\ld_n^l(\xi_n^j-\xi_n^l)|\to\infty\quad\text{or}\quad\bg|\frac{x_n^j-x_n^l-2t_n^l(\ld_n^l)^2(\xi_n^j-\xi_n^l)}{\ld_n^l}\bg|\to\infty.
\end{align*}
For the latter case, we utilize the fact that the symmetry group composing by unbounded translations weakly converges to zero as operators in $L^2(\R^d)$ to deduce the claim; For the former case, we can use the same arguments as the ones for the translation symmetry by considering the Fourier transformation of $I_n$ in the frequency space. This completes the desired proof for the case $\ld_n^k=\infty$. It remains to show the claim for the cases $\ld_\infty^k\in\{0,1\}$. We only need to prove that for $\ld_\infty^l=\infty$, we must have
\begin{align}
e^{-it_n^j\Delta}(g_n^j)^{-1}g_n^le^{it_n^l\Delta}P_{\leq(\ld_n^l)^\theta} \tdu^l\to 0\quad\text{in $\dot{H}^1(\R^d)$},
\end{align}
the other cases can be dealt similarly as by the case $\ld_\infty^k=\infty$ (or alternatively, one can consult \cite[Thm. 7.5]{killip_visan_soliton} for full details). Notice in this case, $e^{-it_n^j\Delta}(g_n^j)^{-1}$ is an isometry on $\dot{H}^1$. Using Bernstein, the boundedness of $(\xi_n^l)_n$ and chain rule we obtain that
\begin{align*}
&\,\|e^{-it_n^j\Delta}(g_n^j)^{-1}g_n^le^{it_n^l\Delta}P_{\leq(\ld_n^l)^\theta}\tdu^l\|_{\dot{H}^1}\nonumber\\
\lesssim &\,(\ld_n^l)^{-1}|\xi_n^l|\|P_{\leq(\ld_n^l)^\theta}\tdu^l\|_2+(\ld_n^l)^{-1}\|P_{\leq(\ld_n^l)^\theta}\tdu^l\|_{\dot{H}^1}\nonumber\\
\lesssim &\,(\ld_n^l)^{-1}\|\tdu^l\|_2+(\ld_n^l)^{-(1-\theta)}\|\tdu^l\|_{2}\to 0.
\end{align*}
This finally completes the proof of item (iv).
\end{proof}

\section{Scattering threshold for the focusing-focusing \eqref{NLS double crit}}\label{section f f}
Throughout this section we restrict ourselves to the focusing-focusing \eqref{NLS double crit}
\begin{align}\label{NLS}
i\pt_t u+\Delta u+|u|^{\tas-2}u+|u|^{\tbs-2}u=0
\end{align}
We also define the set $\mA$ by
\begin{align*}
\mA:=\{u\in H^1(\R^d):\mM(u)<\mM(Q),\,\mH(u)<m_{\mM(u)},\,\mK(u)>0 \}.
\end{align*}
\subsection{Variational estimates and MEI-functional}\label{variational arguments ff}
We derive below the necessary variational estimates which will be later used in Section \ref{Existence of the minimal blow-up solution} and Section \ref{Extinction of the minimal blow-up solution}. Particularly, we give the precise construction of the MEI-functional $\mD$, which will help us to set up the inductive hypothesis given in Section \ref{Existence of the minimal blow-up solution}.
\begin{lemma}\label{positive or negative k}
Let $u\in H^1(\R^d)\setminus\{0\}$ with $\mM(u)<{\mM(Q)}$. Then there exists a unique $\ld(u)>0$ such that
\begin{equation*}
\mK(T_\ld u)\left\{
             \begin{array}{ll}
             >0, &\text{if $\ld\in(0,\ld(u))$},  \\
             =0,&\text{if $\ld=\ld(u)$},\\
             <0,&\text{if $\ld\in(\ld(u),\infty)$}.
             \end{array}
\right.
\end{equation*}
\end{lemma}

\begin{proof}
We first obtain that
\begin{align*}
\mK(T_\ld u)&=\ld^2 \bg(\|\nabla u\|_2^2-\frac{d}{d+2} \|u\|_\tas^\tas\bg)-\ld^{\tbs}\|u\|_\tbs^\tbs,\nonumber\\
\frac{d}{d\ld}\mK(T_\ld u)&=2\ld\bg(\|\nabla u\|_2^2-\frac{d}{d+2}\|u\|_{\tas}^\tas\bg)-\tbs\ld^{\tbs-1}\|u\|_\tbs^\tbs\nonumber\\
\end{align*}
with
\begin{align}\label{GN-L2}
\|\nabla u\|_2^2-\frac{d}{d+2}\|u\|_{\tas}^\tas\geq \bg(1-\bg(\frac{\mM(u)}{{\mM(Q)}}\bg)^{\frac{2}{d}}\bg)\|\nabla u\|_2^2>0.
\end{align}
Since $\tbs>2$, $\frac{d}{d\ld}\mK(T_\ld u)$ has a unique zero $\beta(u)\in(0,\infty)$ which is the global maxima of $\mK(T_\ld u)$. Also, $\mK(T_\ld u)$ is increasing on $(0,\beta(u))$ and decreasing on $(\beta(u),\infty)$. One easily verifies that $\mK(T_\ld u)$ is positive on $(0,\beta(u))$ and $\mK(T_\ld u)\to-\infty$ as $\ld\to\infty$. Consequently, $\mK(T_\ld u)$ has a first and unique zero $\ld(u)\in(\beta(u),\infty)$ and $\mK(T_\ld u)$ is positive on $(0,\ld(u))$ and negative on $(\ld(u,\infty))$. This completes the proof.
\end{proof}

\begin{lemma}\label{pos of k implies pos of h}
Assume that $\mK(u)\geq 0$. Then $\mH(u)\geq 0$. If additionally $\mK(u)> 0$, then also $\mH(u)> 0$.
\end{lemma}
\begin{proof}
We have
\begin{align}
\mH(u)\geq \mH(u)-\frac{1}{2}\mK(u)=\frac{1}{d}\|u\|_{\tbs}^\tbs\geq 0.\label{not strict}
\end{align}
It is trivial that \eqref{not strict} becomes strict when $u\neq 0$, which is the case when $\mK(u)>0$.
\end{proof}

\begin{lemma}\label{bound of gradient by energy}
Let $u\in \mA$. Suppose also that $\mM(u)\leq (1-\delta)^{\frac{d}{2}}\mM(Q)$ with some $\delta\in(0,1)$. Then
\begin{align}
\|u\|_\tbs^\tbs&\leq \|\nabla u\|_2^2,\label{bound of tbs by 2}\\
\frac{\delta}{d} \|\nabla u\|_2^2&\leq \mH(u)\leq \frac{1}{2}\|\nabla u\|_2^2.\label{third}
\end{align}
\end{lemma}

\begin{proof}
\eqref{bound of tbs by 2} follows immediately from the fact that $\mK(u)\geq0$ for $u\in{\mA}$ and the non-positivity of the nonlinear potentials. 
The first $\leq$ in \eqref{third} follows from
\begin{align*}
\mH(u)&\geq \mH(u)-\frac{1}{\tbs}\mK(u)\nonumber\\
&=\frac{1}{d}(\|\nabla u\|_2^2-\frac{d}{d+2}\|u\|_\tas^\tas)\nonumber\\
&\geq \frac{1}{d}\bg(1-\bg(\frac{\mM(u)}{{\mM(Q)}}\bg)^{\frac{2}{d}}\bg)\|\nabla u\|_2^2\geq \frac{\delta}{d}\|\nabla u\|_2^2
\end{align*}
and the second $\leq$ follows immediately from the non-positivity of the power potentials.
\end{proof}

\begin{lemma}\label{monotone lemma}
The mapping $c\mapsto m_c$ is continuous and monotone decreasing on $(0,{\mM(Q)})$.
\end{lemma}

\begin{proof}
The proof follows the arguments of \cite{Bellazzini2013}, where we also need to take the mass constraint into account. We first show that the function $f$ defined by
\begin{align*}
f(a,b):=\max_{t>0}\{at^2-bt^{\tbs}\}
\end{align*}
is continuous on $(0,\infty)^2$. In fact, the global maxima can be calculated explicitly. Let
\begin{align*}
g(t,a,b):=at^2-bt^{\tbs}.
\end{align*}
and let $t^*\in (0,\infty)$ such that $\pt_t g(t^*,a,b)=0$. Then $t^*=\bg(\frac{2a}{\tbs b}\bg)^{\frac{d-2}{4}}$. Particularly, $\pt_t g(t,a,b)$ is positive on $(0,t^*)$ and negative on $(t^*,\infty)$. Thus
$$ f(a,b)=g(t^*,a,b)=\bg(\frac{2a}{\tbs b}\bg)^{\frac{d-2}{2}}\frac{2a}{d}$$
and we conclude the continuity of $f$ on $(0,\infty)^2$.

We now show the monotonicity of $c\mapsto m_c$. It suffices to show that for any $0<c_1<c_2<{\mM(Q)}$ and $\vare>0$ we have
\begin{align*}
m_{c_2}\leq m_{c_1}+\vare.
\end{align*}
Define the set $V(c)$ by
\begin{align*}
V(c):=\{u\in H^1(\R^d):\mM(u)=c,\,\mK(u)=0\}.
\end{align*}
By the definition of $m_{c_1}$ there exists some $u_1\in V(c_1)$ such that
\begin{align}\label{pert 2}
\mH(u_1)\leq m_{c_1}+\frac{\vare}{2}.
\end{align}
Let $\eta\in C^{\infty}_c(\R^d)$ be a cut-off function such that $\eta=1$ for $|x|\leq 1$, $\eta=0$ for $|x|\geq 2$ and $\eta\in [0,1]$ for $|x|\in(1,2)$. For $\delta>0$, define
\begin{equation*}
\tilde{u}_{1,\delta}(x):= \eta(\delta x)\cdot u_1(x).
\end{equation*}
Then $\tilde{u}_{1,\delta}\to u_1$ in $H^1(\R^d)$ as $\delta\to 0$. Therefore,
\begin{align*}
\|\nabla\tilde{u}_{1,\delta}\|_2&\to \|\nabla u_1\|_2, \\
\|\tilde{u}_{1,\delta}\|_p&\to \| u_1\|_p
\end{align*}
for all $p\in[2,\tbs]$ as $\delta\to 0$. Using Gagliardo-Nirenberg we know that $\frac{1}{2}\|\nabla v\|_2^2>\frac{1}{\tas}\|v\|_\tas^\tas$ for all $v\in H^1(\R^d)$ with $\mM(v)<{\mM(Q)}$. Since $c_1\in(0,{\mM(Q)})$, we infer that $\mM(\tilde{u}_{1,\delta})\in (0,{\mM(Q)})$ for sufficiently small $\delta$. Combining with the continuity of $f$ we conclude that
\begin{align}\label{pert 1}
\max_{t>0}\mH(T_t\tilde{u}_{1,\delta})&=\max_{t>0}\{t^2(\frac{1}{2}\|\nabla\tilde{u}_{1,\delta}\|_2^2-\frac{1}{\tas}\|\tilde{u}_{1,\delta}\|_\tas^\tas)
-\frac{t^\tbs}{\tbs}\|\tilde{u}_{1,\delta}\|_\tbs^\tbs\}\nonumber\\
&\leq \max_{t>0}\{t^2(\frac{1}{2}\|\nabla u_1\|_2^2-\frac{1}{\tas}\|u_1\|_\tas^\tas)
-\frac{t^\tbs}{\tbs}\|u_1\|_\tbs^\tbs\}+\frac{\vare}{4}\nonumber\\
&=\max_{t>0}\mH(T_t u_1)+\frac{\vare}{4}
\end{align}
for sufficiently small $\delta>0$. Now let $v\in C_c^\infty(\R^d)$ with $\mathrm{supp}\,v\subset\R^d\backslash B(0,2\delta^{-1})$ and define
\begin{equation*}
v_0:= \frac{(c_2-\mM(\tilde{u}_{1,\delta}))^{\frac{1}{2}}}{(\mM(v))^{\frac{1}{2}}}\,v.
\end{equation*}
We have $\mM(v_0)=c_2-\mM(\tilde{u}_{1,\delta})$. Define
\begin{align*}
w_\ld:=\tilde{u}_{1,\delta}+T_\ld v_0
\end{align*}
with some to be determined $\ld>0$. For sufficiently small $\delta$ the supports of $\tilde{u}_{1,\delta}$ and $v_0$ are disjoint, thus\footnote{The order logic is as follows: we first fix $\delta$ such that $\tilde{u}_{1,\delta}$ and $v_0$ have disjoint supports. Then $\tilde{u}_{1,\delta}$ and $T_\ld v_0$ have disjoint supports for any $\ld\in(0,1)$.}
\begin{align*}
\|w_\ld\|^p_p=\|\tilde{u}_{1,\delta}\|^p_p+\|T_\ld v_0\|^p_p
\end{align*}
for all $p\in [2,2^*]$. Hence $\mM(w_\ld)=c_2$. Moreover, one easily verifies that
\begin{align*}
\|\nabla w_\ld\|_2&\to\|\nabla \tilde{u}_{1,\delta}\|_2,\\
\| w_\ld\|_p&\to\| \tilde{u}_{1,\delta}\|_p
\end{align*}
for all $p\in(2,2^*]$ as $\ld\to 0$. Using the continuity of $f$ once again we obtain that
\begin{align*}
\max_{t>0}\mH(T_tw_\ld)\leq \max_{t>0}\mH(T_t \tilde{u}_{1,\delta})+\frac{\vare}{4}
\end{align*}
for sufficiently small $\ld>0$. Finally, combing with \eqref{pert 2} and \eqref{pert 1} we infer that
\begin{align*}
m_{c_2}&\leq \max_{t>0}\mH(T_t w_\lambda)\leq \max_{t>0}\mH(T_t\tilde{u}_{1,\delta})+\frac{\varepsilon}{4}\nonumber\\
&\leq \max_{t>0}\mH(u_1^t)+\frac{\varepsilon}{2}=\mH(u_1)+\frac{\varepsilon}{2}\leq m_{c_1}+\varepsilon,
\end{align*}
which implies the monotonicity of $c\mapsto m_c$ on $(0,\mM(Q))$.

Finally, we show the continuity of the curve $c\mapsto m_c$. Since $c\mapsto m_c$ is non-increasing, it suffices to show that for any $c\in(0,\mM(Q))$ and any sequence $c_n\downarrow c$ we have
\begin{align*}
m_c\leq \lim_{n\to \infty}m_{c_n}.
\end{align*}
By the same reasoning we can also prove that $m_c\geq \lim_{n\to \infty}m_{c_n}$ for any sequence $c_n\uparrow c$ and the continuity follows. Let $\vare>0$ be an arbitrary positive number. By the definition of $m_{c_n}$ we can find some $u_n\in V(c_n)$ such that
\begin{align}\label{pert 3}
\mH(u_n)\leq m_{c_n}+\frac{\vare}{2}\leq m_c+\frac{\vare}{2}.
\end{align}
We define $\tilde{u}_n=(c_n^{-1}c)^{\frac{1}{2}} \cdot u_n:=\rho_n u_n$. Then $\mM(\tilde{u}_n)=c$ and $\rho_n\uparrow 1$. Since $u_n\in V(c_n)$, we obtain that
\begin{align}
m_{c}+\frac{\vare}{2}\geq m_{c_n}+\frac{\vare}{2}&\geq \mH(u_n)=\mH(u_n)-\frac{1}{\tbs}\mK(u_n)\nonumber\\
&=\frac{1}{d}\bg(\|\nabla u_n\|_2^2-\frac{d}{d+2}\|u_n\|_\tas^\tas\bg)\nonumber\\
&\geq \frac{1}{d}\bg(1-\bg(\frac{\mM(u_n)}{{\mM(Q)}}\bg)^{\frac{2}{d}}\bg)\|\nabla u_n\|_2^2\nonumber\\
&= \frac{1}{d}\bg(1-\bg(\frac{c+o_n(1)}{{\mM(Q)}}\bg)^{\frac{2}{d}}\bg)\|\nabla u_n\|_2^2.
\end{align}
Thus $(u_n)_n$ is bounded in $H^1(\R^d)$ and up to a subsequence we infer that there exist $A,B\geq 0$ such that
\begin{align}
\|\nabla u_n\|_2^2-\frac{d}{d+2}\|u_n\|_\tas^\tas=A+o_n(1),\quad\|u_n\|_\tbs^\tbs=B+o_n(1).
\end{align}
On the other hand, using $\mK(u_n)=0$ and Sobolev inequality we see that
\begin{align}\label{lower bound tbs}
\frac{1}{d}\bg(1-\bg(\frac{c+o_n(1)}{{\mM(Q)}}\bg)^{\frac{2}{d}}\bg)\|\nabla u_n\|_2^2
\leq \frac{1}{d}\bg(\|\nabla u_n\|_2^2-\frac{d}{d+2}\|u_n\|_\tas^\tas\bg)=\frac{1}{d}\|u_n\|_\tbs^\tbs\leq
\frac{\mS^{\frac{d}{2-d}}}{d}\|\nabla u_n\|_2^\tbs.
\end{align}
Hence $\liminf_{n\to\infty}\|\nabla u_n\|_2^2>0$, which combining with \eqref{lower bound tbs} also implies
\begin{align*}
A=\lim_{n\to\infty}\bg(\|\nabla u_n\|_2^2-\frac{d}{d+2}\|u_n\|_\tas^\tas\bg)>0,\quad
B=\lim_{n\to\infty}\|u_n\|_\tbs^\tbs>0.
\end{align*}
Therefore $f$ is continuous at the point $(A,B)$. Using also the fact that $\rho_n\uparrow 1$ we obtain
\begin{align}
m_{c}&\leq \max_{t>0}\mH(T_t \tilde{u}_n)
=\max_{t>0}\bg\{\frac{t^2\rho_n^2}{2}\|\nabla{u}_{n}\|_2^2
- \frac{t^2\rho_n^\tas}{\tas}\|{u}_{n}\|_\tas^\tas
-\frac{t^\tbs\rho_n^\tbs}{\tbs}\|{u}_{n}\|_\tbs^\tbs\bg\}\nonumber\\
&\leq\max_{t>0}\bg\{t^2\frac{A}{2}-t^\tbs\frac{B}{\tbs}\bg\}+\frac{\vare}{4}\nonumber\\
&\leq \max_{t>0}\bg\{\frac{t^2}{2}\|\nabla{u}_{n}\|_2^2
- \frac{t^2}{\tas}\|{u}_{n}\|_\tas^\tas
-\frac{t^\tbs}{\tbs}\|{u}_{n}\|_\tbs^\tbs\bg\}+\frac{\vare}{2}\nonumber\\
&=\max_{t>0}\mH(T_t u_n)+\frac{\vare}{2}=\mH(u_n)+\frac{\vare}{2}\leq m_{c_n}+\vare
\end{align}
by choosing $n$ sufficiently large. The claim follows from the arbitrariness of $\vare$.
\end{proof}

The following lemma shows that the NLS-flow leaves solutions starting from ${\mA}$ invariant.
\begin{lemma}\label{invariant lemma}
Let $u$ be a solution of \eqref{NLS} with $u(0)\in {\mA}$. Then $u(t)\in {\mA}$ for all $t$ in the maximal lifespan. Assume also $\mM(u)=(1-\delta)^{\frac{d}{2}}\mM(Q)$, then
\begin{align}\label{lower bound kt}
&\,\inf_{t\in I_{\max}}\mK(u(t))\nonumber\\
\geq&\,\min\bg\{\frac{4\delta}{d}\mH(u(0)),
\bg(\bg(\frac{d}{\delta(d-2)}\bg)^{\frac{d-2}{4}}-1\bg)^{-1}\bg(m_{\mM(u(0))}-\mH(u(0))\bg)\bg\}.
\end{align}
\end{lemma}

\begin{proof}
By the mass and energy conservation, to show the invariance of solutions starting from $\mA$ under the NLS-flow, we only need to show that $\mK(u(t))> 0$ for all $t\in I_{\max}$. Suppose that there exist some $t$ such that $\mK(u(t))\leq 0$. By continuity of $u(t)$ there exists some $s\in(0,t]$ such that $\mK(u(s))=0$. By conservation of mass we also know that $0<\mM(u(s))<{\mM(Q)}$. By the definition of $m_c$ we immediately obtain that
\begin{align*}
m_{\mM(u(s))}\leq \mH(u(s))<m_{\mM(u(0))}= m_{\mM(u(s))},
\end{align*}
a contradiction. We now show \eqref{lower bound kt}. Direct calculation yields
\begin{align}\label{der of h}
\frac{d^2}{d\ld^2}\mH(T_\ld u(t))=-\frac{1}{\ld^2}\mK(T_\ld u(t))+\frac{2}{\ld^2}\bg(\mK(T_\ld u(t))-\frac{2}{d-2}\|T_\ld u(t)\|_\tbs^\tbs\bg).
\end{align}
If $\mK(u(t))-\frac{2}{d-2}\| u(t)\|_\tbs^\tbs\geq 0$, then using \eqref{GN-L2} we see that
\begin{align*}
\mK(u(t))&=\|\nabla u\|_2^2-\frac{d}{d+2}\|u\|_\tas^\tas-\|u\|_\tbs^\tbs\nonumber\\
&\geq \delta\|\nabla u\|_2^2-\frac{d-2}{2}\mK(u(t)),
\end{align*}
which combining with \eqref{third} implies that
\begin{align}\label{lower bd 1}
\mK(u(t))\geq \frac{2\delta}{d}\|\nabla u(t)\|_2^2\geq  \frac{4\delta}{d}\mH(u(0)),
\end{align}
where for the last inequality we also used the conservation of energy. Suppose now that
\begin{align}\label{negative 0}
\mK(u(t))-\frac{2}{d-2}\| u(t)\|_\tbs^\tbs< 0.
\end{align}
Then
\begin{align*}
\frac{2}{d-2}\| u(t)\|_\tbs^\tbs&>\|\nabla u(t)\|_2^2-\frac{d}{d+2}\|u\|_\tas^\tas-\|u\|_\tbs^\tbs\nonumber\\
&\geq \delta\|\nabla u(t)\|_2^2-\|u(t)\|_\tbs^\tbs.
\end{align*}
Hence
\begin{align}\label{2 by 2*}
\|u(t)\|_\tbs^\tbs> \frac{\delta(d-2)}{d}\|\nabla u(t)\|_2^2.
\end{align}
Since $\mK(u(t))>0$, by Lemma \ref{positive or negative k} we know that there exists some $\ld_*\in(1,\infty)$ such that
\begin{align}\label{positive ld for K}
\mK(T_\ld u(t))>0\quad\forall\,\ld\in[1,\ld_*)
\end{align}
and
\begin{align*}
0=\mK(T_{\ld_*}u(t))=\ld_*^2\bg(\|\nabla u(t)\|_2^2-\frac{d}{d+2}\|u(t)\|_\tas^\tas\bg)-\ld_*^\tbs\|u(t)\|_\tbs^\tbs,
\end{align*}
which gives
\begin{align}\label{2* by 2}
\|u(t)\|_\tbs^\tbs&\leq \ld_*^{2-\tbs}(\|\nabla u(t)\|_2^2-\frac{d}{d+2}\|u(t)\|_\tas^\tas)\leq \ld_*^{2-\tbs}\|\nabla u(t)\|_2^2.
\end{align}
\eqref{2 by 2*} and \eqref{2* by 2} then yield
\begin{align}\label{bound of ld}
\ld_*\leq\bg(\frac{d}{\delta(d-2)}\bg)^{\frac{d-2}{4}}.
\end{align}
On the other hand, one easily checks that
\begin{align}\label{negative 1}
\frac{d}{d\ld}\bg(\frac{1}{\ld^2}\bg(\mK(T_\ld u(t))-\frac{2}{d-2}\|T_\ld u(t)\|_\tbs^\tbs\bg)\bg)
=-\frac{2(\tbs-2)}{d-2}\ld^{2^*-3}\|u(t)\|_\tbs^\tbs<0.
\end{align}
Integrating \eqref{negative 1} and using \eqref{negative 0}, we find that for $\ld\geq 1$ we have
\begin{align}\label{negative der of k}
\frac{1}{\ld^2}\bg(\mK(T_\ld u(t))-\frac{2}{d-2}\|T_\ld u(t)\|_\tbs^\tbs\bg)\leq 0.
\end{align}
\eqref{der of h}, \eqref{positive ld for K} and \eqref{negative der of k} then imply that $\frac{d^2}{d\ld^2}\mH(T_\ld u(t))\leq 0$ for all $\ld\in[1,\ld_*]$. Finally, combining with \eqref{bound of ld}, the fact that $\mK(T_{\ld_*}u(t))=0$ and Taylor expansion we infer that
\begin{align}
&\,\bg(\bg(\frac{d}{\delta(d-2)}\bg)^{\frac{d-2}{4}}-1\bg)\mK(u(t))\nonumber\\
\geq&\, \,(\ld_*-1)\bg(\frac{d}{d\ld}\mH(T_\ld u(t))\bg|_{\ld=1}\bg)\nonumber\\
\geq&\,\, \mH(T_{\ld_*}u(t))-\mH(u(t))\nonumber\\
\geq&\,\, m_{\mM(u(0))}-\mH(u(0)).
\end{align}
This together with \eqref{lower bd 1} yields \eqref{lower bound kt}.
\end{proof}

\begin{lemma}\label{mtilde equal m}
Let
\begin{align}\label{variational problem mtildec}
\tilde{m}_{c}&:=\inf_{u\in H^1(\R^d)}\{\mI(u):\mM(u)= c,\,\mK(u)\leq 0\}.
\end{align}
Then $m_{c}=\tm_c$ for any $c\in(0,\mM(Q))$.
\end{lemma}

\begin{proof}
Let $(u_n)_n$ be a minimizing sequence for the variational problem \eqref{variational problem mtildec}, i.e.
\begin{align*}
\lim_{n\to\infty}\mI(u_n)=\tilde{m}_c,\quad
\mM(u_n)=c,\quad
\mK(u_n)\leq 0.
\end{align*}
Using Lemma \ref{positive or negative k} we know that there exists some $\ld_n\in(0,1]$ such that $\mK(T_{\ld_n}u_n)=0$. Thus
\begin{align*}
m_c\leq \mH(T_{\ld_n}u_n)=\mI(T_{\ld_n}u_n)\leq \mI(u_n)=\tilde{m}_c+o_n(1).
\end{align*}
Sending $n\to\infty$ we infer that $m_c\leq \tm_c$. On the other hand,
\begin{align}
\tm_c
&\leq\inf_{u\in H^1(\R^d)}\{\mI(u):\mM(u)=c,\mK(u)=0\}\nonumber\\
&=\inf_{u\in H^1(\R^d)}\{\mH(u):\mM(u)=c,\mK(u)=0\}=m_c.
\end{align}
This completes the proof.
\end{proof}

Let $m_0:=\lim_{c\downarrow 0}m_c$ and $m_{Q}:=\lim_{c\uparrow{\mM(Q)}}m_{c}$. We define the set $\Omega$ by its complement
\begin{align}
\Omega^c&:=\{(c,h)\in\R^2: c\geq {\mM(Q)}\}\cup\{(c,h)\in\R^2: c\in[0,\mM(Q)),\, h\geq m_c\}
\end{align}
and the function $\mD:\R^2\to[0,\infty]$ by
\begin{align}\label{MEI functional}
\mD(c,e,k)=\left\{
             \begin{array}{ll}
             h+\frac{h+c}{\mathrm{dist}((c,h),\Omega^c)},&\text{if $(c,h)\in \Omega$},\\
             \infty,&\text{otherwise}.
             \end{array}
\right.
\end{align}
For $u\in H^1(\R^d)$ also define $\mD(u):=\mD(\mM(u),\mH(u))$.

\begin{remark}
By modifying the arguments in \cite[Thm. 1.2]{SoaveCritical} and \cite[Lem. 3.3]{wei2021normalized} we are able to show that
\begin{align*}
m_0=\mH^*(W),\quad m_{Q}=0.
\end{align*}
Nevertheless, the precise values of $m_0$ and $m_Q$ have no impact on the scattering result, all we need here is the monotonicity and continuity of the curve $c\mapsto m_c$. We will therefore postpone the proof to Appendix \ref{section endpoint}.
\end{remark}

\begin{lemma}\label{killip visan curve}
Assume $v\in H^1(\R^d)$ such that $\mK(v)\geq 0$. Then
\begin{itemize}
\item[(i)]$\mD(v)=0$ if and only if $v=0$.
\item[(ii)] $0<\mD(v)<\infty$ if and only if $v\in\mA$.
\item[(iii)] $\mD$ leaves $\mA$ invariant under the NLS flow.
\item[(iv)] Let $u_1,u_2\in\mA$ with $\mM(u_1)\leq \mM(u_2)$ and $\mH(u_1)\leq \mH(u_2)$, then $\mD(u_1)\leq \mD(u_2)$. If in addition either $\mM(u_1)<\mM(u_2)$ or $\mH(u_1)<\mH(u_2)$, then $\mD(u_1)<\mD(u_2)$.
\item[(v)] Let $\mD_0\in(0,\infty)$. Then
\begin{align}
\|\nabla u\|^2_{2}&\sim_{\mD_0}\mH(u),\\
\|u\|^2_{H^1}&\sim_{\mD_0}\mH(u)+\mM(u)\sim_{\mD_0}\mD(u)
\end{align}
uniformly for all $u\in \mA$ with $\mD(u)\leq \mD_0$.
\item[(vi)] For all $u\in \mA$ with $\mD(u)\leq \mD_0$ with $\mD_0\in(0,\infty)$we have
\begin{align}\label{small of unaaa}
|\mH(u)-m_{\mM(u)}|\gtrsim 1.
\end{align}
\end{itemize}
\end{lemma}

\begin{proof}
\begin{itemize}
\item[(i)]That $v=0$ implies $\mD(v)=0$ is trivial. The other direction follows immediately from \eqref{third} and the definition of $\mD$.

\item[(ii)]It is trivial that $v\in \mA$ implies $\mD(v)<\infty$. By Lemma \ref{pos of k implies pos of h} we also know that $\mH(v)> 0$, which implies $\mD(v)>0$. Now let $\0<\mD(v)<\infty$. Then $\mM(v)\in(0,{\mM(Q)})$. By definition of $\mD$ and Lemma \ref{pos of k implies pos of h} we infer that $0\leq\mH(v)<m_{\mM(v)}$, which also implies $\mK(v)>0$ by the definition of $m_{\mM(v)}$. Hence we conclude that $v\in\mA$.

\item[(iii)]This follows immediately from the conservation of mass and energy of the NLS flow, the definition of $\mD$ and Lemma \ref{invariant lemma}.
\item[(iv)]This follows from the fact that $c\mapsto m_c$ is monotone decreasing on $(0,{\mM(Q)})$ and the definition of $\mD$.

\item[(v)]Since $u\in\mA$, we know that $\mM(u)\in(0,\mM(Q))$ and using Lemma \ref{pos of k implies pos of h} also $\mH(u)\in [0,m_{\mM(u)})$. Thus
\begin{align*}
\mathrm{dist}\bg((\mM(u),\mH(u)),\Omega^c\bg)\leq \mathrm{dist}\bg((\mM(u),\mH(u)),(\mM(Q),\mH(u))\bg)=\mM(Q)-\mM(u).
\end{align*}
Since $\mH(u)\geq 0$, we have
\begin{align}\label{mass constraint}
\mD(u)\geq \frac{\mM(u)}{\mM(Q)-\mM(u)},
\end{align}
which implies that
\begin{align*}
\frac{1}{1+\mD(u)}\leq 1-\frac{\mM(u)}{\mM(Q)}.
\end{align*}
Since $1-\alpha\lesssim_d 1-\alpha^{\frac{2}{d}}$ for $\alpha\in[0,1]$, we deduce that
\begin{align*}
\frac{1}{1+\mD(u)}\lesssim 1-\bg(\frac{\mM(u)}{\mM(Q)}\bg)^{\frac{2}{d}}.
\end{align*}
Using $\mK(u)\geq 0$ we have
\begin{align}\label{energy constraint}
\mD(u)\geq&\, \mH(u)\geq \mH(u)-\frac{1}{\tbs}\mK(u)\nonumber\\
=&\,\frac{1}{d}(\|\nabla u\|_2^2-\frac{d}{d+2}\|u\|_\tas^\tas)\nonumber\\
\geq &\,\frac{1}{d}\bg(1-\bg(\frac{\mM(u)}{{\mM(Q)}}\bg)^{\frac{2}{d}}\bg)\|\nabla u\|_2^2
\gtrsim \frac{\|\nabla u\|_2^2}{d(1+\mD(u))}.
\end{align}
Therefore $\|\nabla u\|_2^2\lesssim_{\mD_0}\mH(u)$. Combining with \eqref{third} we conclude that
\begin{align*}
\|\nabla u\|_2^2\sim_{\mD_0}\mH(u),\quad
\|u\|_{H^1}^2\sim_{\mD_0}\mH(u)+\mM(u).
\end{align*}
It remains to show $\mH(u)+\mM(u)\sim_{\mD_0}\mD(u)$. Using \eqref{mass constraint} and \eqref{energy constraint} we infer that
\begin{align*}
\mH(u)+\mM(u)\sim_{\mD_0}\|u\|^2_{H^1}\lesssim_{\mD_0} \mD(u).
\end{align*}
To show $\mD(u)\lesssim_{\mD_0} \mH(u)+\mM(u)$ we discuss the following different cases: If $\mM(u)\geq\frac{1}{2}\mM(Q)$, then using the fact that $\mH(u)\geq 0$ we have
\begin{align*}
\mathrm{dist}\bg((\mM(u),\mH(u)),\Omega^c\bg)\geq \frac{\mM(u)}{\mD_0}\geq\frac{\mM(Q)}{2\mD_0},
\end{align*}
which implies
\begin{align*}
\mD(u)\leq \frac{2\mD_0}{\mM(Q)}\bg(\mM(u)+\mH(u)\bg)+\mH(u).
\end{align*}
If $\mM(u)< \frac{1}{2}\mM(Q)$ and $\mH(u)\geq \frac{1}{2}m_{\frac{1}{2}\mM(Q)}$, then analogously we obtain
\begin{align*}
\mD(u)\leq \frac{2\mD_0 }{m_{\frac{1}{2}\mM(Q)}}\bg(\mM(u)+\mH(u)\bg)+\mH(u).
\end{align*}
If $\mM(u)< \frac{1}{2}\mM(Q)$ and $\mH(u)< \frac{1}{2}m_{\frac{1}{2}\mM(Q)}$, then
\begin{align*}
\mathrm{dist}\bg((\mM(u),\mH(u)),\Omega^c\bg)\geq\mathrm{dist}\bg(\bg(\frac{1}{2}\mM(Q),\frac{1}{2}m_{\frac{1}{2}\mM(Q)}\bg),\Omega^c\bg)=:\alpha_0>0,
\end{align*}
where the first inequality and the positivity of $\alpha_0$ follows form the monotonicity of $c\mapsto m_c$. Therefore
\begin{align*}
\mD(u)\leq \frac{1}{\alpha_0}\bg(\mM(u)+\mH(u)\bg)+\mH(u).
\end{align*}
Summing up the proof of (v) is complete.

\item[(vi)]
If this were not the case, then we could find a sequence $(u_n)_n\subset\mA$ such that
\begin{align}\label{small of un}
|\mH(u_n)-m_{\mM(u_n)}|=o_n(1).
\end{align}
But then
\begin{align*}
\mathrm{dist}\bg(\bg(\mM(u_n),\mH(u_n)\bg),\Omega^c\bg)\leq& \, \mathrm{dist}\bg(\bg(\mM(u_n),\mH(u_n)\bg),\bg(\mM(u_n),m_{\mM(u_n)}\bg)\bg)\nonumber\\
=&\,|m_{\mM(u_n)}-\mH(u_n)|=o_n(1).
\end{align*}
If $\mM(u_n)\gtrsim 1$, then $\mD(u_n)\gtrsim \frac{1}{o_n(1)}$, contradicting $\mD(u_n)\leq \mD_0$. If $\mM(u_n)=o_n(1)$, then by \eqref{small of un} we know that $\mH(u_n)\gtrsim 1$ and similarly we may again derive the contradiction $\mD(u_n)\gtrsim \frac{1}{o_n(1)}$. This finishes the proof of (vi) and also the desired proof of Lemma \ref{killip visan curve}.
\end{itemize}
\end{proof}

\subsection{Large scale approximation}
In this section, we show that the nonlinear profiles corresponding to low frequency and high frequency bubbles can be well approximated by the solutions of the mass- and energy-critical NLS respectively.

\begin{lemma}[Large scale approximation for $\ld_\infty=\infty$]\label{persistance l2}
Let $u$ be the solution of the focusing mass-critical NLS
\begin{align}
i\pt_t u+\Delta u+|u|^{\frac{4}{d}}u=0\label{mass-critical nls}
\end{align}
with $u(0)=u_0\in H^1(\R^d)$ and $\mM(u_0)<\mM(Q)$. Then $u$ is global and
\begin{align}
\|u\|_{W_\tas(\R)}&\leq C(\mM(u_0)),\label{persistancy l2 1}\\
\||\nabla|^s u\|_{S(\R)}&\lesssim_{\mM(u_0)}\||\nabla|^s u_0\|_2\label{persistancy l2 2}
\end{align}
for $s\in\{0,1\}$. Moreover, we have the following large scale approximation result for \eqref{mass-critical nls}: Let $(\ld_n)_n\subset(0,\infty)$ such that $\ld_n\to \infty$, $(t_n)_n\subset\R$ such that either $t_n\equiv 0$ or $t_n\to\pm\infty$ and $(\xi_n)_n\subset\R^d$ such that $(\xi_n)_n$ is bounded. Define
$$\phi_n:=g_{\xi_n,x_n,\ld_n}e^{it_n\Delta}P_{\leq \ld_n^\theta}\tdu$$
for some $\theta\in(0,1)$. Then for all sufficiently large $n$ the solution $u_n$ of \eqref{NLS} with $u_n(0)=\phi_n$ is global and scattering in time with
\begin{align}
\limsup_{n\to\infty}\|\la\nabla\ra u_n\|_{S(\R)}&\leq C(\mM(\tdu)),\label{L2 proxy 1}\\
\lim_{n\to\infty}\|u_n\|_{W_\tbs(\R)}&= 0.\label{L2 proxy 2}
\end{align}
Furthermore, for every $\beta>0$ there exists $N_\beta\in\N$ and $\phi_\beta\in C_c^\infty(\R\times \R^d)$ such that
\begin{align}
\bg\|u_n-\ld_n^{-\frac{d}{2}}e^{-it|\xi_n|^2}e^{i\xi_n\cdot x}\phi_\beta\bg(\frac{t}{\ld_n^2}+t_n,\frac{x-x_n-2t\xi_n}{\ld_n}\bg)\bg\|_{W_\tas(\R)}\leq \beta\label{L2 proxy 3},\\
\bg\|\nabla u_n-i\xi_n\ld_n^{-\frac{d}{2}}e^{-it|\xi_n|^2}e^{i\xi_n\cdot x}\phi_\beta\bg(\frac{t}{\ld_n^2}+t_n,\frac{x-x_n-2t\xi_n}{\ld_n}\bg)\bg\|_{W_\tas(\R)}\leq \beta\label{L2 proxy 6}
\end{align}
for all $n\geq N_\beta$.
\end{lemma}

\begin{proof}
\eqref{persistancy l2 1} and the fact that $u$ is global are proved in \cite{Dodson4dmassfocusing}. We denote $C_1=C(\mM(u_0))$. By Strichartz and H\"older, for any time interval $I\ni s_0$, we have
\begin{align}
\||\nabla|^s v\|_{S(I)}\leq C_2(\||\nabla|^s v(s_0)\|_{2}+\|v\|^{\frac{4}{d}}_{W_\tas(I)}\||\nabla|^s v\|_{S(I)})
\end{align}
for any solution $v$ of \eqref{mass-critical nls} defined on $I$, where $C_2$ is some positive constant depending only on $d$. We divide $I$ into $m$ intervals $I_1,I_2,\cdots,I_m$ such that
\begin{align*}
\|u\|_{W_\tas(I_j)}\leq (2C_2)^{-\frac{d}{4}}\quad\forall j=1,\cdots,m.
\end{align*}
Then $m\leq C_1(2C_2)^{\frac{d}{4}}+1$. For $I_1=[t_0,t_1]$, we have particularly
\begin{align*}
\||\nabla|^s u\|_{S(I_1)}\leq 2C_2\||\nabla|^s u(t_0)\|_{2}
\end{align*}
and thus also
\begin{align*}
\||\nabla|^s u(t_1)\|_2\leq 2C_2\||\nabla|^s u(t_0)\|_{2}.
\end{align*}
Arguing inductively for all $j=2,\cdots,m-1$ and summing the estimates on all subintervals up yield \eqref{persistancy l2 2}, since $C_1$ depends only on $\mM(u_0)$ and $C_2$ only on $d$.

Next, we prove the claims concerning the large scale approximation. Let $w$ and $w_n$ be the solutions of \eqref{mass-critical nls} with $w(0)=\phi$ and $w_n(0)=\phi_n$ respectively when $t_n\equiv 0$. For $t_n\to \pm\infty$ we define $w$ and $w_n$ as solutions of \eqref{mass-critical nls} which scatter to $e^{it\Delta}\tdu$ and $e^{it\Delta}P_{\leq\ld_n^\theta}\tdu$ in $L^2(\R^d)$ as $t\to\pm\infty$ respectively. By \cite{Dodson4dmassfocusing} we know that $w$ is global, scatters in time and
\begin{align*}
\|w\|_{S(\R)}\leq C(\mM(\tdu)).
\end{align*}
On the other hand, since
\begin{align*}
&\lim_{n\to\infty}\lim_{t\to\pm\infty}\|w_n(t)-w(t)\|_2\nonumber\\
\leq &\lim_{n\to\infty}\lim_{t\to\pm\infty}\bg(\|w_n(t)-e^{it\Delta}P_{\leq \ld_n^\theta}\tdu\|_2
+\|w(t)-e^{it\Delta}\tdu\|_2+\|\tdu-P_{\leq \ld_n^\theta}\tdu\|_2\bg)=0,
\end{align*}
by the standard stability result for mass-critical NLS (see for instance \cite{KillipVisanNotes}) we know that $w_n$ is global and scattering in time for all sufficiently large $n$ and
\begin{align*}
\limsup_{n\to\infty}\|w_n\|_{W_\tas(\R)}\lesssim_{\mM(\tdu)}1.
\end{align*}
Using Bernstein, Strichartz and Lemma \ref{persistance l2} we additionally have
\begin{align*}
\|w_n\|_{W_\tbs(\R)}\lesssim \|\nabla w_n\|_{S(\R)}\lesssim_{\mM(\tdu)} \ld_n^\theta.
\end{align*}
We now define
\begin{align}
\tilde{u}_n(t,x):=\ld_n^{-\frac{d}{2}}e^{i\xi_n\cdot x}e^{-it|\xi_n|^2}w_n\bg(\frac{t}{\ld_n^2}+t_n,\frac{x-x_n-2t\xi_n}{\ld_n}\bg).
\end{align}
Using the symmetry invariance for mass-critical NLS one easily verifies that $\tilde{u}_n$ is also a global and scattering solution of \eqref{mass-critical nls}. In particular, we have
\begin{align}
\|\la\nabla\ra\tilde{u}_n\|_{S(\R)}&\lesssim(1+|\xi_n|)\|w_n\|_{S(\R)}+\ld_n^{-1}\|w_n\|_{S(\R)}\lesssim 1+\ld_n^{-(1-\theta)}\to 1,\label{L2 proxy 4}\\
\|\tilde{u}_n\|_{W_\tbs(\R)}&=\ld_n^{-1}\|w_n\|_{W_\tbs(\R)}\lesssim\ld^{-1}_n\|\nabla w_n\|_{S(\R)}\lesssim \ld_n^{-(1-\theta)}\to 0\label{L2 proxy 5}
\end{align}
as $n\to\infty$. We next show that $\tilde{u}_n$ is asymptotically a good approximation of $u_n$ using Lemma \ref{long time pert}. Rewrite \eqref{mass-critical nls} for $\tilde{u}_n$ to
\begin{align}
i\pt_t \tilde{u}_n+\Delta \tilde{u}_n+|\tilde{u}_n|^{\frac{4}{d}}\tilde{u}_n
+|\tilde{u}_n|^{\frac{4}{d-2}}\tilde{u}_n+e=0,
\end{align}
where $e=-|\tilde{u}_n|^{\frac{4}{d-2}}\tilde{u}_n$. Using \eqref{GN-L2}, Sobolev and conservation of energy we obtain that
\begin{align*}
\|\nabla u_n(t)\|^2_2\lesssim \mH(u_n(t))+\frac{1}{\tbs}\|u_n(t)\|_\tbs^\tbs\lesssim
\mH(\phi_n)+\|\nabla u_n(t)\|_2^\tbs.
\end{align*}
But using Bernstein we also see that
\begin{align*}
\|\nabla \phi_n\|_2\lesssim \ld_n^{-1}|\xi_n|\|\phi\|_2+\ld_n^{-(1-\theta)}\|\phi\|_2\to 0,
\end{align*}
which implies
\begin{align*}
\mH(\phi_n)\lesssim \|\nabla \phi_n\|_2^2\to 0.
\end{align*}
By standard continuity arguments we conclude that $\limsup_{n\to\infty}\|u_n\|_{L_t^\infty \dot{H}^1_x(I)}<\infty$, and \eqref{condition c1} is satisfied by combining with conservation of mass for sufficiently large $n$. It remains to show \eqref{condition b}. Indeed, using H\"older we obtain that
\begin{align}
\|\la\nabla\ra e\|_{L_{t,x}^{\frac{2(d+2)}{d+4}}}\leq\|\tilde{u}_n\|_{W_\tbs(\R)}^{\frac{4}{d-2}}\|\la\nabla\ra\tilde{u}_n\|_{W_\tas(\R)}.
\end{align}
Then \eqref{condition b} follows from \eqref{L2 proxy 4} and \eqref{L2 proxy 5}. \eqref{L2 proxy 1} and \eqref{L2 proxy 2} now follow from Lemma \ref{nls persistence}, \eqref{also proxy} and \eqref{L2 proxy 5}. Finally, to show \eqref{L2 proxy 3} and \eqref{L2 proxy 6} we first choose $\phi_\beta\in C_c^\infty(\R\times\R^d)$ and sufficiently large $n$ such that
\begin{align*}
\|w-\phi_\beta\|_{W_\tas(\R)}+
\|w-w_n\|_{W_\tas(\R)}+
\|\la\nabla\ra \tilde{u}_n-\la\nabla\ra u_n\|_{W_\tas(\R)}\lesssim \beta.
\end{align*}
Using chain rule and Bernstein we also deduce that
\begin{align}
\|\nabla \tilde{u}_n-i\xi_n \tilde{u}_n\|_{W_\tas(\R)}=\ld_n^{-1}\|\nabla w_n\|_{W_\tas(\R)}\lesssim\ld_n^{-(1-\theta)}\to 0.
\end{align}
Then \eqref{L2 proxy 3} and \eqref{L2 proxy 6} follow from triangular inequality and taking $n$ sufficiently large.
\end{proof}

Analogously, we have the following energy-critical version of Lemma \ref{persistance l2}, where the arguments from \cite{Dodson4dmassfocusing} are replaced by \cite{KenigMerle2006,Dodson4dfocusing,KillipVisan2010focusing}. We therefore omit the proof.

\begin{lemma}[Large scale approximation for $\ld_\infty=0$]\label{persistance h1}
Let $u$ be the solution of the focusing energy-critical NLS
\begin{align}
i\pt_t u+\Delta u+|u|^{\frac{4}{d-2}}u=0\label{energy-critical nls}
\end{align}
with $u(0)=u_0\in H^1(\R^d)$, $\mH^*(u_0)<\mH^*(W)$ and $\|u_0\|_{\dot{H}^1}< \|W\|_{\dot{H}^1}$. Additionally assume that $u_0$ is radial when $d=3$. Then $u$ is global and
\begin{align}
\|u\|_{W_\tbs(\R)}&\leq C(\mH^*(u_0)),\label{persistancy h1 1}\\
\||\nabla|^s u\|_{S(\R)}&\lesssim_{\mH^*(u_0)}\||\nabla|^s u_0\|_2\label{persistancy h1 2}
\end{align}
for $s\in\{0,1\}$. Moreover, we have the following large scale approximation result for \eqref{energy-critical nls}: Let $(\ld_n)_n\subset(0,\infty)$ such that $\ld_n\to 0$, $(t_n)_n\subset\R$ such that either $t_n\equiv 0$ or $t_n\to\pm\infty$. Define
$$\phi_n:=\ld_n g_{0,x_n,\ld_n}e^{it_n\Delta}P_{> \ld_n^\theta}\tdu$$
for some $\theta\in(0,1)$. Then for all sufficiently large $n$ the solution $u_n$ of \eqref{NLS} with $u_n(0)=\phi_n$ is global and scattering in time with
\begin{align}
\limsup_{n\to\infty}\|\la\nabla\ra u_n\|_{S(\R)}&\leq C(\mH^*(\tdu)),\label{H1 proxy 1}\\
\lim_{n\to\infty}\|u_n\|_{W_\tas(\R)}&= 0.\label{H1 proxy 2}
\end{align}
Furthermore, for every $\beta>0$ there exists $N_\beta\in\N$, $\phi_\beta\in C_c^\infty(\R\times \R^d)$ and $\psi_\beta\in C_c^\infty(\R\times \R^d;\C^d)$ such that
\begin{align}
\bg\|u_n-\ld_n^{-\frac{d}{2}+1}\phi_\beta\bg(\frac{t}{\ld_n^2}+t_n,\frac{x-x_n}{\ld_n}\bg)\bg\|_{W_\tbs(\R)}\leq \beta\label{H1 proxy 3},\\
\bg\|\nabla u_n-\ld_n^{-\frac{d}{2}}\psi_\beta\bg(\frac{t}{\ld_n^2}+t_n,\frac{x-x_n}{\ld_n}\bg)\bg\|_{W_\tas(\R)}\leq \beta\label{H1 proxy 4}
\end{align}
for all $n\geq N_\beta$.
\end{lemma}

\subsection{Existence of the minimal blow-up solution}\label{Existence of the minimal blow-up solution}
Having all the preliminaries we are ready to construct the minimal blow-up solution. Define
\begin{align*}
\tau(\mD_0):=\sup\bg\{&\|\psi\|_{{W_\tas\cap W_\tbs}(I_{\max})}:\nonumber\\
&\quad\quad\quad\quad\text{$\psi$ is solution of \eqref{NLS}, }\psi(0)\in {\mA},\mD(\psi(0))\leq \mD_0\bg\}
\end{align*}
and
\begin{align}\label{introductive hypothesis}
\mD^*&:=\sup\{\mD_0>0:\tau(\mD_0)<\infty\}.
\end{align}
By Lemma \ref{well posedness lemma}, Remark \ref{remark} and Lemma \ref{killip visan curve} (v) we know that $\mD^*>0$ and $\tau(\mD_0)<\infty$ for sufficiently small $\mD_0$. We will therefore assume that $\mD^*<\infty$ and aim to derive a contradiction, which will imply $\mD^*=\infty$ and the whole proof will be complete in view of Lemma \ref{killip visan curve} (ii). By the inductive hypothesis we may find a sequence $(\psi_n)_n$ with $(\psi_n(0))_n\subset {\mA}$ which are solutions of \eqref{NLS} with maximal lifespan $(I_{n})_n$ such that
\begin{gather}
\lim_{n\to\infty}\|\psi_n\|_{{W_\tas\cap W_\tbs}((\inf I_n,0])}=\lim_{n\to\infty}\|\psi_n\|_{{W_\tas\cap W_\tbs}([0, \sup I_n))}=\infty,\label{oo1}\\
\lim_{n\to\infty}\mD(\psi_n(0))=\mD^*.\label{oo2}
\end{gather}
Up to a subsequence we may also assume that
\begin{align*}
(\mM(\psi_n(0)),\mH(\psi_n(0)),\mI(\psi_n(0)))\to(\mM_0,\mH_0,\mI_0)\quad\text{as $n\to\infty$}.
\end{align*}
By continuity of $\mD$ and finiteness of $\mD^*$ we know that
\begin{align*}
\mD^*=\mD(\mM_0,\mH_0),\quad
\mM_0\in(0,\mM(Q)),\quad
\mH_0\in[0,m_{\mM_0}).
\end{align*}
From Lemma \ref{killip visan curve} (v) it follows that $(\psi_n(0))_n$ is a bounded sequence in $H^1(\R^d)$ and Lemma \ref{linear profile} is applicable for $(\psi_n(0))_n$. We define the nonlinear profiles as follows: For $\ld_\infty^k\in\{0,\infty\}$, we define $v_n^k$ as the solution of \eqref{NLS} with $v_n^k(0)=T_n^kP_n^k\tdu^k$. For $\ld_\infty^k=1$ and $t^k_\infty=0$, we define $v^k$ as the solution of \eqref{NLS} with $v^k(0)=\tdu^k$; For $\ld_\infty^k=1$ and $t^k_\infty\to\pm\infty$, we define $v^k$ as the solution of \eqref{NLS} that scatters forward (backward) to $e^{it\Delta}\tdu^k$ in $H^1(\R^d)$. In both cases for $\ld_\infty^k=1$ we define
\begin{align*}
v_n^k:=v^j(t+t_n,x-x_n^k).
\end{align*}
Then $v_n^j$ is also a solution of \eqref{NLS}. In all cases we have for each finite $1\leq k \leq K^*$
\begin{align}\label{conv of nonlinear profiles in h1}
\lim_{n\to\infty}\|v_n^k(0)-T_n^kP_n^k\tdu^k\|_{H^1}=0.
\end{align}

In the following, we establish a Palais-Smale type lemma which is essential for the construction of the minimal blow-up solution.

\begin{lemma}[Palais-Smale-condition]\label{Palais Smale}
Let $(\psi_n)_n$ be a sequence of solutions of \eqref{NLS} with maximal lifespan $I_n$, $\psi_n\in\mA$ and $\lim_{n\to\infty}\mD(u_n)=\mD^*$. Assume also that there exists a sequence $(t_n)_n\subset\prod_n I_n$ such that
\begin{align}\label{precondition}
\lim_{n\to\infty}\|\psi_n\|_{W_\tas\cap W_\tbs((\inf I_n,\,t_n])}=\lim_{n\to\infty}\|\psi_n\|_{W_\tas\cap W_\tbs([t_n,\,\sup I_n)}=\infty.
\end{align}
Then up to a subsequence, there exists a sequence $(x_n)_n\subset\R^d$ such that $(\psi_n(t_n, \cdot+x_n))_n$ strongly converges in $H^1(\R^d)$.
\end{lemma}

\begin{proof}
By time translation invariance we may assume that $t_n\equiv 0$. Let $(v_n^j)_{j,n}$ be the nonlinear profiles corresponding to the linear profile decomposition of $(\psi_n(0))_n$. Define
\begin{align*}
\Psi_n^k:=\sum_{j=1}^k v_n^j+e^{it\Delta}w_n^k.
\end{align*}
We will show that there exists exactly one non-trivial bad linear profile, relying on which the desired claim follows. We divide the remaining proof into three steps.
\subsubsection*{Step 1: Decomposition of energies and large scale proxies}
In the first step we show that the low and high frequency bubbles asymptotically meet the preconditions of Lemma \ref{persistance l2} and Lemma \ref{persistance h1} respectively. We first show that
\begin{align}
\mH(T_n^jP_n^j\phi^j)&> 0,\label{bd for S}\\
\mK(T_n^jP_n^j\phi^j)&> 0\label{pos of K}
\end{align}
for any finite $1\leq j\leq K^*$ and all sufficiently large $n=n(j)\in\N$. Since $\phi^j\neq 0$ we know that $T_n^jP_n^j\phi^j\neq 0$ for sufficiently large $n$. Suppose now that \eqref{pos of K} does not hold. Up to a subsequence we may assume that $\mK(T_n^jP_n^j\phi^j)\leq 0$ for all sufficiently large $n$. By the non-negativity of $\mI$, \eqref{conv of i} and \eqref{small of unaaa} we know that there exists some sufficiently small $\delta>0$ depending on $\mD^*$ and some sufficiently large $N_1$ such that for all $n>N_1$ we have
\begin{align}\label{contradiction1}
\tm_{\mM(T_n^jP_n^j\phi^j)}&\leq\mI(T_n^jP_n^j\phi^j)\leq \mI(\psi_n(0))+\delta\nonumber\\
&\leq\mH(\psi_n(0))+\delta\leq m_{\mM(\psi_n(0))}-2\delta,
\end{align}
where $\tm$ is the quantity defined by Lemma \ref{mtilde equal m}. By continuity of $c\mapsto m_c$ we also know that for sufficiently large $n$ we have
\begin{align}\label{contradiction3}
m_{\mM(\psi_n(0))}-2\delta\leq m_{\mM_0}-\delta.
\end{align}
Using \eqref{orthog L2 and H1} we deduce that for any $\vare>0$ there exists some large $N_2$ such that for all $n>N_2$ we have
\begin{align*}
\mM(T_n^jP_n^j\phi^j)\leq \mM_0+\vare.
\end{align*}
From the continuity and monotonicity of $c\mapsto m_c$ and Lemma \ref{mtilde equal m}, we may choose some sufficiently small $\vare$ to see that
\begin{align}\label{contradiction2}
\tm_{\mM(T_n^jP_n^j\phi^j)}=m_{\mM(T_n^jP_n^j\phi^j)}\geq m_{\mM_0+\vare}\geq m_{\mM_0}-\frac{\delta}{2}.
\end{align}
Now \eqref{contradiction1}, \eqref{contradiction3} and \eqref{contradiction2} yield a contradiction. Thus \eqref{pos of K} holds, which combining with Lemma \ref{pos of k implies pos of h} also yields \eqref{bd for S}. Similarly, for each $1\leq k\leq K^*$ we deduce
\begin{align}
\mH(w_n^k)&> 0,\label{bd for S wnj} \\
\mK(w_n^k)&> 0\label{pos of K wnj}
\end{align}
for sufficiently large $n$. Now using \eqref{orthog L2 and H1} to \eqref{conv of i} we have for any $1\leq k\leq K^*$
\begin{align}
\mM_0&=\mM(\psi_n(0))+o_n(1)=\sum_{j=1}^k \mM(S_n^j\tdu^j)+\mM(w_n^k)+o_n(1),\label{mo sum}\\
\mH_0&=\mH(\psi_n(0))+o_n(1)=\sum_{j=1}^k \mH(S_n^j\tdu^j)+\mH(w_n^k)+o_n(1)\label{eo sum},\\
\mI_0&=\mH(\psi_n(0))+o_n(1)=\sum_{j=1}^k \mI(S_n^j\tdu^j)+\mI(w_n^k)+o_n(1).\label{io sum}
\end{align}
From \eqref{mo sum} it is immediate that Lemma \ref{persistance l2} is applicable for solutions with initial data $T_n^j P_n^j\tdu^j$ for all sufficiently large $n$ in the case $\ld_\infty^j=\infty$. We will show that Lemma \ref{persistance h1} is applicable for solutions with initial data $T_n^j P_n^j\tdu^j$ for all sufficiently large $n$ in the case $\ld_\infty^j=0$. From Theorem \ref{soave}, Lemma \ref{mtilde equal m} and Lemma \ref{killip visan curve} we know that there exists some $\vare>0$ such that
\begin{align}\label{small of mass and energy}
\mM(u_0)\leq\mM(Q)-2\vare,\quad \mH_0\leq \mH^*(W)-2\vare,\quad \mI_0\leq \mH^*(W)-2\vare.
\end{align}
Since $\|T_n^j P_n^j\tdu^j\|_2\to 0$, by interpolation we have that
\begin{align*}
\mH(T_n^j P_n^j\tdu^j)-\mH^*(T_n^j P_n^j\tdu^j)\to 0,
\end{align*}
which implies
\begin{align*}
\mH^*(T_n^j P_n^j\tdu^j)\leq \mH_0+\vare \leq \mH^*(W)-\vare
\end{align*}
for all sufficiently large $n$. Similarly,
\begin{align*}
\|T_n^j P_n^j\tdu^j\|_{\dot{H}^1}&=2\mH^*(T_n^j P_n^j\tdu^j)+\frac{d-2}{d}\mI(T_n^j P_n^j\tdu^j)\nonumber\\
&\leq 2(\mH_0+\vare)+\frac{d-2}{d}(\mI_0+\vare)\nonumber\\
&\leq 2(\mH^*(W)-\vare)+\frac{d-2}{d}(\mH^*(W)-\vare)=\|W\|_{\dot{H}^1}-\bg(3-\frac{2}{d}\bg)\vare
\end{align*}
for all sufficiently large $n$. This completes the proof of Step 1.

\subsubsection*{Step 2: There exists at least one bad profile.}
First we claim that there exists some $1\leq J\leq K^*$ such that for all $j\geq J+1$ and all sufficiently large $n$, $v_n^j$ is global and
\begin{align}\label{uniform bound of unj}
\sup_{J+1\leq j\leq K^*}\lim_{n\to\infty}\|v_n^j\|_{W_\tas\cap W_\tbs(\R)}\lesssim 1.
\end{align}
Indeed, using \eqref{orthog L2 and H1} we infer that
\begin{align}\label{small initial data}
\lim_{k\to K^*}\lim_{n\to\infty}\sum_{j=1}^k\|T_n^jP_n^j\tdu^j\|_{H^1}<\infty.
\end{align}
Then \eqref{uniform bound of unj} follows from Lemma \ref{well posedness lemma}. In the same manner, by Lemma \ref{well posedness lemma} we infer that
\begin{align}\label{kkkk uniform bound of unj}
\sup_{J+1\leq k\leq K^*}\lim_{n\to\infty}\|\sum_{j=J+1}^k \la\nabla \ra v_n^j\|_{S(\R)}\lesssim 1
\end{align}
for any $J+1\leq k\leq K^*$. We now claim that there exists some $1\leq J_0\leq J$ such that
\begin{align}
\limsup_{n\to\infty}\|v_n^{J_0}\|_{W_\tas\cap W_\tbs(\R)}=\infty.
\end{align}
We argue by contradiction and assume that
\begin{align}\label{uniform bound of unj small}
\limsup_{n\to\infty}\|v_n^j\|_{W_\tas\cap W_\tbs(\R)}<\infty\quad\forall\,1\leq j\leq J.
\end{align}
Combining with \eqref{kkkk uniform bound of unj} and Lemma \ref{nls persistence} we deduce that
\begin{align}\label{super uniform}
\sup_{J+1\leq k\leq K^*}\lim_{n\to\infty}\|\sum_{j=1}^k \la\nabla\ra v_n^j\|_{S(\R)}\lesssim 1.
\end{align}
Therefore, using \eqref{orthog L2 and H1}, \eqref{conv of nonlinear profiles in h1} and Strichartz we confirm that the conditions \eqref{condition c1} to \eqref{condition a} are satisfied for sufficiently large $k$ and $n$, where we set $u=\psi_n$ and $w=\Psi_n^k$ therein. Once we can show that \eqref{condition b} is satisfied, we may apply Lemma \ref{long time pert} to obtain the contradiction
\begin{align}\label{contradiction 1}
\limsup_{n\to\infty}\|\psi_n\|_{W_\tas\cap W_\tbs(\R)}<\infty.
\end{align}
It is readily to see that
\begin{align}
e&=\,i\pt_t\Psi_n^k+\Delta\Psi_n^k+|\Psi_n^k|^{\frac{4}{d}}\Psi_n^k+|\Psi_n^k|^{\frac{4}{d-2}}\Psi_n^k\nonumber\\
&=\bg(\sum_{j=1}^k (i\pt_tv_n^j+\Delta v_n^j)+|\sum_{j=1}^kv_n^j|^{\frac{4}{d}}\sum_{j=1}^kv_n^j+|\sum_{j=1}^kv_n^j|^{\frac{4}{d-2}}\sum_{j=1}^kv_n^j\bg)\nonumber\\
&\quad+\bg(|\Psi_n^k|^{\frac{4}{d}}\Psi_n^k-|\Psi_n^k-e^{it\Delta}w_n^k|^{\frac{4}{d}}(\Psi_n^k-e^{it\Delta}w_n^k)\bg)\nonumber\\
&\quad+\bg(|\Psi_n^k|^{\frac{4}{d-2}}\Psi_n^k-|\Psi_n^k-e^{it\Delta}w_n^k|^{\frac{4}{d-2}}(\Psi_n^k-e^{it\Delta}w_n^k)\bg)\nonumber\\
&=:I_1+I_2+I_3.
\end{align}
In the following we show the asymptotic smallness of $I_1$ to $I_3$.
\subsubsection*{Step 2a: Smallness of $I_1$}
We will show that
\begin{align}
\lim_{k\to K^*}\lim_{n\to\infty}\|\la\nabla\ra I_1\|_{L_{t,x}^{\frac{2(d+2)}{d+2}}}=0.
\end{align}
Since $v_n^j$ solves \eqref{NLS}, we can rewrite $I_1$ to
\begin{align*}
I_1&=\sum_{j=1}^k \bg(-|v_n^j|^{\frac{4}{d}}v_n^j-|v_n^j|^{\frac{4}{d-2}}v_n^j\bg)
+|\sum_{j=1}^kv_n^j|^{\frac{4}{d}}\sum_{j=1}^kv_n^j-|\sum_{j=1}^kv_n^j|^{\frac{4}{d-2}}\sum_{j=1}^kv_n^j\nonumber\\
&=-\bg(\sum_{j=1}^k|v_n^j|^{\frac{4}{d}}v_n^j-|\sum_{j=1}^kv_n^j|^{\frac{4}{d}}\sum_{j=1}^kv_n^j\bg)
-\bg(\sum_{j=1}^k|v_n^j|^{\frac{4}{d-2}}v_n^j-|\sum_{j=1}^kv_n^j|^{\frac{4}{d-2}}\sum_{j=1}^kv_n^j\bg)
\end{align*}
By H\"older and \eqref{elementary ineq} we obtain for $s\in\{0,1\}$ that
\begin{align}
&\||\nabla|^s I_1\|_{L_{t,x}^{\frac{2(d+2)}{d+4}}}\nonumber\\
\lesssim_k&
\left\{
             \begin{array}{ll}
             \sum_{j\neq j'}\bg(\|v_n^j|\nabla|^s v_n^{j'}\|_{L_{t,x}^{\frac{d+2}{d}}(\R)}(\|v_n^j\|^{\frac{4}{d}-1}_{W_\tas(\R)}
             +\|v_n^{j'}\|^{\frac{4}{d}-1}_{W_\tas(\R)})\\
             \quad\quad\quad\quad\quad\quad\quad\quad+\|v_n^j|\nabla|^s v_n^{j'}\|_{L_{t,x}^{\frac{d+2}{d-1}}(\R)}(\|v_n^j\|^{\frac{4}{d-2}-1}_{W_\tbs(\R)}
             +\|v_n^{j'}\|^{\frac{4}{d-2}-1}_{W_\tbs(\R)})\bg),
             &\text{if $d=3$,}\\
             \\
             \sum_{j\neq j'}\bg(\|v_n^j|\nabla|^s v_n^{j'}\|^{\frac{4}{d}}_{L_{t,x}^{\frac{d+2}{d}}(\R)}\||\nabla|^s v_n^{j'}\|^{1-\frac{4}{d}}_{W_\tas(\R)}
             \\
             \quad\quad\quad\quad\quad\quad\quad\quad
             +\|v_n^j|\nabla|^s v_n^{j'}\|_{L_{t,x}^{\frac{d+2}{d-1}}(\R)}
             (\|v_n^j\|^{\frac{4}{d-2}-1}_{W_\tbs(\R)}+\|v_n^{j'}\|^{\frac{4}{d-2}-1}_{W_\tbs(\R)})\bg),
             &\text{if $d\in\{4,5\}$,}\\
             \\
             \sum_{j\neq j'}\bg(\|v_n^j|\nabla|^s v_n^{j'}\|^{\frac{4}{d}}_{L_{t,x}^{\frac{d+2}{d}}(\R)}
             \||\nabla|^s v_n^{j'}\|^{1-\frac{4}{d}}_{W_\tas(\R)}\\
             \quad\quad\quad\quad\quad\quad\quad\quad+\|v_n^j|\nabla|^s v_n^{j'}\|^{\frac{4}{d-2}}_{L_{t,x}^{\frac{d+2}{d-1}}(\R)}
             \||\nabla|^s v_n^{j'}\|^{1-\frac{4}{d-2}}_{W_\tbs(\R)}\bg),
             &\text{if $d\geq6$.}
             \end{array}\label{verylong}
\right.
\end{align}
In view of \eqref{uniform bound of unj} and \eqref{uniform bound of unj small} we only need to show that for any fixed $1\leq i,j\leq K^*$ with $i\neq j$ and any $s\in\{0,1\}$
\begin{align}
\lim_{n\to\infty}\bg(\|v_n^i |\nabla|^s v_n^j\|_{L_{t,x}^{\frac{d+2}{d}}(\R^d)}+\|v_n^i |\nabla|^s v_n^j\|_{L_{t,x}^{\frac{d+2}{d-1}}(\R^d)}\bg)=0.
\end{align}
We first consider the term $\|v_n^i v_n^j\|_{L_{t,x}^{\frac{d+2}{d}}(\R^d)}$. Notice that it suffices to consider the case $\ld_\infty^i,\ld_\infty^j\in\{1,\infty\}$. Indeed, using \eqref{H1 proxy 2} (which is applicable due to Step 1) and H\"older we already conclude that
\begin{align}\label{reduction to infinity case}
\|v_n^i v_n^j\|_{L_{t,x}^{\frac{d+2}{d}}(\R)}\lesssim \|v_n^i\|_{W_\tas(\R)}\|v_n^j\|_{W_\tas(\R)}\to 0
\end{align}
when $\ld_\infty^i$ or $\ld_\infty^j$ is equal to zero. Next, we claim that for any $\beta>0$ there exists some $\psi^i_\beta,\psi_\beta^j\in C_c^\infty(\R\times\R^d)$ such that
\begin{align}
\bg\|v^i_n-(\ld^i_n)^{-\frac{d}{2}}e^{-it|\xi^i_n|^2}e^{i\xi^i_n\cdot x}\psi^i_\beta\bg(\frac{t}{(\ld^i_n)^2}+t^i_n,\frac{x-x^i_n-2t\xi^i_n}{\ld^i_n}\bg)\bg\|_{W_\tas(\R)}\leq \beta,\\
\bg\|v^j_n-(\ld^j_n)^{-\frac{d}{2}}e^{-it|\xi^j_n|^2}e^{i\xi^j_n\cdot x}\psi^j_\beta\bg(\frac{t}{(\ld^j_n)^2}+t^j_n,\frac{x-x^j_n-2t\xi^j_n}{\ld^j_n}\bg)\bg\|_{W_\tas(\R)}\leq \beta.
\end{align}
Indeed, for $\ld_\infty^i,\ld_\infty^j=\infty$, this follows already from \eqref{L2 proxy 3}, while for $\ld_\infty^i,\ld_\infty^j=1$ we choose some $\psi^i_\beta,\psi^j_\beta\in C_c^\infty(\R\times\R^d)$ such that
\begin{align}
\|v^i-\psi^i_\beta\|_{W_\tas(\R)}\leq \beta,\,\|v^j-\psi^j_\beta\|_{W_\tas(\R)}\leq \beta
\end{align}
and the claim follows. Define
$$ \Lambda_n (\psi_\beta^i):=(\ld^i_n)^{-\frac{d}{2}}\psi^i_\beta\bg(\frac{t}{(\ld^i_n)^2}+t^i_n,\frac{x-x^i_n-2t\xi^i_n}{\ld^i_n}\bg).$$
Using H\"older we infer that
\begin{align*}
\|v_n^i v_n^j\|_{L_{t,x}^{\frac{d+2}{d}}(\R^d)}\lesssim \beta+\|\Lambda_n (\psi_\beta^i)\Lambda_n (\psi_\beta^j)\|_{L_{t,x}^{\frac{d+2}{d}}(\R^d)}.
\end{align*}
Since $\beta$ can be chosen arbitrarily small, it suffices to show
\begin{align}\label{step2a1}
\lim_{n\to\infty}\|\Lambda_n (\psi_\beta^i)\Lambda_n (\psi_\beta^j)\|_{L_{t,x}^{\frac{d+2}{d}}(\R^d)}=0.
\end{align}
Assume that $\frac{\ld_n^i}{\ld_n^j}+\frac{\ld_n^j}{\ld_n^i}\to \infty$. By symmetry we may w.l.o.g. assume that $\frac{\ld_n^i}{\ld_n^j}\to 0$. Using change of variables we obtain that
\begin{align}
&\,\|\Lambda_n (\psi_\beta^i)\Lambda_n (\psi_\beta^j)\|_{L_{t,x}^{\frac{d+2}{d}}(\R^d)}\nonumber\\
=&\,\bg(\frac{\ld_n^i}{\ld_n^j}\bg)^{\frac{d}{2}}\bg\|\psi_\beta^i(t,x)\psi_\beta^j\bg(
\bg(\frac{\ld_n^i}{\ld_n^j}\bg)^{2}t-\bg(\bg(\frac{\ld_n^i}{\ld_n^j}\bg)^{2}t_n^i-t_n^j\bg),
\nonumber\\
&\quad\quad\quad\quad
\bg(\frac{\ld_n^i}{\ld_n^j}\bg)x+2\bg(\frac{\ld_n^i}{\ld_n^j}\bg)\ld_n^i(\xi_n^i-\xi_n^j)t
+\frac{x_n^i-x_n^j-2t_n^i(\ld_n^i)^2(\xi_n^i-\xi_n^j)}{\ld_n^j}
 \bg)\bg\|_{L_{t,x}^{\frac{d+2}{d}}(\R^d)}\label{verylong2}\\
\lesssim&\,\bg(\frac{\ld_n^i}{\ld_n^j}\bg)^{\frac{d}{2}}\|\psi_\beta^i\|_{L_{t,x}^{\frac{d+2}{d}}(\R^d)}\|\psi_\beta^j\|_{L_{t,x}^\infty(\R^d)}
\to 0\nonumber.
\end{align}

Suppose therefore $\frac{\ld_n^i}{\ld_n^j}+\frac{\ld_n^j}{\ld_n^i}\to \ld_0\in(0,\infty)$. If $\bg(\frac{\ld_n^i}{\ld_n^j}\bg)^{2}t_n^i-t_n^j\to\pm\infty$, then by \eqref{verylong2} the supports of the integrands become disjoint in the temporal direction.

We may therefore further assume that $\bg(\frac{\ld_n^i}{\ld_n^j}\bg)^{2}t_n^i-t_n^j\to t_0\in\R$. If $\bg|\frac{x_n^i-x_n^j-2t_n^i(\ld_n^i)^2(\xi_n^i-\xi_n^j)}{\ld_n^j}\bg|\to\infty$ and $\xi_n^i= \xi_n^j$ for infinitely many $n$, then the supports of the integrands become disjoint in the spatial direction. If $\bg|\frac{x_n^i-x_n^j-2t_n^i(\ld_n^i)^2(\xi_n^i-\xi_n^j)}{\ld_n^j}\bg|\to\infty$ and $\xi_n^i\neq \xi_n^j$ for infinitely many $n$, then we apply the change of temporal variable $t\mapsto \frac{t}{\ld_n^i|\xi_n^i-\xi_n^j|}$ to see the decoupling of the supports of the integrands in the spatial direction.

Finally, if $\frac{x_n^i-x_n^j-2t_n^i(\ld_n^i)^2(\xi_n^i-\xi_n^j)}{\ld_n^j}\to x_0\in\R^d$, then by \eqref{orthog of pairs} we must have $\ld_n^i|\xi_n^i-\xi_n^j|\to\infty$. Hence for all $t\neq 0$ the integrand converges pointwise to zero. Using the dominated convergence theorem (setting $\|\psi_\beta^j\|_{L_{t,x}^\infty(\R)}\psi_\beta^i$ as the majorant) we finally conclude \eqref{step2a1}.

We now consider the remaining terms. For $\|v_n^i \nabla v_n^j\|_{L_{t,x}^{\frac{d+2}{d}}(\R)}$, arguing similarly as by \eqref{reduction to infinity case} and using \eqref{H1 proxy 2} we know that $\ld_\infty^i\in\{1,\infty\}$. For $\nabla v_n^j$, we use \eqref{L2 proxy 6} or \eqref{H1 proxy 4} as proxy for $\nabla v_n^j$, depending on the value of $\ld_\infty^j$; For $\|v_n^i v_n^j\|_{L_{t,x}^{\frac{d+2}{d-1}}(\R)}$, we first obtain that
\begin{align*}
\|v_n^i v_n^j\|_{L_{t,x}^{\frac{d+2}{d-1}}(\R)}\leq \min\bg\{\|v_n^i\|_{W_\tbs(\R)}^{\frac{d-2}{d-1}}
\|v_n^j\|_{W_\tas(\R)}^{\frac{1}{d-1}},
\|v_n^j\|_{W_\tbs(\R)}^{\frac{d-2}{d-1}}
\|v_n^i\|_{W_\tas(\R)}^{\frac{1}{d-1}}\bg\}.
\end{align*}
Therefore using \eqref{L2 proxy 2} and \eqref{H1 proxy 2} we can reduce the analysis to the case $\ld_\infty^i,\ld_\infty^j=1$; Finally, for $\|v_n^i \nabla v_n^j\|_{L_{t,x}^{\frac{d+2}{d-1}}(\R)}$ we can reduce our analysis to the case $\ld_\infty^i\in\{0,1\}$ and use \eqref{L2 proxy 6} or \eqref{H1 proxy 4} as proxy for $\nabla v_n^j$ and \eqref{H1 proxy 3} for $v_n^i$. Combining also with the boundedness of $(\xi_n^j)_n$, we can proceed as before to conclude the claim. We omit the details of the similar arguments. This completes the proof of Step 2a.

\subsubsection*{Step 2b: Smallness of $I_2$ and $I_3$}
We establish in this substep the asymptotic smallness of $I_2$ and $I_3$. Using H\"older and \eqref{elementary ineq} we obtain that
\begin{align}
&\||\nabla|^s (I_2+I_3)\|_{L_{t,x}^{\frac{2(d+2)}{d+4}}}\nonumber\\
\lesssim_k&
\left\{
             \begin{array}{ll}
             \|\Psi_n^k|\nabla|^s e^{it\Delta}w_n^k\|_{L_{t,x}^{\frac{d+2}{d}}(\R)}(\|\Psi_n^k\|^{\frac{4}{d}-1}_{W_\tas(\R)}
             +\|e^{it\Delta}w_n^k\|^{\frac{4}{d}-1}_{W_\tas(\R)})\\
             \quad\quad\quad\quad
             +\||\nabla|^s\Psi_n^k e^{it\Delta}w_n^k\|_{L_{t,x}^{\frac{d+2}{d}}(\R)}(\|\Psi_n^k\|^{\frac{4}{d}-1}_{W_\tas(\R)}
             +\|e^{it\Delta}w_n^k\|^{\frac{4}{d}-1}_{W_\tas(\R)})\\
             \quad\quad\quad\quad+\|\Psi_n^k|\nabla|^s e^{it\Delta}w_n^k\|_{L_{t,x}^{\frac{d+2}{d-1}}(\R)}(\|\Psi_n^k\|^{\frac{4}{d-2}-1}_{W_\tbs(\R)}
             +\|e^{it\Delta}w_n^k\|^{\frac{4}{d-2}-1}_{W_\tbs(\R)})\\
             \quad\quad\quad\quad+\||\nabla|^s\Psi_n^k e^{it\Delta}w_n^k\|_{L_{t,x}^{\frac{d+2}{d-1}}(\R)}(\|\Psi_n^k\|^{\frac{4}{d-2}-1}_{W_\tbs(\R)}
             +\|e^{it\Delta}w_n^k\|^{\frac{4}{d-2}-1}_{W_\tbs(\R)})\\
             \quad\quad\quad\quad+\|e^{it\Delta}w_n^k\|^{\frac{4}{d}}_{W_\tas(\R)}\||\nabla|^s e^{it\Delta}w_n^k\|_{W_\tas(\R)}\\
             \quad\quad\quad\quad+\|e^{it\Delta}w_n^k\|^{\frac{4}{d-2}}_{W_\tbs(\R)}\||\nabla|^s e^{it\Delta}w_n^k\|_{W_\tas(\R)},
             &\text{if $d=3$,}\\
             \\
             \|\Psi_n^k|\nabla|^s e^{it\Delta}w_n^k\|^{\frac{4}{d}}_{L_{t,x}^{\frac{d+2}{d}}(\R)}\||\nabla|^s e^{it\Delta}w_n^k\|^{1-\frac{4}{d}}_{W_\tas(\R)}\\
             \quad\quad\quad\quad
             +\||\nabla|^s\Psi_n^k e^{it\Delta}w_n^k\|^{\frac{4}{d}}_{L_{t,x}^{\frac{d+2}{d}}(\R)}\||\nabla|^s \Psi_n^k\|^{1-\frac{4}{d}}_{W_\tas(\R)}\\
             \quad\quad\quad\quad+\|\Psi_n^k|\nabla|^s e^{it\Delta}w_n^k\|_{L_{t,x}^{\frac{d+2}{d-1}}(\R)}(\|\Psi_n^k\|^{\frac{4}{d-2}-1}_{W_\tbs(\R)}
             +\|e^{it\Delta}w_n^k\|^{\frac{4}{d-2}-1}_{W_\tbs(\R)})\\
             \quad\quad\quad\quad+\||\nabla|^s\Psi_n^k e^{it\Delta}w_n^k\|_{L_{t,x}^{\frac{d+2}{d-1}}(\R)}(\|\Psi_n^k\|^{\frac{4}{d-2}-1}_{W_\tbs(\R)}
             +\|e^{it\Delta}w_n^k\|^{\frac{4}{d-2}-1}_{W_\tbs(\R)})\\
             \quad\quad\quad\quad+\|e^{it\Delta}w_n^k\|^{\frac{4}{d}}_{W_\tas(\R)}\||\nabla|^s e^{it\Delta}w_n^k\|_{W_\tas(\R)}\\
             \quad\quad\quad\quad+\|e^{it\Delta}w_n^k\|^{\frac{4}{d-2}}_{W_\tbs(\R)}\||\nabla|^s e^{it\Delta}w_n^k\|_{W_\tas(\R)},
             &\text{if $d\in\{4,5\}$,}\\
             \\
             \|\Psi_n^k|\nabla|^s e^{it\Delta}w_n^k\|^{\frac{4}{d}}_{L_{t,x}^{\frac{d+2}{d}}(\R)}\||\nabla|^s e^{it\Delta}w_n^k\|^{1-\frac{4}{d}}_{W_\tas(\R)}\\
             \quad\quad\quad\quad
             +\||\nabla|^s\Psi_n^k e^{it\Delta}w_n^k\|^{\frac{4}{d}}_{L_{t,x}^{\frac{d+2}{d}}(\R)}\||\nabla|^s\Psi_n^k\|^{1-\frac{4}{d}}_{W_\tas(\R)}\\
             \quad\quad\quad\quad+\|\Psi_n^k|\nabla|^s e^{it\Delta}w_n^k\|^{\frac{4}{d-2}}_{L_{t,x}^{\frac{d+2}{d-1}}(\R)}
             \||\nabla|^s e^{it\Delta}w_n^k\|^{1-\frac{4}{d-2}}_{W_\tbs(\R)}\\
             \quad\quad\quad\quad+\||\nabla|^s\Psi_n^k e^{it\Delta}w_n^k\|^{\frac{4}{d-2}}_{L_{t,x}^{\frac{d+2}{d-1}}(\R)}
             \||\nabla|^s \Psi_n^k\|^{1-\frac{4}{d-2}}_{W_\tbs(\R)}\\
             \quad\quad\quad\quad+\|e^{it\Delta}w_n^k\|^{\frac{4}{d}}_{W_\tas(\R)}\||\nabla|^s e^{it\Delta}w_n^k\|_{W_\tas(\R)}\\
             \quad\quad\quad\quad+\|e^{it\Delta}w_n^k\|^{\frac{4}{d-2}}_{W_\tbs(\R)}\||\nabla|^s e^{it\Delta}w_n^k\|_{W_\tas(\R)},
             &\text{if $d\geq6$.}
             \end{array}\label{verylong3}
\right.
\end{align}
In view of \eqref{to zero wnk}, \eqref{orthog L2 and H1}, Strichartz and \eqref{super uniform} it suffices to show that for $s\in\{0,1\}$
\begin{align}
\lim_{k\to K^*}\lim_{n\to\infty}&\bg(\|\Psi_n^k |\nabla|^s e^{it\Delta}w_n^k\|_{L_{t,x}^{\frac{d+2}{d}}(\R)}
+\||\nabla|^s\Psi_n^k e^{it\Delta}w_n^k\|_{L_{t,x}^{\frac{d+2}{d}}(\R)}\nonumber\\
&\quad\quad+\|\Psi_n^k |\nabla|^s e^{it\Delta}w_n^k\|_{L_{t,x}^{\frac{d+2}{d-1}}(\R)}
+\||\nabla|^s\Psi_n^k e^{it\Delta}w_n^k\|_{L_{t,x}^{\frac{d+2}{d-1}}(\R)}\bg)=0.
\end{align}
For $\||\nabla|^s\Psi_n^k e^{it\Delta}w_n^k\|_{L_{t,x}^{\frac{d+2}{d}}(\R)}$ and $\||\nabla|^s\Psi_n^k e^{it\Delta}w_n^k\|_{L_{t,x}^{\frac{d+2}{d-1}}(\R)}$, using H\"older, \eqref{super uniform}, Strichartz, \eqref{orthog L2 and H1} and \eqref{to zero wnk} we have
\begin{align}
&\lim_{k\to K^*}\lim_{n\to\infty}\bg(\||\nabla|^s\Psi_n^k e^{it\Delta}w_n^k\|_{L_{t,x}^{\frac{d+2}{d}}(\R)}
+\||\nabla|^s\Psi_n^k e^{it\Delta}w_n^k\|_{L_{t,x}^{\frac{d+2}{d-1}}(\R)}\bg)\nonumber\\
\lesssim &\lim_{k\to K^*}\lim_{n\to\infty}\bg(\||\nabla|^s\bg(\sum_{j=1}^k v_n^k\bg)\|_{W_\tas(\R)}\|e^{it\Delta}w_n^k\|_{W_\tas(\R)}
+\||\nabla|^s\bg(\sum_{j=1}^k v_n^k\bg)\|^{\frac{1}{d-2}}_{W_\tas(\R)}\|e^{it\Delta}w_n^k\|^{\frac{d-2}{d-1}}_{W_\tbs(\R)}\nonumber\\
&\quad\quad+\||\nabla|^se^{it\Delta}w_n^k\|_{W_\tas(\R)}\|e^{it\Delta}w_n^k\|_{W_\tas(\R)}
+\||\nabla|^se^{it\Delta}w_n^k \|^{\frac{1}{d-2}}_{W_\tas(\R)}\|e^{it\Delta}w_n^k\|^{\frac{d-2}{d-1}}_{W_\tbs(\R)}\bg)\nonumber\\
\lesssim &\lim_{k\to K^*}\lim_{n\to\infty}\bg(\bg(1+\|w_n^k\|_{H^1}+\|w_n^k\|^{\frac{1}{d-2}}_{H^1}\bg)\bg(\|e^{it\Delta}w_n^k\|_{W_\tas(\R)}
+\|e^{it\Delta}w_n^k\|^{\frac{d-2}{d-1}}_{W_\tbs(\R)}\bg)\bg)=0.
\end{align}
It is left to estimate $\|\Psi_n^k \nabla e^{it\Delta}w_n^k\|_{L_{t,x}^{\frac{d+2}{d}}(\R)}$ and $\|\Psi_n^k \nabla e^{it\Delta}w_n^k\|_{L_{t,x}^{\frac{d+2}{d-1}}(\R)}$. By \eqref{kkkk uniform bound of unj}, H\"older, Strichartz and \eqref{orthog L2 and H1} we know that for each $\eta>0$ there exists some $1\leq J'=J'(\eta)\leq K^*$ such that
\begin{align}
\sup_{J'\leq k\leq K^*}\lim_{n\to\infty}\bg(\|(\sum_{j=J'}^k v_n^j)\nabla e^{it\Delta}w_n^k\|_{L_{t,x}^{\frac{d+2}{d}}(\R)}
+\|(\sum_{j=J'}^k v_n^j)\nabla e^{it\Delta}w_n^k\|_{L_{t,x}^{\frac{d+2}{d-1}}(\R)}\bg)\lesssim \eta.
\end{align}
Hence, it suffices to show that
\begin{align}
\lim_{k\to K^*}\lim_{n\to\infty}\bg(\|v_n^j\nabla e^{it\Delta}w_n^k\|_{L_{t,x}^{\frac{d+2}{d}}(\R)}
+\|v_n^j\nabla e^{it\Delta}w_n^k\|_{L_{t,x}^{\frac{d+2}{d-1}}(\R)}\bg)=0.
\end{align}
for any $1\leq j< J'$. For $\|v_n^j\nabla e^{it\Delta}w_n^k\|_{L_{t,x}^{\frac{d+2}{d}}(\R)}$, using \eqref{H1 proxy 2} we may further assume that $\ld_\infty^j\in\{1,\infty\}$. For $\beta>0$, let $\phi_\beta\in C_c^\infty(\R\times \R^d)$ be given according to \eqref{L2 proxy 3}. Let $T,R>0$ such that $\mathrm{supp}\,\phi_\beta\subset [-T,T]\times \{|x|\leq R\}$. Then using H\"older we infer that
\begin{align}
\|v_n^j\nabla e^{it\Delta}w_n^k\|_{L_{t,x}^{\frac{d+2}{d}}(\R)}&\lesssim \beta\|\nabla e^{it\Delta}w_n^k\|_{W_\tas(\R)}
+\Lambda,
\end{align}
where
\begin{align}
\Lambda:&=\bg\|\phi_\beta(t,x)\,\bg((\ld_n^j)^{\frac{d}{2}}[ e^{it\Delta}\nabla w_n^k]\bg((\ld_n^j)^2t-(\ld_n^j)^2 t_n^j,\nonumber\\
&\quad\quad\quad\quad\ld_n^jx+2\xi_n^j(\ld_n^j)^2t+x_n^j-2\xi_n^j(\ld_n^j)^2t_n^j\bg)\bg)\bg\|_{L_{t,x}^{\frac{d+2}{d}}(\R)}\nonumber\\
&=\bg\|\phi_\beta(t,x)\,G_n^j\bg([e^{it\Delta}\nabla w_n^k](t,x+2\xi_n^jt)\bg)\bg\|_{L_{t,x}^{\frac{d+2}{d}}(\R)}
\end{align}
and
$$G_n^j u(t,x):=(\ld_n^j)^{\frac{d}{2}}u((\ld_n^j)^2(t-t_n^j),\ld_n^jx+x_n^j).$$
By the arbitrariness of $\beta$ it suffices to show the asymptotic smallness of $\Lambda$. Using the invariance of the NLS flow under Galilean transformation we know that
\begin{align}
&\,[e^{it\Delta}\nabla w_n^k](t,x+2\xi_n^jt)\nonumber\\
=&\,e^{i\xi_n^j\cdot x}e^{it|\xi_n^j|^2}\bg[e^{it\Delta}[e^{-i\xi_n^j\cdot x}\nabla w_n^k]\bg](t,x)\nonumber\\
=&\,e^{i\xi_n^j\cdot x}e^{it|\xi_n^j|^2}\bg[e^{it\Delta}[\nabla (e^{-i\xi_n^j\cdot x} w_n^k)]\bg](t,x)
+i\xi_n^je^{i\xi_n^j\cdot x}e^{it|\xi_n^j|^2}\bg[e^{it\Delta}[e^{-i\xi_n^j\cdot x} w_n^k]\bg](t,x)\nonumber\\
=&\,e^{i\xi_n^j\cdot x}e^{it|\xi_n^j|^2}\bg[\nabla \bg[e^{it\Delta}[ e^{-i\xi_n^j\cdot x} w_n^k]\bg]\bg](t,x)
+i\xi_n^j[e^{it\Delta}w_n^k](t,x+2\xi_n^jt)\nonumber\\
=:&\,e^{i\xi_n^j\cdot x}e^{it|\xi_n^j|^2}\Lambda_{1}+\Lambda_{2}.
\end{align}
Using H\"older, \eqref{to zero wnk} and the boundedness of $(\xi_n^j)_n$ we infer that
\begin{align}
\|\phi_\beta G_n^j(\Lambda_{2})\|_{L_{t,x}^{\frac{d+2}{d}}(\R)}
&\lesssim
\|\phi_\beta\|_{L_{t,x}^{\frac{2(d+2)}{d}}(\R)}\|\Lambda_{2}\|_{L_{t,x}^{\frac{2(d+2)}{d}}(\R)}\nonumber\\
&=|\xi_n^j|\|\phi_\beta\|_{L_{t,x}^{\frac{2(d+2)}{d}}(\R)}\|e^{it\Delta}w_n^k\|_{L_{t,x}^{\frac{2(d+2)}{d}}(\R)}=o_n(1).
\end{align}
Finally, using H\"older, the change of variables, \eqref{local killip visan 2} and the boundedness of $(\xi_n^j)_n$ we obtain that
\begin{align}
\|\phi_\beta G_n^j(e^{i\xi_n^j\cdot x}e^{it|\xi_n^j|^2}\Lambda_{2})\|_{L_{t,x}^{\frac{d+2}{d}}(\R)}
&\leq C(T,R)\|G_n^j(\Lambda_{2})\|_{L_{t,x}^2([-T,T]\times\{|x|\leq R\})}\nonumber\\
&\leq C(T,R)\|e^{it\Delta}w_n^k\|^{\frac{1}{3}}_{W_\tbs(\R)}\|e^{-i\xi_n^j\cdot x}w_n^k\|^{\frac{2}{3}}_{\dot{H}^1}\nonumber\\
&\leq C(T,R,\sup_{n}|\xi_n^j|)\|e^{it\Delta}w_n^k\|^{\frac{1}{3}}_{W_\tbs(\R)}\|w_n^k\|^{\frac{2}{3}}_{H^1}.
\end{align}
The claim then follows by invoking \eqref{to zero wnk} and \eqref{orthog L2 and H1}. For $d\geq 4$, $\|v_n^j\nabla e^{it\Delta}w_n^k\|_{L_{t,x}^{\frac{d+2}{d-1}}(\R)}$ can be estimated similarly as for $\|v_n^j\nabla e^{it\Delta}w_n^k\|_{L_{t,x}^{\frac{d+2}{d}}(\R)}$. In this case we can further assume that $\ld_\infty^j\in\{0,1\}$ and $\xi_n^j\equiv 0$ (which also holds for $d=3$) and the proof is in fact much easier, we therefore omit the details here. For $d=3$, we notice that $\frac{d+2}{d-1}>2$ and hence we will use the interpolation estimate
\begin{align}
\|\phi_\beta\nabla\tilde{w}_n^k\|_{L_{t,x}^{\frac{5}{2}}(\R)}\lesssim C(T,R)\|\nabla \tilde{w}^k_n\|^{\frac{1}{2}}_{W_\tas}
\|\nabla \tilde{w}^k_n\|^{\frac{1}{2}}_{L^2_{t,x}([-T,T]\times\{|x|\leq R\})}
\end{align}
in order to apply \eqref{local killip visan 2}, where $\phi_\beta$ is deduced from \eqref{H1 proxy 3} and $\tilde{w}_n^k:=\ld_n^j G_n^j w_n^k$. This completes the proof of Step 2b and thus also the desired proof of Step 2.
\subsubsection*{Step 3: Reduction to one bad profile and conclusion.}
From Step 2 we conclude that there exists some $1\leq J_1\leq K^*$ such that
\begin{align}
\|v_n^j\|_{W_\tas\cap W_\tbs(\R)}&=\infty\quad \forall \,1\leq j\leq J_1,\label{infinite}\\
\|v_n^j\|_{W_\tas\cap W_\tbs(\R)}&<\infty\quad \forall \,J_1+1\leq j\leq K^*.
\end{align}
By Lemma \ref{persistance l2} and Lemma \ref{persistance h1} we deduce that $\ld_\infty^j=1$ for all $1\leq j\leq J_1$. If $J_1>1$, then using \eqref{mo sum}, \eqref{eo sum} and Lemma \ref{killip visan curve} (iv) we know that $\mD^*(v_n^i),\mD^*(v_n^j)<\mD^*$, which violates \eqref{infinite} due to the inductive hypothesis. Thus $J_1=1$ and
$$ \psi_n(0,x)=e^{it_n^1 \Delta}\tdu^1(x-x_n^1)+w_n^1(x).$$
In particular, $\tdu^1\in H^1(\R^d)$. Similarly, we must have $\mM(w_n^1)=o_n(1)$ and $\mH(w_n^1)=o_n(1)$, otherwise we could deduce the contradiction \eqref{contradiction 1} using Lemma \ref{long time pert}. Combining with Lemma \ref{killip visan curve} (v) we conclude that $\|w_n^1\|_{H^1}=o_n(1)$. Finally, we exclude the case $t^1_n\to\pm \infty$. We only consider the case $t_n^1\to \infty$, the case $t_n^1 \to -\infty$ can be similarly dealt. Indeed, using Strichartz we obtain that
\begin{align}
\|e^{it\Delta}\psi_n(0)\|_{W_{\tas}\cap W_\tbs([0,\infty))}\lesssim
\|e^{it\Delta}\tdu^1\|_{W_{\tas}\cap W_\tbs([t_n,\infty))}+\|w_n^1\|_{H^1}\to 0
\end{align}
and using Lemma \ref{well posedness lemma} we deduce the contradiction \eqref{contradiction 1} again. This completes the desired proof.
\end{proof}

\begin{lemma}[Existence of the minimal blow-up solution]\label{category 0 and 1}
Suppose that $\mD^*<\infty$. Then there exists a global solution $u_c$ of \eqref{NLS} such that $\mD(u_c)=\mD^*$ and
\begin{align}
\|u_c\|_{W_\tas\cap W_\tbs((-\infty,0])}=\|u_c\|_{W_\tas\cap W_\tbs([0,\infty))}=\infty.
\end{align}
Moreover, $u_c$ is almost periodic in $H^1(\R^d)$ modulo translations, i.e. the set $\{u(t):t\in\R\}$ is precompact in $H^1(\R^d)$ modulo translations.
\end{lemma}

\begin{proof}
As discussed at the beginning of this section, under the assumption $\mD^*<\infty$ one can find a sequence such that \eqref{oo1} and \eqref{oo2} hold. We apply Lemma \ref{Palais Smale} to the sequence $(\psi_n(0))_n$ to infer that $(\psi_n(0))_n$ (up to modifying time and space translation) is precompact in $H^1(\R^d)$. We denote its strong $H^1$-limit by $\psi$. Let $u_c$ be the solution of \eqref{NLS} with $u_c(0)=\psi$. Then $\mD(u_c(t))=\mD(\psi)=\mD^*$ for all $t$ in the maximal lifespan $I_{\max}$ of $u_c$ (recall that $\mD$ is a conserved quantity by Lemma \ref{killip visan curve}).

We first show that $u_c$ is a global solution. We only show that $s_0:=\sup I_{\max}=\infty$, the negative direction can be similarly proved. If this does not hold, then by Lemma \ref{well posedness lemma} there exists a sequence $(s_n)_n\subset \R$ with $s_n\to s_0$ such that
\begin{align*}
\limsup_{n\to\infty}\|u_c\|_{W_\tas\cap W_\tbs([s_n,s_0))}=\infty.
\end{align*}
Define $\psi_n(t):=u_c(t+s_n)$. Then \eqref{precondition} is satisfied with $t_n\equiv 0$. We then apply Lemma \ref{Palais Smale} to the sequence $(\psi_n(0))_n$ to infer that there exists some $\varphi\in H^1(\R^d)$ such that, up to modifying the space translation, $u_c(s_n)$ strongly converges to $\varphi$ in $H^1(\R^d)$. But then using Strichartz we obtain
\begin{align*}
\|e^{it\Delta}u_c(s_n)\|_{W_\tas\cap W_\tbs([s_n,s_0))}=\|e^{it\Delta}\varphi\|_{W_\tas\cap W_\tbs([s_n,s_0))}+o_n(1)=o_n(1).
\end{align*}
By Lemma \ref{well posedness lemma} we can extend $u_c$ beyond $s_0$, which contradicts the maximality of $s_0$. Now by \eqref{oo1} and Lemma \ref{long time pert} it is necessary that
\begin{align}\label{blow up uc}
\|u_c\|_{W_\tas\cap W_\tbs((-\infty,0])}=\|u_c\|_{W_\tas\cap W_\tbs([0,\infty))}=\infty.
\end{align}

We finally show that the orbit $\{u_c(t):t\in\R\}$ is precompact in $H^1(\R^d)$ modulo translations. Let $(\tau_n)_n\subset\R$ be an arbitrary time sequence. Then \eqref{blow up uc} implies
\begin{align*}
\|u_c\|_{W_\tas\cap W_\tbs((-\infty,\tau_n])}=\|u_c\|_{W_\tas\cap W_\tbs([\tau_n,\infty))}=\infty.
\end{align*}
The claim follows by applying Lemma \ref{Palais Smale} to $(u_c(\tau_n))_n$.
\end{proof}

\subsection{Extinction of the minimal blow-up solution}\label{Extinction of the minimal blow-up solution}
The following lemma is an immediate consequence of the fact that $u_c$ is almost periodic in $H^1(\R^d)$ and conservation of momentum. The proof is standard, we refer to \cite{non_radial} for the details of the proof.
\begin{lemma}\label{holmer}
Let $u_c$ be the minimal blow-up solution given by Lemma \ref{category 0 and 1}. Then there exists some function $x:\R\to\R^d$ such that
\begin{itemize}

\item[(i)] For each $\vare>0$, there exists $R>0$ so that
\begin{align}
\int_{|x+x(t)|\geq R}|\nabla u_c(t)|^2+|u_c(t)|^2+|u_c|^\tas+|u_c|^\tbs\,dx\leq\vare\quad\forall\,t\in\R.
\end{align}


\item[(ii)] The center function $x(t)$ obeys the decay condition $x(t)=o(t)$ as $|t|\to\infty$.
\end{itemize}
\end{lemma}

\begin{proof}[Proof of Theorem \ref{main theorem} for the focusing-focusing regime]
We will show the contradiction that the minimal blow-up solution $u_c$ given by Lemma \ref{category 0 and 1} is equal to zero, which will finally imply Theorem \ref{main theorem} for the focusing-focusing case. Let $\chi$ be a smooth radial cut-off function satisfying
\begin{align*}
\chi=\left\{
             \begin{array}{ll}
             |x|^2,&\text{if $|x|\leq 1$},\\
             0,&\text{if $|x|\geq 2$}.
             \end{array}
\right.
\end{align*}
Define also the local virial function
\begin{align*}
z_{R}(t):=\int R^2\chi\bg(\frac{x}{R}\bg)|u_c(t,x)|^2\,dx.
\end{align*}
Direct calculation yields
\begin{align}
\pt_t z_R(t)=&\,2\,\mathrm{Im}\int R\nabla \chi\bg(\frac{x}{R}\bg)\cdot\nabla u_c(t)\bar{u}_c(t)\,dx,\label{final4}\\
\pt_{tt} z_R(t)=&\,4\int \pt^2_{jk}\chi\bg(\frac{x}{R}\bg)\pt_j u_c\pt_k\bar{u}_c-\frac{1}{R^2}\int\Delta^2\chi\bg(\frac{x}{R}\bg)|u_c|^2\nonumber\\
&\,-\frac{4}{d+2}\int\Delta\chi\bg(\frac{x}{R}\bg)|u_c|^\tas\,dx-\frac{4}{d}\int\Delta\chi\bg(\frac{x}{R}\bg)|u_c|^\tbs\,dx.
\end{align}
We then obtain that
\begin{align}\label{final 2}
\pt_{tt} z_R(t)=8\mK(u_c)+A_R(u_c(t)),
\end{align}
where
\begin{align*}
A_R(u_c(t))=&\,4\int\bg(\pt_{jj}\chi\bg(\frac{x}{R}\bg)-2\bg)|\pt_j u_c|^2+4\sum_{j\neq k}\int_{R\leq|x|\leq 2R}\pt_{jk}\chi\bg(\frac{x}{R}\bg)\pt_j u\pt_k\bar{u}_c\nonumber\\
&\,-\frac{1}{R^2}\int\Delta^2\chi\bg(\frac{x}{R}\bg)|u_c|^2
-\frac{4}{d+2}\int\bg(\Delta\chi\bg(\frac{x}{R}\bg)-2d\bg)|u_c|^\tas\,dx\nonumber\\
&\,-\frac{4}{d}\int\bg(\Delta\chi\bg(\frac{x}{R}\bg)-2d\bg)|u_c|^\tbs\,dx.
\end{align*}
We thus infer the estimate
\begin{align*}
|A_R(u(t))|\leq C_1\int_{|x|\geq R}|\nabla u(t)|^2+\frac{1}{R^2}|u(t)|^2+|u|^\tas+|u|^\tbs
\end{align*}
for some $C_1>0$. Assume that $\mM(u_c)=(1-\delta)^{\frac{d}{2}}\mM(Q)$ for some $\delta\in(0,1)$. Using \eqref{lower bound kt} we deduce that
\begin{align}\label{small extinction ff}
\mK(u_c(t))\geq\min\bg\{\frac{4\delta}{d}\mH(u(0)),
\bg(\bg(\frac{d}{\delta(d-2)}\bg)^{\frac{d-2}{4}}-1\bg)^{-1}\bg(m_{\mM(u(0))}-\mH(u(0))\bg)\bg\}
=: \frac{\eta_1}{4}
\end{align}
for all $t\in\R$. From Lemma \ref{holmer} it follows that there exists some $R_0\geq 1$ such that
\begin{align*}
\int_{|x+x(t)|}|\nabla u_c|^2+|u_c|^2+|u|^\tas+|u|^\tbs\,dx\leq \frac{\eta}{C_1}.
\end{align*}
Thus for any $R\geq R_0+\sup_{t\in[t_0,t_1]}|x(t)|$ with some to be determined $t_0,t_1\in[0,\infty)$, we have
\begin{align}\label{final3}
\pt_{tt} z_R(t)\geq \eta_1
\end{align}
for all $t\in[t_0,t_1]$. By Lemma \ref{holmer} we can choose $t_0$ sufficiently large such that there exists some $\eta_2$ to be determined later (and can be chosen sufficiently small) such that $|x(t)|\leq\eta_2 t$ for all $t\geq t_0$. Now set $R=R_0+\eta_2 t_1$. Integrating \eqref{final3} over $[t_0,t_1]$ yields
\begin{align}\label{12}
\pt_t z_R(t_1)-\pt_t z_R(t_0)\geq \eta_1 (t_1-t_0).
\end{align}
Using \eqref{final4}, Cauchy-Schwarz and Lemma \ref{killip visan curve} we have
\begin{align}\label{13}
|\pt_t z_{R}(t)|\leq C_2 \mD^*R= C_2 \mD^*(R_0+\eta_2 t_1)
\end{align}
for some $C_2=C_2(\mD^*)>0$. \eqref{12} and \eqref{13} give us
\begin{align*}
2C_2 \mD^*(R_0+\eta_2 t_1)\geq\eta_1 (t_1-t_0).
\end{align*}
Setting $\eta_2=\frac{1}{4C_2\mD^*}$ and then sending $t_1$ to infinity we will obtain a contradiction unless $\eta_1=0$, which implies $\mH_0=\mH(u_c)=0$. From Lemma \ref{killip visan curve} we know that $\nabla u_c=0$, which implies $u_c=0$. This completes the proof.
\end{proof}

\section{Scattering threshold for the focusing-defocusing \eqref{NLS double crit}} \label{section d f}
In this Section we prove Theorem \ref{main theorem} for the defocusing-focusing model and Proposition \ref{proposition mc for df}. Throughout the section, we assume that \eqref{NLS double crit} reduces to
\begin{align}\label{NLSdf}
i\pt u+\Delta u -|u|^{\frac{4}{d}}u+|u|^{\frac{4}{d-2}}u=0
\end{align}
We also define the set $\mA$ by
\begin{align*}
\mA:=\{u\in H^1(\R^d):\mH(u)<\mH^*(W),\,\mK(u)>0\}.
\end{align*}

\subsection{Variational formulation for $m_c$}
\begin{lemma}\label{mtilde equal m df}
The following statements hold true:
\begin{itemize}
\item[(i)]Let $u\in H^1(\R^d)\setminus\{0\}$. Then there exists a unique $\ld(u)>0$ such that
\begin{equation*}
\mK(T_\ld u)\left\{
             \begin{array}{ll}
             >0, &\text{if $\ld\in(0,\ld(u))$},  \\
             =0,&\text{if $\ld=\ld(u)$},\\
             <0,&\text{if $\ld\in(\ld(u),\infty)$}.
             \end{array}
\right.
\end{equation*}

\item[(ii)] The mapping $c\mapsto m_c$ is continuous and monotone decreasing on $(0,\infty)$.

\item[(iii)]Let
\begin{align*}
\tilde{m}_{c}&:=\inf_{u\in H^1(\R^d)}\{\mI(u):\mM(u)= c,\,\mK(u)\leq 0\}.
\end{align*}
Then $m_{c}=\tm_c$ for any $c\in(0,\infty)$.
\end{itemize}
\end{lemma}

\begin{proof}
This is a straightforward modification of Lemma \ref{positive or negative k}, Lemma \ref{monotone lemma} and Lemma \ref{mtilde equal m}, we therefore omit the details here.
\end{proof}

\begin{lemma}\label{no hat equal hat}
Let $\mK^c(u):=\|\nabla u\|_2^2-\|u\|_\tbs^\tbs$ and
\begin{align}
\widehat{m}_c:=\inf_{u\in H^1(\R^d)}\{\mI(u):\mM(u)=c,\,\mK^c(u)\leq 0\}.
\end{align}
Then $m_c=\hat{m}_c$ for any $c\in(0,\infty)$.
\end{lemma}

\begin{proof}
If $\mM(u)=c$ and $\mK(u)=0$, then it is clear that $\mK^c(u)<0$ and $\mH(u)=\mI(u)$, which implies $m_c\geq \widehat{m}_c$. For the inverse direction, in view of Lemma \ref{mtilde equal m df}, it suffices to show $\tilde{m}_c\leq \widehat{m}_c$. By Lemma \ref{mtilde equal m df} we can further define $\tilde{m}_c$ by
\begin{align}\label{charac m tilde c}
\tilde{m}_{c}&=\inf_{u\in H^1(\R^d)}\{\mI(u):\mM(u)\in(0,c],\,\mK(u)\leq 0\}.
\end{align}
Assume that $u\in H^1(\R^d)$ with $\mM(u)=c$ and $\mK^c(u)\leq 0$. Then
\begin{align}
\frac{d}{dt}\mK^c(T_t u)\bg|_{t=1}=2\mK^c(u)-\frac{4}{d-2}\|u\|_\tbs^\tbs<0.
\end{align}
Hence there exists some sufficiently small $\delta>0$ such that $\mK^c(T_t u)<0$ for all $t\in(1,1+\delta)$. In particular,
\begin{align*}
\mI(T_t u)\to\mI(u),\, \mK^c(T_t u)\to \mK^c(u)\quad\text{as $t\downarrow 1$}.
\end{align*}
We now define
$$ U_\ld u(x):=\ld^{\frac{d-2}{2}}u(\ld x).$$
Then $\mK^c(U_\ld u)=\mK^c(u)$ and $\mI(U_\ld u)=\mI(u)$ for any $\ld>0$. Moreover,
\begin{align}
\mK(U_\ld u)&=\mK^c(u)+\frac{2\ld^{-\frac{4}{d}}}{d+2}\|u\|_\tas^\tas\to \mK^c(u),\\
\mM(U_\ld u)&=\ld^{-2}\mM(u)\to 0\label{smallness uldu}
\end{align}
as $\ld\to\infty$. Let $\vare>0$ be an arbitrary positive number. We can then find some $t>1$ sufficiently close to $1$ such that
\begin{align*}
|\mI(T_t u)-\mI(u)|\leq \vare.
\end{align*}
Moreover, we can further find some sufficiently large $\ld=\ld(t)$ such that $\mK(U_\ld T_t u)<0$. Then by \eqref{charac m tilde c} and \eqref{smallness uldu} we infer that
\begin{align*}
\mI(u)\geq \mI(U_\ld T_t u)-\vare\geq \tilde{m}_c-\vare.
\end{align*}
The claim follows by the arbitrariness of $u$ and $\vare$.
\end{proof}

\begin{proof}[Proof of Proposition \ref{proposition mc for df}]
Let $c>0$ and let $u_\vare\in C_c^\infty(\R^d)$ with $\|u_\vare-W\|_{\dot{H}^1}\leq \vare$ for some given small $\vare>0$. We define
$$ v_\vare:=\sqrt{\frac{c}{\mM(u_\vare)}}\,u_\vare,$$
Then $\mM(v_\vare)=c$. Let $t_\vare\in(0,\infty)$ be given such that $ \mK^c(T_{t_\vare}v_\vare)=0$. Direct calculation yields
\begin{align}
t_\vare=\bg(\frac{\|\nabla v_\vare\|_2^2}{\|v_\vare\|_\tbs^\tbs}\bg)^{\frac{d-2}{4}}.
\end{align}
By Lemma \ref{no hat equal hat} we have
\begin{align}\label{complicated}
m_c\leq \mI(T_{t_\vare}v_\vare)=\frac{1}{d}\bg(\frac{\|\nabla v_\vare\|_2^2}{\|v_\vare\|_\tbs^2}\bg)^{\frac{d}{2}}
=\frac{1}{d}\bg(\frac{\|\nabla u_\vare\|_2^2}{\|u_\vare\|_\tbs^2}\bg)^{\frac{d}{2}}.
\end{align}
Taking $\vare\to 0$ we immediately conclude that $m_c\leq \frac{1}{d}\cdot\bg(\frac{\|\nabla W\|_2^2}{\|W\|_\tbs^2}\bg)^{\frac{d}{2}}=\mH^*(W).$ On the other hand, one easily verifies that
$$ \mK^c(u)\leq 0\Rightarrow\mI(u)\geq \bg(\frac{\|\nabla u\|_2^2}{\|u\|_\tbs^2}\bg)^{\frac{d}{2}}.$$
But by Sobolev inequality we always have $\bg(\frac{\|\nabla u\|_2^2}{\|u\|_\tbs^2}\bg)^{\frac{d}{2}}\geq\mS^{\frac{d}{2}}=d\mH^*(W)$. Hence $m_c=\mH^*(W)$. By \cite[Thm. 1.2]{SoaveCritical}, any optimizer $P$ of $m_c$ must satisfy $\mH(P)>\mH^*(W)$, which is a contradiction. This completes the proof of Proposition \ref{proposition mc for df}.
\end{proof}

\subsection{Scattering for the defocusing-focusing \eqref{NLS double crit}}
In this section we establish similar variational estimates as the ones given in Section \ref{variational arguments ff}. The scattering result then follows from the variational estimates by using the arguments given in Section \ref{Existence of the minimal blow-up solution} and \ref{Extinction of the minimal blow-up solution} verbatim.

\begin{lemma}\label{bound of gradient by energy df}
The following statements hold true:
\begin{itemize}
\item[(i)]Assume that $\mK(u)\geq 0$. Then $\mH(u)\geq 0$. If additionally $\mK(u)> 0$, then also $\mH(u)> 0$.

\item[(ii)]Let $u\in \mA$. Then
\begin{align}
\|u\|_\tbs^\tbs&\leq \|\nabla u\|_2^2+\frac{d}{d+2}\|u\|_\tas^\tas,\label{bound of tbs by 2 df}\\
\frac{1}{d}\bg(\|\nabla u\|_2^2+\frac{d}{d+2}\|u\|_\tas^\tas\bg)&\leq \mH(u)\leq \frac{1}{2}\bg(\|\nabla u\|_2^2+\frac{d}{d+2}\|u\|_\tas^\tas\bg).\label{third df}
\end{align}

\item[(iii)]Let $u$ be a solution of \eqref{NLSdf} with $u(0)\in {\mA}$. Then $u(t)\in {\mA}$ for all $t$ in the maximal lifespan. Moreover, we have
\begin{align}\label{lower bound kt df}
&\,\inf_{t\in I_{\max}}\mK(u(t))\nonumber\\
\geq&\,\min\bg\{\frac{4}{d}\mH(u(0)),
\bg(\bg(\frac{d}{d-2}\bg)^{\frac{d-2}{4}}-1\bg)^{-1}\bg(\mH^*(W)-\mH(u(0))\bg)\bg\}.
\end{align}
\end{itemize}
\end{lemma}

\begin{proof}
This is a straightforward modification of Lemma \ref{pos of k implies pos of h}, Lemma \ref{bound of gradient by energy} and Lemma \ref{invariant lemma}, we therefore omit the details here.
\end{proof}

We now define the MEI-functional for \eqref{NLSdf}. Let $\Omega:=\R^2\setminus([0,\infty)\times [\mH^*(W),\infty))$ and let the MEI-functional $\mD$ be given by \eqref{MEI functional}. One has the following analogue of Lemma \ref{killip visan curve}.

\begin{lemma}\label{killip visan curve df}
Assume $v\in H^1(\R^d)$ such that $\mK(v)\geq 0$. Then
\begin{itemize}
\item[(i)]$\mD(v)=0$ if and only if $v=0$.
\item[(ii)] $0<\mD(v)<\infty$ if and only if $v\in\mA$.
\item[(iii)] $\mD$ leaves $\mA$ invariant under the NLS flow.
\item[(iv)] Let $u_1,u_2\in\mA$ with $\mM(u_1)\leq \mM(u_2)$ and $\mH(u_1)\leq \mH(u_2)$, then $\mD(u_1)\leq \mD(u_2)$. If in addition either $\mM(u_1)<\mM(u_2)$ or $\mH(u_1)<\mH(u_2)$, then $\mD(u_1)<\mD(u_2)$.
\item[(v)] Let $\mD_0\in(0,\infty)$. Then
\begin{align}
\|\nabla u\|^2_{2}&\sim_{\mD_0}\mH(u),\label{df 1}\\
\|u\|^2_{H^1}&\sim_{\mD_0}\mH(u)+\mM(u)\sim_{\mD_0}\mD(u)\label{df 2}
\end{align}
uniformly for all $u\in \mA$ with $\mD(u)\leq \mD_0$.
\item[(vi)] For all $u\in \mA$ with $\mD(u)\leq \mD_0$ with $\mD_0\in(0,\infty)$we have
\begin{align}\label{small of unaaa df}
|\mH(u)-\mH^*(W)|\gtrsim 1.
\end{align}
\end{itemize}
\end{lemma}

\begin{proof}
(i) to (iv) can be similarly proved as the ones from Lemma \ref{killip visan curve}, we omit the details here.

Next we verify (v). Let $u\in\mA$ with $\mD(u)\leq\mD_0$. Using \eqref{third df} we already have $\|\nabla u\|_2^2\leq d\mH(u)$. On the other hand, by the definition of $\mD$ it is readily to see that
\begin{align}\label{df 0}
\mD_0\geq \mD(u)=\mH(u)+\frac{\mH(u)+\mM(u)}{\mH^*(W)-\mH(u)}\geq \frac{\mM(u)}{\mH^*(W)},
\end{align}
which implies $\mM(u)\leq\mD_0\mH^*(W)$. Using Gagliardo-Nirenberg we infer that
\begin{align}
\frac{d}{d+2}\|u\|_\tas^\tas\leq \bg(\frac{\mM(u)}{\mM(Q)}\bg)^{\frac{2}{d}}\|\nabla u\|_2^2\leq
\bg(\frac{\mD_0\mH^*(W)}{\mM(Q)}\bg)^{\frac{2}{d}}\|\nabla u\|_2^2
\end{align}
Applying \eqref{third df} one more time we conclude that
\begin{align}\label{df 3}
\mH(u)\leq \frac{1}{2}\bg(\|\nabla u\|_2^2+\frac{d}{d+2}\|u\|_\tas^\tas\bg)\leq
\frac{1}{2}\bg(1+\bg(\frac{\mD_0\mH^*(W)}{\mM(Q)}\bg)^{\frac{2}{d}}\bg)\|\nabla u\|_2^2
\end{align}
and \eqref{df 1} and the first equivalence of \eqref{df 2} follow. From \eqref{df 0} it also follows $\mH(u)+\mM(u)\lesssim_{\mD_0} \mD(u)$. To prove the inverse direction, we first obtain that
\begin{align*}
\mD_0\geq \mD(u)=\mH(u)+\frac{\mH(u)+\mM(u)}{\mH^*(W)-\mH(u)}\geq \frac{\mH(u)}{\mH^*(W)-\mH(u)},
\end{align*}
which implies $\mH(u)\leq(1+\mD_0)^{-1}\mD_0\mH^*(W)$. Then
\begin{align*}
\mD(u)&= \mH(u)+\frac{\mH(u)+\mM(u)}{\mH^*(W)-\mH(u)}\leq \mH(u)+\frac{\mH(u)+\mM(u)}{(1-(1+\mD_0)^{-1}\mD_0)\mH^*(W)}\nonumber\\
&=\mH(u)+\frac{(1+\mD_0)(\mH(u)+\mM(u))}{\mH^*(W)},
\end{align*}
which finishes the proof of (v). For (vi), if this were not the case, then we could find a sequence $(u_n)_n\subset\mA$ such that
\begin{align}\label{small of un df}
\mH^*(W)-\mH(u_n)=o_n(1).
\end{align}
Then \eqref{small of un df} implies $\mH(u_n)\gtrsim 1$ and therefore
$$ \mD(u_n)\geq \frac{\mH(u_n)}{\mH^*(W)-\mH(u_n)}\to\infty,$$
which is a contradiction to $\mD(u_n)\leq \mD_0$. This completes the proof of (vi) and also the desired proof of Lemma \ref{killip visan curve df}.
\end{proof}

\begin{proof}[Proof of Theorem \ref{main theorem} for the defocusing-focusing regime]
The proof is almost identical to the one for the focusing-focusing regime, one only needs to replace the results from \cite{Dodson4dmassfocusing} applied in Lemma \ref{persistance l2} by the ones from \cite{dodson1d,dodson2d,dodson3d}, the arguments from Lemma \ref{killip visan curve} by the ones from Lemma \ref{killip visan curve df} and \eqref{small extinction ff} by \eqref{lower bound kt df}. We therefore omit the details here.
\end{proof}

\section{Scattering threshold, existence and non-existence of ground states for the focusing-defocusing \eqref{NLS double crit}} \label{section f d}
In this Section we prove Theorem \ref{main theorem} for the focusing-defocusing model and Proposition \ref{proposition for ground state}. Throughout the section, we assume that \eqref{NLS double crit} reduces to
\begin{align}\label{NLSfd}
i\pt u+\Delta u +|u|^{\frac{4}{d}}u-|u|^{\frac{4}{d-2}}u=0
\end{align}
The corresponding stationary equation reads
\begin{align}\label{standing wave eq fd}
-\Delta u+\omega u-|u|^{\frac{4}{d}}u+|u|^{\frac{4}{d-2}}u=0.
\end{align}
We also define the set $\mA$ by
\begin{align*}
\mA:=\{u\in H^1(\R^d):0<\mM(u)<\mM(Q)\}.
\end{align*}

\subsection{Monotonicity formulae and nonexistence of minimizers for $c\leq \mM(Q)$}
\begin{lemma}\label{pohozaev}
Suppose that $u$ is a solution of \eqref{standing wave eq fd}. Then
\begin{align}
0=&\|\nabla u\|_2^2+\omega\|u\|_2^2-\|u\|_\tas^\tas+\|u\|_\tbs^\tbs,\label{poho1}\\
0=&\|\nabla u\|_2^2+\frac{d}{d-2}\omega\|u\|_2^2-\frac{d^2}{d^2-4}\|u\|_\tas^\tas+\|u\|_\tbs^\tbs,\label{poho2}
\end{align}
and
\begin{align}
\omega\|u\|_2^2=\frac{2}{d+2}\|u\|_\tas^\tas.\label{poho3}
\end{align}
Moreover, if $u\neq 0$, then $\omega\in(0,\frac{2}{d}\bg(\frac{d}{d+2}\bg)^{\frac{d}{2}})$.
\end{lemma}

\begin{proof}
\eqref{poho1} follows from multiplying \eqref{standing wave eq fd} with $\bar{u}$ and then integrating by parts. \eqref{poho2} is the Pohozaev inequality, see for instance \cite{Berestycki1983}. \eqref{poho3} follows immediately from \eqref{poho1} and \eqref{poho2}. That $\omega>0$ for $u\neq 0$ follows directly from \eqref{poho3}. To see $\omega<\frac{2}{d}\bg(\frac{d}{d+2}\bg)^{\frac{d}{2}}$, one can easily check this by using the fact that the polynomial
\begin{align*}
t^{\frac{4}{d-2}}-\frac{d^2}{d^2-4}t^{\frac{4}{d}}+\frac{d}{d-2}\omega
\end{align*}
is non-negative for $\omega\geq \frac{2}{d}\bg(\frac{d}{d+2}\bg)^{\frac{d}{2}}$.
\end{proof}

\begin{lemma}
The mapping $c\to\gamma_c$ is non-positive on $(0,\infty)$ and equal to zero on $(0,\mM(Q)]$. Consequently, $\gamma_c$ has no minimizer for any $c\in (0,\mM(Q)]$.
\end{lemma}

\begin{proof}
First we obtain that
\begin{align*}
\mH(T_\ld u)&=\frac{\ld^2}{2} \bg(\|\nabla u\|_2^2-\frac{d}{d+2} \|u\|_\tas^\tas\bg)+\frac{\ld^{\tbs}}{\tbs}\|u\|_\tbs^\tbs.
\end{align*}
By sending $\ld\to 0$ we see that $\gamma_c\leq 0$. On the other hand, using \eqref{GN-L2} we infer that
\begin{align*}
\mH(u)\geq \frac{1}{2}\bg(1-\bg(\frac{\mM(u)}{{\mM(Q)}}\bg)^{\frac{2}{d}}\bg)\|\nabla u\|_2^2+\frac{\ld^{\tbs}}{\tbs}\|u\|_\tbs^\tbs\geq 0
\end{align*}
for $\mM(u)\in(0, \mM(Q)]$. In particular, since $\bg(1-\bg(\frac{\mM(u)}{{\mM(Q)}}\bg)^{\frac{2}{d}}\bg)$ is non-negative for $\mM(u)\in(0, \mM(Q)]$, we deduce that $\mH(u)=0$ is only possible when $u=0$, which is a contradiction since $\mM(u)>0$. Thus there is no minimizer for $\gamma_c$ when $c\in(0,\mM(Q)]$.
\end{proof}

\begin{lemma}
The mapping $c\mapsto \gamma_c$ is monotone decreasing and $\gamma_c>-\infty$ on $(0,\infty)$. Moreover, $\gamma_c$ is negative on $(\mM(Q),\infty)$.
\end{lemma}

\begin{proof}
We define the scaling operator $U_\ld$ by
\begin{align*}
U_\ld u(x):=\ld^{\frac{d-2}{2}}u(\ld x).
\end{align*}
Then
\begin{align*}
\mH(U_\ld u)&=\mH(u)+\frac{1}{\tas}(1-\ld^{-\frac{4}{d}})\|u\|_\tas^\tas,\\
\mM(U_\ld u)&=\ld^{-2}\mM(u).
\end{align*}
For $u\neq 0$ we see that $\mH(U_\ld u)\to-\infty$ and $\mM(U_\ld u)\to \infty$ as $\ld\to 0$, which implies that $\gamma_c<0$ for large $c$. Next we show the monotonicity of $c\mapsto \gamma_c$. Let $0< c_1<c_2<\infty$. By definition of $\gamma_{c_1}$ there exists a sequence $(u_n)_n\subset H^1(\R^d)$ satisfying
\begin{align*}
\mM(u_n)&=c_1,\\
\mH(u_n)&=\gamma_{c_1}+o_n(1).
\end{align*}
Let $\ld_*:=\sqrt{\frac{c_1}{c_2}}<1$. Then $\mM(U_{\ld_*}u_n)=c_2$ and we conclude that
\begin{align*}
\gamma_{c_1}=\mH(u_n)+o_n(1)\geq\mH(U_{\ld_*}u_n)+o_n(1)\geq \gamma_{c_2}+o_n(1).
\end{align*}
Sending $n\to\infty$ follows the monotonicity. To see that $\gamma_c$ is negative on $(\mM(Q),\infty)$, we define $S=tQ$ for some to be determined $t\in(1,\infty)$. Using Pohozaev we infer that
\begin{align*}
\|\nabla Q\|_2^2=\frac{d}{d+2}\|Q\|_\tas^\tas,
\end{align*}
which yields
\begin{align*}
\mH(T_\ld S)=-\frac{\ld^2}{\tas}(t^{\tas}-t^2)\|Q\|_\tas^\tas+\frac{\ld^\tbs}{\tbs}t^\tbs\|Q\|_\tbs^\tbs.
\end{align*}
By direct calculation we also see that
\begin{align*}
0<\ld<\bg(\frac{\tbs(t^\tas-t^2)\|Q\|_\tas^\tas}{\tas t^\tbs\|Q\|_\tbs^\tbs}\bg)^{\frac{d-2}{4}}\Rightarrow\mH(T_\ld S)<0.
\end{align*}
This shows that $\gamma_c<0$ on $(\mM(Q),\infty)$. Finally we show that $\gamma_c$ is bounded below. By H\"older inequality we obtain that
\begin{align*}
\|u\|_\tas^\tas\leq (\mM(u))^{\frac{2}{d}}\|u\|_\tbs^2.
\end{align*}
Then for $u\in H^1(\R^d)$ with $\mM(u)=c$ we have
\begin{align}\label{bounded below}
\mH(u)\geq -\frac{c^{\frac{2}{d}}}{\tas}\|u\|_\tbs^2+\frac{1}{\tbs}\|u\|_\tbs^\tbs.
\end{align}
But the function $t\mapsto -\frac{c^{\frac{2}{d}}}{\tas}t^2+\frac{1}{\tbs}t^\tbs$ is bounded below on $[0,\infty)$. This completes the proof.
\end{proof}

\subsection{Existence of minimizers of $\gamma_c$ for $c> \mM(Q)$}
\begin{lemma}\label{existence of ground state lemma}
For each $c>\mM(Q)$, the variational problem $\gamma_c$ has a minimizer which is positive and radially symmetric.
\end{lemma}

\begin{proof}
Let $(u_n)_n\subset H^1(\R^d)$ be a minimizing sequence, i.e.
\begin{align*}
\mM(u_n)&=c,\\
\mH(u_n)&=\gamma_c+o_n(1).
\end{align*}
Since $\mH$ is stable under the Steiner symmetrization, we may further assume that $u_n$ is radially symmetric. Using \eqref{bounded below} we infer that
\begin{align*}
\gamma_c+o_n(1)\geq -\frac{c^{\frac{2}{d}}}{\tas}\|u_n\|_\tbs^2+\frac{1}{\tbs}\|u_n\|_\tbs^\tbs,
\end{align*}
thus $(\|u_n\|_\tbs)_n$ is a bounded sequence. Hence
\begin{align*}
\frac{1}{2}\|\nabla u_n\|^2_2\leq\gamma_c+o_n(1)+\frac{c^{\frac{2}{d}}}{\tas}\|u_n\|_\tbs^2\lesssim 1,
\end{align*}
and therefore $(u_n)_n$ is a bounded sequence in $H^1(\R^d)$. Up to a subsequence $(u_n)_n$ converges to some radially symmetric $u\in H^1(\R^d)$ weakly in $H^1(\R^d)$ and $\mM(u)\leq c$. By weak lower semicontinuity of norms and the Strauss compact embedding for radial functions we know that
\begin{align*}
\mH(u)\leq \gamma_c<0,
\end{align*}
and therefore $u\neq 0$. Suppose that $\mM(u)<c$. Then $\mM(U_\ld u)=\ld^{-2}\mM(u)<c$ for $\ld$ in a neighborhood of $1$ and
\begin{align*}
\mH(U_\ld u)&=\mH(u)+\frac{1}{\tas}(1-\ld^{-\frac{4}{d}})\|u\|_\tas^\tas=\gamma_c+\frac{1}{\tas}(1-\ld^{-\frac{4}{d}})\|u\|_\tas^\tas<\gamma_c
\end{align*}
for $\ld<1$ sufficiently close to $1$. This contradicts the monotonicity of $c\mapsto \gamma_c$, thus $\mM(u)=c$. By Lagrange multiplier theorem we know that any minimizer of $\gamma_c$ is automatically a solution of \eqref{standing wave eq fd} and thus the positivity of $u$ follows from the strong maximum principle. The proof is then complete.
\end{proof}

\begin{proof}[Proof of Proposition \ref{proposition for ground state}]
This follows immediately from Lemma \ref{pohozaev} to Lemma \ref{existence of ground state lemma}.
\end{proof}

\subsection{Scattering for the focusing-defocusing \eqref{NLS double crit}}
\begin{lemma}\label{invariant lemma fd}
Let $u$ be a solution of \eqref{NLSfd} with $u(0)\in \mA$. Then $u(t)\in {\mA}$ for all $t\in\R$. Assume also $\mM(u)=(1-\delta)^{\frac{d}{2}}\mM(Q)$, then
\begin{align}\label{lower bound kt fd}
\inf_{t\in I_{\max}}\mK(u(t))\geq2\mH(u(0)).
\end{align}
\end{lemma}

\begin{proof}
That $u(t)\in\mA$ for all $t\in\R$ follows immediately from the conservation of mass. Moreover, \eqref{lower bound kt fd} follows from
\begin{align*}
\mK(u(t))=2\mH(u(t))+\frac{2}{d}\|u\|_\tbs^\tbs\geq 2\mH(u(0)),
\end{align*}
where we also used the conservation of energy.
\end{proof}

We now define the MEI-functional for \eqref{NLSfd}. Let $ \Omega:=(-\infty,\mM(Q))\times\R$ and let the MEI-functional $\mD$ be given by \eqref{MEI functional}. One has the following analogue of Lemma \ref{killip visan curve}.

\begin{lemma}\label{killip visan curve fd}
Assume $u\in H^1(\R^d)$. Then
\begin{itemize}
\item[(i)]$\mD(u)=0$ if and only if $u=0$.
\item[(ii)] $0<\mD(u)<\infty$ if and only if $u\in\mA$.
\item[(iii)] $\mD$ leaves $\mA$ invariant under the NLS flow.
\item[(iv)] Let $u_1,u_2\in\mA$ with $\mM(u_1)\leq \mM(u_2)$ and $\mH(u_1)\leq \mH(u_2)$, then $\mD(u_1)\leq \mD(u_2)$. If additionally either $\mM(u_1)<\mM(u_2)$ or $\mH(u_1)<\mH(u_2)$, then $\mD(u_1)<\mD(u_2)$.
\item[(v)] Let $\mD_0\in(0,\infty)$. Then
\begin{align}
\|\nabla u\|^2_{2}&\sim_{\mD_0}\mH(u),\label{fd 1}\\
\|u\|^2_{H^1}&\sim_{\mD_0}\mH(u)+\mM(u)\sim_{\mD_0}\mD(u)\label{fd 2}
\end{align}
uniformly for all $u\in \mA$ with $\mD(u)\leq \mD_0$.
\end{itemize}
\end{lemma}

\begin{remark}
Due to the positivity of the defocusing energy-critical potential we do not need to impose the additional condition $\mK(u)\geq 0$.
\end{remark}

\begin{proof}
(i) to (iv) are trivial. We still need to verify (v). Let $u\in\mA$ with $\mD(u)\leq\mD_0$. It is readily to see that
\begin{align}\label{fd 0}
\mD_0\geq \mD(u)=\mH(u)+\frac{\mH(u)+\mM(u)}{\mM(Q)-\mM(u)}\geq \frac{\mM(u)}{\mM(Q)-\mM(u)},
\end{align}
which implies $\mM(u)\leq(1+\mD_0)^{-1}\mD_0\mM(Q)$. Hence
\begin{align}\label{fd 3}
\mH(u)\geq\frac{1}{2}\bg(1-\bg(\frac{\mM(u)}{\mM(Q)}\bg)^{\frac{2}{d}}\bg)\|\nabla u\|_2^2+\frac{1}{\tbs}\|u\|_\tbs^\tbs
\geq \frac{1}{2}\bg(1-\big((1+\mD_0)^{-1}\mD_0\big)^{\frac{2}{d}}\bg)\|\nabla u\|_2^2.
\end{align}
Similarly, we obtain
\begin{align*}
\mD_0\geq \mH(u)\geq\frac{1}{2}\bg(1-\big((1+\mD_0)^{-1}\mD_0\big)^{\frac{2}{d}}\bg)\|\nabla u\|_2^2,
\end{align*}
which implies
\begin{align*}
\|\nabla u\|_2^2\leq \frac{2\mD_0}{1-\big((1+\mD_0)^{-1}\mD_0\big)^{\frac{2}{d}}}.
\end{align*}
Using Sobolev inequality and \eqref{fd 3} we obtain that
\begin{align}\label{fd 4}
\mH(u)&\leq\frac{1}{2}\|\nabla u\|_2^2+\frac{1}{\tbs}\|u\|_\tbs^\tbs
\leq\frac{1}{2}\|\nabla u\|_2^2+\frac{1}{\tbs}\|u\|_\tbs^\tbs\nonumber\\
&\leq\frac{1}{2}\|\nabla u\|_2^2+\frac{\mS^{-\frac{d}{d-2}}}{\tbs}\bg(\frac{2\mD_0}{1-\big((1+\mD_0)^{-1}\mD_0\big)^{\frac{2}{d}}}\bg)^{\frac{2}{d-2}}\|\nabla u\|_2^2.
\end{align}
\eqref{fd 1} and the first equivalence of \eqref{fd 2} now follow from \eqref{fd 3} and \eqref{fd 4}. From \eqref{fd 0} it also follows $\mH(u)+\mM(u)\lesssim_{\mD_0} \mD(u)$. That $\mD(u)\lesssim_{\mD_0}\mH(u)+\mM(u)$ follows immediately from
\begin{align*}
\mD(u)&= \mH(u)+\frac{\mH(u)+\mM(u)}{\mM(Q)-\mM(u)}\leq \mH(u)+\frac{\mH(u)+\mM(u)}{(1-(1+\mD_0)^{-1}\mD_0)\mM(Q)}\nonumber\\
&=\mH(u)+\frac{(1+\mD_0)(\mH(u)+\mM(u))}{\mM(Q)}.
\end{align*}
\end{proof}

\begin{proof}[Proof of Theorem \ref{main theorem} for the focusing-defocusing regime]
The proof is almost identical to the one for the focusing-focusing regime, one only needs to replace the results from \cite{KenigMerle2006,KillipVisan2010focusing,Dodson4dfocusing} applied in Lemma \ref{persistance h1} by the ones from \cite{defocusing3d,defocusing4d,defocusing5dandhigher}, the arguments from Lemma \ref{killip visan curve} by the ones from Lemma \ref{killip visan curve fd} and \eqref{small extinction ff} by \eqref{lower bound kt fd}. We therefore omit the details here.
\end{proof}

\addcontentsline{toc}{section}{Acknowledgments}
\subsubsection*{Acknowledgments}
\selectlanguage{German}
The author acknowledges the funding by Deutsche Forschungsgemeinschaft (DFG) through the Priority Programme SPP-1886.

\selectlanguage{English}

\appendix
\section{Endpoint values of the curve $c\mapsto m_c$ for the focusing-focusing \eqref{NLS double crit}}\label{section endpoint}
\begin{proposition}
Let $\mu_1=\mu_2=1$ and $m_c$ be defined through \eqref{variational problem for mc}. Let
\begin{align*}
m_0:=\lim_{c\,\downarrow \,0}m_c,\quad
m_{Q}:=\lim_{c\,\uparrow\, \mM(Q)}m_c.
\end{align*}
Then $m_0=\mH^*(W)$ and $m_{Q}=0$.
\end{proposition}

\begin{proof}
By Theorem \ref{soave} we already know that $m_0\leq \mH^*(W)$. For $c\in(0,\mM(Q))$, let $P_c$ be an optimizer of the variational problem $m_c$, whose existence is guaranteed by Theorem \ref{soave}. We first show $m_0=\mH^*(W)$ and let $c\downarrow 1$. Then by $\mK(P_c)=0$ and \eqref{GN-L2} we obtain that
\begin{align}\label{lower bd mc}
m_c&=\mH(P_c)=\mH(P_c)-\frac{1}{\tbs}\mK(P_c)\nonumber\\
&=\frac{1}{d}\bg(\|\nabla P_c\|_2^2-\frac{d}{d+2}\|P_c\|_\tas^\tas\bg)\nonumber\\
&\geq\frac{1}{d}\bg(1-\bg(\frac{\mM(P_c)}{{\mM(Q)}}\bg)^{\frac{2}{d}}\bg)\|\nabla P_c\|_2^2\nonumber\\
&=\frac{1}{d}\bg(1-\bg(\frac{o_c(1)}{{\mM(Q)}}\bg)^{\frac{2}{d}}\bg)\|\nabla P_c\|_2^2.
\end{align}
Hence $(P_c)_{c\,\downarrow \,0}$ is bounded in $H^1(\R^d)$. On the other hand, using $\mK(P_c)=0$ and Sobolev inequality we infer that
\begin{align*}
\frac{1}{d}\bg(1-\bg(\frac{o_c(1)}{{\mM(Q)}}\bg)^{\frac{2}{d}}\bg)\|\nabla P_c\|_2^2
\leq \frac{1}{d}\bg(\|\nabla P_c\|_2^2-\frac{d}{d+2}\|P_c\|_\tas^\tas\bg)=\frac{1}{d}\|P_c\|_\tbs^\tbs\leq
\frac{\mS^{\frac{d}{2-d}}}{d}\|\nabla P_c\|_2^\tbs,
\end{align*}
which implies that (up to a subsequence) $l:=\lim_{c\,\downarrow \, 0}\|\nabla P_c\|_2^2>0$. But then by the Gagliardo-Nirenberg inequality and $\mK(P_c)=0$ we obtain that
\begin{align*}
\|\nabla P_c\|_2^\tbs\mS^{\frac{d}{2-d}}\geq\|P_c\|_\tbs^\tbs=\|\nabla P_c\|_2^2-\frac{d}{d+2}\|P_c\|_\tas^\tas
\geq \bg(1-\bg(\frac{o_n(1)}{{\mM(Q)}}\bg)^{\frac{2}{d}}\bg)\|\nabla P_c\|_2^2 \to l.
\end{align*}
Therefore $l^{\tbs}\mS^{\frac{d}{2-d}}\geq l$. Since $l\neq 0$, we infer that $l\geq \mS^{\frac{d}{2}}$. But then \eqref{lower bd mc} implies $m_c\geq \frac{\mS^{\frac{d}{2}}}{d}=\mH^*(W)$, which completes the proof.

Next we show $m_Q=0$. Let $(u_n)_n$ be a minimizing sequence for \eqref{def gn l2 crit}. By rescaling we may assume that $\mM(u_n)=\delta \mM(Q)$ and $\|u_n\|_\tas=1$ for a fixed $\delta\in(0,1)$ which will be sended to $1$ later. Then combining with \eqref{GN-L1} we obtain that $\|\nabla u_n\|_2^2=\frac{d}{d+2}\delta^{-\frac{2}{d}}+o_n(1)$. We then conclude that
\begin{align*}
\mK(T_\ld u_n)=\frac{d\ld^2}{d+2}\bg(\delta^{-\frac{2}{d}}-1+o_n(1)\bg)-\ld^\tbs\|u_n\|_\tbs^\tbs.
\end{align*}
By setting
\begin{align*}
\ld_{n,\delta}=\bg(\frac{d}{(d+2)\|u_n\|_\tbs^\tbs}\bg(\delta^{-\frac{2}{d}}-1+o_n(1)\bg)\bg)^{\frac{d-2}{4}}
\end{align*}
we see that $\mK(T_{\ld_{n,\delta}}u_n)=0$. By H\"older we deduce that
\begin{align*}
\|u_n\|_\tbs^\tbs\geq \mM(u_n)^{-\frac{2}{d-2}}\|u_n\|_\tas^{\frac{2(d+2)}{d-2}}=(\delta\mM(Q))^{-\frac{2}{d-2}}.
\end{align*}
We now choose $N=N(\delta)\in\N$ such that $|o_n(1)|\leq \delta^{-\frac{2}{d}}-1$ for all $n>N$. Summing up and using the definition of $m_c$ we finally conclude that
\begin{align*}
m_{\delta M(Q)}&\leq \sup_{n>N}\mH(T_{\ld_{n,\delta}}u_n)=\sup_{n>N}\bg(\mH(T_{\ld_{n,\delta}}u_n)-\frac{1}{2}\mK(T_{\ld_{n,\delta}}u_n)\bg)\nonumber\\
&=\sup_{n>N}\frac{1}{\tbs}\|T_{\ld_{n,\delta}}u_n\|_\tbs^\tbs=\sup_{n>N}
\frac{\ld_{n,\delta}^\tbs}{\tbs}\|u_n\|_\tbs^\tbs\nonumber\\
&\leq \frac{2^{\frac{d}{2}}}{\tbs}\bg(\frac{d}{d+2}\bg)^{\frac{d}{2}}(\delta^{-\frac{2}{d}}-1)\delta\mM(Q)\to 0
\end{align*}
as $\delta\to 1$. This proves $m_Q=0$.
\end{proof}

\addcontentsline{toc}{section}{References}
\bibliography{bib_double_crit}

\begin{thebibliography}{10}

\bibitem{Akahori2013}
{\sc Akahori, T., Ibrahim, S., Kikuchi, H., and Nawa, H.}
\newblock Existence of a ground state and scattering for a nonlinear
  {S}chr\"{o}dinger equation with critical growth.
\newblock {\em Selecta Math. (N.S.) 19}, 2 (2013), 545--609.

\bibitem{phy_double_crit_1}
{\sc Barashenkov, I.~V., Gocheva, A.~D., Makhankov, V.~G., and Puzynin, I.~V.}
\newblock Stability of the soliton-like ``bubbles''.
\newblock {\em Phys. D 34}, 1-2 (1989), 240--254.

\bibitem{Bellazzini2013}
{\sc Bellazzini, J., Jeanjean, L., and Luo, T.}
\newblock Existence and instability of standing waves with prescribed norm for
  a class of {S}chr\"odinger-{P}oisson equations.
\newblock {\em Proc. Lond. Math. Soc. (3) 107}, 2 (2013), 303--339.

\bibitem{lions1}
{\sc Berestycki, H., and Lions, P.-L.}
\newblock Nonlinear scalar field equations. {I}. {E}xistence of a ground state.
\newblock {\em Arch. Rational Mech. Anal. 82}, 4 (1983), 313--345.

\bibitem{Berestycki1983}
{\sc Berestycki, H., and Lions, P.-L.}
\newblock Nonlinear scalar field equations. {II}. {E}xistence of infinitely
  many solutions.
\newblock {\em Arch. Rational Mech. Anal. 82}, 4 (1983), 347--375.

\bibitem{BourgainRadial}
{\sc Bourgain, J.}
\newblock Global wellposedness of defocusing critical nonlinear
  {S}chr\"{o}dinger equation in the radial case.
\newblock {\em J. Amer. Math. Soc. 12}, 1 (1999), 145--171.

\bibitem{Buslaev2001}
{\sc Buslaev, V.~S., and Grikurov, V.~E.}
\newblock Simulation of instability of bright solitons for {NLS} with
  saturating nonlinearity.
\newblock {\em Math. Comput. Simulation 56}, 6 (2001), 539--546.

\bibitem{carles2020soliton}
{\sc Carles, R., Klein, C., and Sparber, C.}
\newblock On soliton (in-)stability in multi-dimensional cubic-quintic
  nonlinear schr\"odinger equations, 2020, 2012.11637.

\bibitem{Carles_Sparber_2021}
{\sc Carles, R., and Sparber, C.}
\newblock Orbital stability vs. scattering in the cubic-quintic
  {S}chr\"{o}dinger equation.
\newblock {\em Rev. Math. Phys. 33}, 3 (2021), 2150004, 27.

\bibitem{Cazenave2003}
{\sc Cazenave, T.}
\newblock {\em Semilinear {S}chr\"odinger equations}, vol.~10 of {\em Courant
  Lecture Notes in Mathematics}.
\newblock New York University, Courant Institute of Mathematical Sciences, New
  York; American Mathematical Society, Providence, RI, 2003.

\bibitem{CazenaveWeissler1}
{\sc Cazenave, T., and Weissler, F.~B.}
\newblock Some remarks on the nonlinear {S}chr\"{o}dinger equation in the
  critical case.
\newblock In {\em Nonlinear semigroups, partial differential equations and
  attractors ({W}ashington, {DC}, 1987)}, vol.~1394 of {\em Lecture Notes in
  Math.} Springer, Berlin, 1989, pp.~18--29.

\bibitem{Cazenave1990}
{\sc Cazenave, T., and Weissler, F.~B.}
\newblock The {C}auchy problem for the critical nonlinear {S}chr\"{o}dinger
  equation in {$H^s$}.
\newblock {\em Nonlinear Anal. 14}, 10 (1990), 807--836.

\bibitem{Cheng2020}
{\sc Cheng, X.}
\newblock Scattering for the mass super-critical perturbations of the mass
  critical nonlinear {S}chr\"{o}dinger equations.
\newblock {\em Illinois J. Math. 64}, 1 (2020), 21--48.

\bibitem{torus0}
{\sc Cheng, X., Guo, Z., Yang, K., and Zhao, L.}
\newblock On scattering for the cubic defocusing nonlinear {S}chr\"{o}dinger
  equation on the waveguide {$\Bbb R^2 \times \Bbb T$}.
\newblock {\em Rev. Mat. Iberoam. 36}, 4 (2020), 985--1011.

\bibitem{MiaoDoubleCrit}
{\sc Cheng, X., Miao, C., and Zhao, L.}
\newblock Global well-posedness and scattering for nonlinear {S}chr\"{o}dinger
  equations with combined nonlinearities in the radial case.
\newblock {\em J. Differential Equations 261}, 6 (2016), 2881--2934.

\bibitem{defocusing3d}
{\sc Colliander, J., Keel, M., Staffilani, G., Takaoka, H., and Tao, T.}
\newblock Global well-posedness and scattering for the energy-critical
  nonlinear {S}chr\"{o}dinger equation in {$\Bbb R^3$}.
\newblock {\em Ann. of Math. (2) 167}, 3 (2008), 767--865.

\bibitem{dodson3d}
{\sc Dodson, B.}
\newblock Global well-posedness and scattering for the defocusing,
  {$L^{2}$}-critical nonlinear {S}chr\"{o}dinger equation when {$d\geq3$}.
\newblock {\em J. Amer. Math. Soc. 25}, 2 (2012), 429--463.

\bibitem{Dodson4dmassfocusing}
{\sc Dodson, B.}
\newblock Global well-posedness and scattering for the mass critical nonlinear
  {S}chr\"{o}dinger equation with mass below the mass of the ground state.
\newblock {\em Adv. Math. 285\/} (2015), 1589--1618.

\bibitem{dodson1d}
{\sc Dodson, B.}
\newblock Global well-posedness and scattering for the defocusing, {$L^2$}
  critical, nonlinear {S}chr\"{o}dinger equation when {$d=1$}.
\newblock {\em Amer. J. Math. 138}, 2 (2016), 531--569.

\bibitem{dodson2d}
{\sc Dodson, B.}
\newblock Global well-posedness and scattering for the defocusing,
  {$L^2$}-critical, nonlinear {S}chr\"{o}dinger equation when {$d=2$}.
\newblock {\em Duke Math. J. 165}, 18 (2016), 3435--3516.

\bibitem{Dodson4dfocusing}
{\sc Dodson, B.}
\newblock Global well-posedness and scattering for the focusing, cubic
  {S}chr\"{o}dinger equation in dimension {$d=4$}.
\newblock {\em Ann. Sci. \'{E}c. Norm. Sup\'{e}r. (4) 52}, 1 (2019), 139--180.

\bibitem{non_radial}
{\sc Duyckaerts, T., Holmer, J., and Roudenko, S.}
\newblock Scattering for the non-radial 3{D} cubic nonlinear {S}chr\"{o}dinger
  equation.
\newblock {\em Math. Res. Lett. 15}, 6 (2008), 1233--1250.

\bibitem{Glassey1977}
{\sc Glassey, R.~T.}
\newblock On the blowing up of solutions to the {C}auchy problem for nonlinear
  {S}chr\"{o}dinger equations.
\newblock {\em J. Math. Phys. 18}, 9 (1977), 1794--1797.

\bibitem{Ibrahim2011}
{\sc Ibrahim, S., Masmoudi, N., and Nakanishi, K.}
\newblock Scattering threshold for the focusing nonlinear {K}lein-{G}ordon
  equation.
\newblock {\em Anal. PDE 4}, 3 (2011), 405--460.

\bibitem{EndpointStrichartz}
{\sc Keel, M., and Tao, T.}
\newblock Endpoint {S}trichartz estimates.
\newblock {\em Amer. J. Math. 120}, 5 (1998), 955--980.

\bibitem{KenigMerle2006}
{\sc Kenig, C.~E., and Merle, F.}
\newblock Global well-posedness, scattering and blow-up for the
  energy-critical, focusing, non-linear {S}chr\"{o}dinger equation in the
  radial case.
\newblock {\em Invent. Math. 166}, 3 (2006), 645--675.

\bibitem{killip_visan_soliton}
{\sc Killip, R., Oh, T., Pocovnicu, O., and Vi\c{s}an, M.}
\newblock Solitons and scattering for the cubic-quintic nonlinear
  {S}chr\"{o}dinger equation on {$\Bbb{R}^3$}.
\newblock {\em Arch. Ration. Mech. Anal. 225}, 1 (2017), 469--548.

\bibitem{KillipVisanNotes}
{\sc Killip, R., and Vi\c{s}an, M.}
\newblock Nonlinear {S}chr\"{o}dinger equations at critical regularity.
\newblock In {\em Evolution equations}, vol.~17 of {\em Clay Math. Proc.} Amer.
  Math. Soc., Providence, RI, 2013, pp.~325--437.

\bibitem{KillipVisan2010focusing}
{\sc Killip, R., and Visan, M.}
\newblock The focusing energy-critical nonlinear {S}chr\"{o}dinger equation in
  dimensions five and higher.
\newblock {\em Amer. J. Math. 132}, 2 (2010), 361--424.

\bibitem{Kwong_uniqueness}
{\sc Kwong, M.~K.}
\newblock Uniqueness of positive solutions of {$\Delta u-u+u^p=0$} in {${\bf
  R}^n$}.
\newblock {\em Arch. Rational Mech. Anal. 105}, 3 (1989), 243--266.

\bibitem{LeMesurier1988}
{\sc LeMesurier, B.~J., Papanicolaou, G., Sulem, C., and Sulem, P.-L.}
\newblock Focusing and multi-focusing solutions of the nonlinear
  {S}chr\"{o}dinger equation.
\newblock {\em Phys. D 31}, 1 (1988), 78--102.

\bibitem{luo2021scattering}
{\sc Luo, Y.}
\newblock Scattering threshold for radial defocusing-focusing mass-energy
  double critical nonlinear {S}chr\"odinger equation in $d\geq 5$, 2021,
  2106.06993.

\bibitem{luo2021sharp}
{\sc Luo, Y.}
\newblock Sharp scattering threshold for the cubic-quintic {NLS} in the
  focusing-focusing regime, 2021, 2105.15091.

\bibitem{Murphy2021CPDE}
{\sc Murphy, J.}
\newblock Threshold scattering for the 2d radial cubic-quintic {NLS}.
\newblock {\em Comm. Partial Differential Equations\/} (2021), 1--22.

\bibitem{defocusing4d}
{\sc Ryckman, E., and Visan, M.}
\newblock Global well-posedness and scattering for the defocusing
  energy-critical nonlinear {S}chr\"{o}dinger equation in {$\Bbb R^{1+4}$}.
\newblock {\em Amer. J. Math. 129}, 1 (2007), 1--60.

\bibitem{SoaveCritical}
{\sc Soave, N.}
\newblock Normalized ground states for the {NLS} equation with combined
  nonlinearities: the {S}obolev critical case.
\newblock {\em J. Funct. Anal. 279}, 6 (2020), 108610, 43.

\bibitem{TaoVisanZhang}
{\sc Tao, T., Visan, M., and Zhang, X.}
\newblock The nonlinear {S}chr\"{o}dinger equation with combined power-type
  nonlinearities.
\newblock {\em Comm. Partial Differential Equations 32}, 7-9 (2007),
  1281--1343.

\bibitem{defocusing5dandhigher}
{\sc Visan, M.}
\newblock The defocusing energy-critical nonlinear {S}chr\"{o}dinger equation
  in higher dimensions.
\newblock {\em Duke Math. J. 138}, 2 (2007), 281--374.

\bibitem{wei2021normalized}
{\sc Wei, J., and Wu, Y.}
\newblock Normalized solutions for {S}chr\"odinger equations with critical
  {S}obolev exponent and mixed nonlinearities, 2021, 2102.04030.

\bibitem{weinstein}
{\sc Weinstein, M.~I.}
\newblock Nonlinear {S}chr\"{o}dinger equations and sharp interpolation
  estimates.
\newblock {\em Comm. Math. Phys. 87}, 4 (1982/83), 567--576.

\bibitem{Zhang2006}
{\sc Zhang, X.}
\newblock On the {C}auchy problem of 3-{D} energy-critical {S}chr\"{o}dinger
  equations with subcritical perturbations.
\newblock {\em J. Differential Equations 230}, 2 (2006), 422--445.

\end{thebibliography}
\bibliographystyle{hacm}
\end{document}